\documentclass{article}
\usepackage{amsmath}
\usepackage{amssymb}
\usepackage{amsthm}
\usepackage{enumerate}
\usepackage{mathdots}
\usepackage{cite}
\usepackage{caption}
\usepackage{euscript}
\usepackage[arrow,curve,matrix,arc,2cell]{xy}
\usepackage[utf8]{inputenc}
\usepackage[unicode]{hyperref}
\usepackage{thmtools}
\UseAllTwocells
\DeclareFontFamily{U}{rsfs}{} \DeclareFontShape{U}{rsfs}{n}{it}{<->rsfs10}{}
\DeclareSymbolFont{mscr}{U}{rsfs}{n}{it}
\DeclareSymbolFontAlphabet{\scr}{mscr}
\def\mathscr{\scr}
\declaretheorem[name=Theorem,numberwithin=section]{thm}
\declaretheorem[name=Proposition,numberlike=thm]{prop}
\declaretheorem[name=Lemma,numberlike=thm]{lem}
\declaretheorem[name=Corollary,numberlike=thm]{cor}

\declaretheorem[name=Conjecture,numberlike=thm]{conj}
\declaretheorem[name=Definition,style=definition,qed=$\blacktriangle$,numberlike=thm]{dfn}
\declaretheorem[name=Example,style=definition,qed=$\blacktriangle$,numberlike=thm]{ex}
\declaretheorem[name=Remark,style=definition,qed=$\blacktriangle$,numberlike=thm]{rem}
\declaretheorem[name=Assumption,style=definition,numberlike=thm]{ass}
\numberwithin{equation}{section}
\numberwithin{table}{section}
\begin{document}
%%%%%%%%%%%%%%%%%%%%%%%%%%%%%%%%%%%%%%%%%%%%%%%%%%%%%%%%%%%%%%%%%%%%%%%%
%%%%%%%%%%%%%%%%%%%%%%%%%%     Macros      %%%%%%%%%%%%%%%%%%%%%%%%%%%%%
%%%%%%%%%%%%%%%%%%%%%%%%%%%%%%%%%%%%%%%%%%%%%%%%%%%%%%%%%%%%%%%%%%%%%%%%
\def\e#1\e{\begin{equation}#1\end{equation}}
\def\ea#1\ea{\begin{align}#1\end{align}}
\def\eq#1{{\rm(\ref{#1})}}
\def\dim{\mathop{\rm dim}\nolimits}
\def\Fix{\mathop{\rm Fix}}
\def\Ker{\mathop{\rm Ker}}
\def\Coker{\mathop{\rm Coker}}
\def\id{\mathop{\rm id}\nolimits}
\def\SO{\mathop{\rm SO}\nolimits}
\def\ad{\mathop{\rm ad}\nolimits}
\def\Hom{\mathop{\rm Hom}\nolimits}
\def\End{\mathop{\rm End}\nolimits}
\def\Hol{\mathop{\rm Hol}\nolimits}
\def\Diff{\mathop{\rm Diff}\nolimits}
\def\Re{\mathop{\rm Re}}
\def\Im{\mathop{\rm Im}}
\def\SU{\mathop{\rm SU}}
\def\Sp{\mathop{\rm Sp}}
\def\Spin{\mathop{\rm Spin}}
\def\Gt{\mathop{\rm G_2}}
\def\GL{\mathop{\rm GL}}
\def\ind{\mathop{\rm ind}}
\def\area{\mathop{\rm area}}
\def\U{{\rm U}}
\def\vol{\mathop{\rm vol}\nolimits}
\def\inc{{\rm inc}}
\def\bs{\boldsymbol}
\def\ge{\geqslant\nobreak}
\def\le{\leqslant\nobreak}
\def\O{{\mathbin{\cal O}}}
\def\cA{{\mathbin{\cal A}}}
\def\cB{{\mathbin{\cal B}}}
\def\cC{{\mathbin{\cal C}}}
\def\cD{{\mathbin{\scr D}}}
\def\cE{{\mathbin{\cal E}}}
\def\cF{{\mathbin{\cal F}}}
\def\cG{{\mathbin{\cal G}}}
\def\cH{{\mathbin{\cal H}}}
\def\cI{{\mathbin{\cal I}}}
\def\cJ{{\mathbin{\cal J}}}
\def\cK{{\mathbin{\cal K}}}
\def\cL{{\mathbin{\cal L}}}
\def\cM{{\mathbin{\cal M}}}
\def\cN{{\mathbin{\cal N}}}
\def\cO{{\mathbin{\cal O}}}
\def\cP{{\mathbin{\cal P}}}
\def\cS{{\mathbin{\cal S}}}
\def\cT{{\mathbin{\cal T}}}
\def\cU{{\mathbin{\cal U}}}
\def\cQ{{\mathbin{\cal Q}}}
\def\cW{{\mathbin{\cal W}}}
\def\C{{\mathbin{\mathbb C}}}
\def\CP{{\mathbin{\mathbb{CP}}}}
\def\H{{\mathbin{\mathbb H}}}
\def\N{{\mathbin{\mathbb N}}}
\def\P{{\mathbin{\mathbb P}}}
\def\Q{{\mathbin{\mathbb Q}}}
\def\R{{\mathbin{\mathbb R}}}
\def\CP{{\mathbin{\mathbb{CP}}}}
\def\RP{{\mathbin{\mathbb{RP}}}}
\def\Z{{\mathbin{\mathbb Z}}}
\def\sF{{\mathbin{\mathscr F}}}
\def\al{\alpha}
\def\be{\beta}
\def\ga{\gamma}
\def\de{\delta}
\def\io{\iota}
\def\ep{\epsilon}
\def\la{\lambda}
\def\ka{\kappa}
\def\th{\theta}
\def\ze{\zeta}
\def\up{\upsilon}
\def\vp{\varphi}
\def\si{\sigma}
\def\om{\omega}
\def\De{\Delta}
\def\La{\Lambda}
\def\Si{\Sigma}
\def\Th{\Theta}
\def\Om{\Omega}
\def\Ga{\Gamma}
\def\Up{\Upsilon}
\def\pd{\partial}
\def\ts{\textstyle}
\def\st{\scriptstyle}
\def\sst{\scriptscriptstyle}
\def\w{\wedge}
\def\sm{\setminus}
\def\bu{\bullet}
\def\op{\oplus}
\def\ot{\otimes}
\def\ov{\overline}
\def\ul{\underline}
\def\bigop{\bigoplus}
\def\bigot{\bigotimes}
\def\iy{\infty}
\def\es{\emptyset}
\def\ra{\rightarrow}
\def\Ra{\Rightarrow}
\def\Longra{\Longrightarrow}
\def\ab{\allowbreak}
\def\longra{\longrightarrow}
\def\hookra{\hookrightarrow}
\def\dashra{\dashrightarrow}
\def\t{\times}
\def\ci{\circ}
\def\ti{\tilde}
\def\d{{\rm d}}
\def\ha{{\ts\frac{1}{2}}}
\def\md#1{\vert #1 \vert}
\def\bmd#1{\big\vert #1 \big\vert}
\def\ms#1{\vert #1 \vert^2}
\def\nm#1{\Vert #1 \Vert}
\def\bnm#1{\big\Vert #1 \big\Vert}
\def\an#1{\langle #1 \rangle}
\def\ban#1{\big\langle #1 \big\rangle}
\def\dx{\d_{\scriptscriptstyle X}}
\def\lapx{\Delta_{\scriptscriptstyle X}}
\def\dsx{\d^*_{\scriptscriptstyle X}}
\def\stx{\ast_{\scriptscriptstyle X}}
\def\volx{\vol_{\scriptscriptstyle X}}
\def\gx{g_{\scriptscriptstyle X}}
\def\Dx{\mathcal D_{\scriptscriptstyle X}}
\def\Dsx{\mathcal D^*_{\scriptscriptstyle X}}
%%%%%%%%%%%%%%%%%%%%%%%%%%%%%%%%%%%%%%%%%%%%%%%%%%%%%%%%%%%%%%%%%%%%%%%%
%%%%%%%%%%%%%%%%%%%%%% Text of paper %%%%%%%%%%%%%%%%%%%%%%%%%%%%%%%%%%%
%%%%%%%%%%%%%%%%%%%%%%%%%%%%%%%%%%%%%%%%%%%%%%%%%%%%%%%%%%%%%%%%%%%%%%%%
\title{A new construction of compact torsion-free \\ $\Gt$-manifolds by gluing families of \\ Eguchi--Hanson spaces}
\author{Dominic Joyce \\ {\it Mathematical Institute, University of Oxford} \\ \tt{joyce@maths.ox.ac.uk} \and
Spiro Karigiannis \\ {\it Department of Pure Mathematics, University of Waterloo} \\ \tt{karigiannis@uwaterloo.ca} }
\date{}
\maketitle

\begin{abstract}
We give a new construction of compact Riemannian 7-manifolds with holonomy $\Gt$. Let $M$ be a torsion-free $\Gt$-manifold (which can have holonomy a proper subgroup of $\Gt$) such that $M$ admits an involution $\iota$ preserving the $\Gt$-structure. Then $M/\an{\iota}$ is a $\Gt$-orbifold, with singular set $L$ an  associative submanifold of $M$, where the singularities are locally of the form $\R^3 \times (\R^4 / \{\pm 1\})$. We resolve this orbifold by gluing in a family of Eguchi--Hanson spaces, parametrized by a nonvanishing closed and coclosed $1$-form $\lambda$ on~$L$.

Much of the analytic difficulty lies in constructing appropriate closed $\Gt$-structures with sufficiently small torsion to be able to apply the general existence theorem of the first author. In particular, the construction involves solving a family of elliptic equations on the noncompact Eguchi--Hanson space, parametrized by the singular set $L$. We also present two generalizations of the main theorem, and we discuss several methods of producing examples from this construction.
\end{abstract}

\setcounter{tocdepth}{2}
\tableofcontents

\section{Introduction}
\label{nc1}

Compact $\Gt$ holonomy manifolds have much in common with the better understood compact Calabi--Yau manifolds. In particular, they are both Ricci-flat compact Riemannian manifolds, and they both admit parallel spinors, making them candidates for \emph{supersymmetric compactification spaces} in physics \cite{AcGu}. Both Calabi--Yau metrics, which are of holonomy $\SU(n)$, and holonomy $\Gt$ metrics on compact manifolds, are in some sense \emph{transcendental} objects, in the sense that explicit formulas for such metrics are not expected to exist. Instead, one proves existence of such metrics by establishing that some nonlinear elliptic partial differential equation on the manifold has a smooth solution.

However, Calabi--Yau manifolds can be, at least partially, studied fruitfully using methods of algebraic geometry. Such tools are not available for general torsion-free $\Gt$-manifolds, as they are $7$-dimensional objects. Moreover, there is no analogue of a Calabi--Yau Theorem in $\Gt$ geometry. That is, we are far from knowing useful sufficient conditions for a compact smooth oriented spin $7$-manifold to admit $\Gt$ holonomy metrics. Because of this, there are far fewer known examples of compact $\Gt$ holonomy manifolds.

The first construction of compact $\Gt$ holonomy manifolds was due to the first author, originally presented in \cite{Joyc1, Joyc2} and extended in \cite{Joyc3}. The construction involves the resolution of orbifold singularities obtained by taking the quotient of the flat $7$-torus $T^7$ by a finite group action preserving the canonical flat $\Gt$-structure. One then very carefully defines a closed $\Gt$-structure $\vp$ on the resolved manifold, with sufficiently small torsion, and invokes a theorem \cite[Th.~11.6.1]{Joyc1} of the first author (this is Theorem \ref{nc2thm3} below) to conclude that there exists a torsion-free $\Gt$-structure $\tilde \vp$ in the same cohomology class as $\vp$. This generalizes the familiar Kummer-type construction of the Calabi--Yau metric on a $K3$ surface \cite{Topi}, and is an example of a geometric gluing construction. Several hundred distinct topological types, distinguishable by their second and third Betti numbers $b^2$ and $b^3$, were produced from this construction. 

Somewhat later, a second construction of compact $\Gt$ holonomy manifolds was given by Kovalev \cite{Kova}, based on an idea of Donaldson \cite[Method 3, page 303]{Joyc3}. Kovalev's construction involves the \emph{twisted connect sum} of two asymptotically cylindrical $\Gt$-manifolds built from asymptotically cylindrical Calabi--Yau manifolds. This construction is also a geometric gluing. The main difficulty of the construction is in finding suitable building blocks that satisfy the correct matching conditions to be able to carry out the gluing. More recently, Corti--Haskins--Nordstr\"om--Pacini \cite{CHNP1,CHNP2} generalized the Kovalev construction, producing tens of thousands of topologically distinct compact $\Gt$ holonomy manifolds. Moreover, further work by Crowley--Goette--Nordstr\"om in~\cite{CGN} then resulted in the first examples of inequivalent torsion-free $\Gt$-structures on the same underlying smooth manifold.

In the present paper we present a new construction of smooth compact $\Gt$ holonomy manifolds, based on an idea due to the first author \cite[Method 2, page 303]{Joyc3}. This construction is also essentially a geometric gluing, however it differs from the two earlier constructions in a significant way:

\begin{itemize}
\item In the new construction of the present paper, there are three pieces being glued together, as opposed to two, and two of the three pieces being glued in \emph{do not come initially equipped} with torsion-free $\Gt$-structures. This does not happen in the earlier constructions of the first author or of Kovalev and Corti--Haskins--Nordstr\"om--Pacini. In those cases, the pieces being glued together came already equipped with torsion-free $\Gt$-structures, and the main analytic difficulty was in controlling the size of the torsion in the gluing annulus. In the present construction we need to work hard to construct closed $\Gt$-structures with sufficiently small torsion on two of the three pieces to be able to apply Theorem \ref{nc2thm3}.
\end{itemize}

In fact, our construction actually produces a smooth compact torsion-free $\Gt$-manifold $N$, starting from a smooth compact torsion-free $\Gt$-manifold $M$ which admits a $\Gt$-involution $\iota$, defined in \S\ref{nc27}. In practice, the initial manifold $M$ would be of the form $M = \cS^1 \times Y$ for a Calabi--Yau $3$-fold $Y$. An antiholomorphic isometric involution on $Y$ induces a $\Gt$-involution on $M$. This is explained in \S\ref{nc27}.

Our construction can be very roughly summarized as follows. We take the quotient $M / \an{\io}$ of $M$ by the action of $\iota$. The fixed point set $L$ of $\iota$ is a smooth compact $3$-dimensional totally geodesic associative submanifold of $M$. The space $M / \an{\io}$ is a compact orbifold, with singular set $L$. Locally, the orbifold singularities are of the form $\R^3 \times (\R^4 \slash \{ \pm 1 \})$. We then cut out a tubular neighbourhood of $L$ in $M / \an{\io}$, and glue in a smooth family of Eguchi--Hanson spaces, smoothly parametrized by $L$, to obtain a smooth compact $7$-manifold $N$. We show that $N$ admits a $1$-parameter family of $\Gt$-structures $(\vp_t^N, g_t^N)$ that are closed and, for small $t$, have torsion small enough in a precise sense to be able to invoke Theorem \ref{nc2thm3} to establish the existence of a torsion-free $\Gt$-structure on $N$. When the fundamental group of $N$ is finite, this will be a holonomy $\Gt$ metric on $N$.

As mentioned above, the main technical difficulties arise in constructing closed $\Gt$-structures with sufficiently small torsion on two of the three pieces being glued together. These three pieces are roughly described as follows:
\begin{enumerate}[(i)]
\item the interior of the complement of a neighbourhood of $L$ in~$M / \an{\io}$;
\item an annulus in the quotient $\nu/ \{\pm 1\}$ of the normal bundle $\nu$ of $L$ in $M$ by the action of $\{ \pm 1\}$ on fibres;
\item the product of $L$ with a neighbourhood of the ``bolt'' $\cS^2$ in the Eguchi--Hanson space~$T^*\cS^2$.
\end{enumerate}
The pieces (ii) and (iii) do not come equipped with torsion-free $\Gt$-structures. On both pieces, there is a natural way to define a $\Gt$-structure, but that is not even closed. In both cases, one can then naturally correct these $\Gt$-structures to closed versions. But in both cases the torsion is still too large to be able to use these closed $\Gt$-structures to construct a $\Gt$-structure on $N$ that satisfies the hypotheses of Theorem \ref{nc2thm3}. Thus, in both cases we must perform a \emph{further correction}. For (ii), this correction involves modifying both the exponential map of $L$ in $M$ and the connection on $\nu$, both of which are ingredients used in constructing the initial $\Gt$-structure on (ii). For (iii), this correction involves solving a family of elliptic equations on the noncompact Eguchi--Hanson space, smoothly parametrized by the points in $L$.

A key point is that in order to perform our construction, we require the existence of a nonvanishing harmonic $1$-form $\lambda$ on $L$. This is a very strong assumption. (An explanation for the necessity of this assumption is given in Remark~\ref{nc6rem}.) In the case when $M = \cS^1 \times Y$, the submanifold $L$ is the disjoint union of two copies of a single special Lagrangian submanifold of $Y$. Its metric is induced from the Calabi--Yau metric on $Y$, which exists but cannot be made explicit. We discuss this point in \S\ref{nc7}, where we present several applications of the theorem to situations that could produce new examples. Some of these examples are obtained by using generalizations of our theorem described in~\S\ref{nc65}--\S\ref{nc66}.

Here is our main result, which is Theorem \ref{nc6thm} below. The various objects and terminology used in the theorem are made precise in the paper.

\begin{thm} Let\/ $(M,\vp,g)$ be a compact torsion-free $\Gt$-manifold, and let\/ $\io: M \ra M$ be a nontrivial involution preserving $(\vp,g),$ so that the fixed locus $L$ of\/ $\io$ is a compact associative $3$-fold in $M,$ as described in {\rm\S\ref{nc27}} below.

Suppose $L$ is nonempty, and suppose there exists a closed, coclosed, nonvanishing $1$-form $\la$ on $L$. That is, $\la\in\Om^1(L)$ with\/ $\d\la=0$ and $\d^*\la=0,$ where $\d^*$ is defined using $g\vert_L,$ and\/ $\la\vert_x\ne 0$ in $T_x^*L$ for all\/~$x\in L$.

Then there exists a compact\/ $7$-manifold\/ $N$ defined as a resolution of singularities $\pi:N\ra M/\an{\io}$ of the $7$-orbifold\/ $M/\an{\io}$ along its singular locus $L\subset M/\an{\io},$ by gluing in a bundle $\si:P\ra L$ along $L$, with fibre the Eguchi--Hanson space $X$ described in {\rm\S\ref{nc25},} where $P$ is constructed using $\la$. The preimage $\pi^{-1}(L)$ is a $5$-submanifold\/ $Q$ of\/ $N,$ and\/ $\pi\vert_Q:Q\ra L$ is a smooth bundle with fibre\/ $\cS^2$. The fundamental group of\/ $N$ satisfies $\pi_1(N)\!\cong\!\pi_1(M/\an{\io}),$ and the Betti numbers~are 
\begin{equation*}
b^k(N)=b^k(M/\an{\io})+b^{k-2}(L).
\end{equation*}

There exists a smooth family $(\ti\vp_t^N, \ti g_t^N)$ of torsion-free $\Gt$-structures on $N$ for $t \in (0,\ep],$ with\/ $\ep>0$ small, such that\/ $(\ti \vp_t^N, \ti g_t^N) \ra \pi^*(\vp,g)$ in $C^0$ away from $Q$ as $t \ra 0,$ and for each\/ $x\in L$ the fibre $\pi^{-1}(x) \cong \cS^2$ with metric $\ti g_t^N \vert_{\pi^{-1}(x)}$ approximates a small round\/ $2$-sphere with area $\pi t^2\bmd{\la\vert_x}$ for small\/ $t$. The metrics $\ti g_t^N$ on $N$ have holonomy $\Gt$ if and only if\/ $M/\an{\io}$ has finite fundamental group.

\label{nc1thm1}	
\end{thm}

In \S\ref{nc65}--\S\ref{nc66} we generalize Theorem \ref{nc1thm1}, replacing $M/\an{\io}$ by a more general $G_2$-orbifold, and twisting the 1-form $\la$ by a principal $\Z_2$-bundle on~$L$.

The paper is organized as follows. In \S\ref{nc2} we discuss some background material. The material in \S\ref{nc3}--\S\ref{nc5} is developed to prove our main theorem. Specifically, in \S\ref{nc3} we study $\Gt$-structures and correction forms on the normal bundle $\nu$ of $L$ in $M$, in \S\ref{nc4} we construct $\Gt$-structures and correction forms on the resolution $P$ of $\nu/\{\pm1\}$, and in \S\ref{nc5} we state and prove a theorem for the further correction on piece (iii) mentioned above. We then use the material from \S\ref{nc3}--\S\ref{nc5} to prove our main theorem in \S\ref{nc6}. Section \ref{nc7} applies our theorem to several situations where new examples can be obtained, including examples that would arise as consequences of the SYZ conjecture. Finally, in \S\ref{nc8} we briefly outline some directions for future study. An appendix presents two calculations that are needed in~\S\ref{nc34}.
\smallskip

\noindent {\bf Notation.} All manifolds and tensors are assumed to be 
smooth unless explicitly stated otherwise. Similarly, all bundles are at least smooth fibre bundles, although some of them will be vector bundles. The term \emph{nonvanishing} means the same thing as \emph{nowhere zero} or \emph{nowhere vanishing}. If $E$ is a vector bundle over $M$, then $\Gamma^{\infty} (E)$ denotes the space of smooth sections of $E$.

We use $v \cdot \beta$ to denote the interior product of a vector $v$ with a form $\beta$. The symbol $\cdot$ is also used to denote various natural bilinear pairings throughout, and occasionally in \S\ref{nc6} to denote ordinary multiplication in some particularly complicated expressions.

As is customary, we use $\square$ to denote the end of a proof. Less common, however, is our use of $\blacktriangle$ to denote the end of a definition, example, or remark. We do this because several definitions and remarks extend over many paragraphs, and it would otherwise be unclear to the reader where they ended.
\smallskip

\noindent {\bf Convention.} There are two sign conventions in $\Gt$ geometry. The convention we choose to use in the present paper is the one used by Bryant \cite{Brya} and the first author \cite{Joyc3}, but differs from the convention used by Bryant--Salamon \cite{BrSa} or Harvey--Lawson \cite{HaLa}. A detailed discussion of sign conventions and orientations in $\Gt$ geometry can be found in~\cite{Kari3}.
\smallskip

\noindent {\bf Acknowledgements.} The authors thank Xenia de la Ossa, Jason Lotay, and Johannes Nordstr\"om for helpful conversations. The first author was supported by a grant from the Simons Foundation, the `Simons Collaboration on Special Holonomy in Geometry, Analysis and Physics', 2016--2020. The second author was supported by a European Commission Marie Curie Fellowship number MIF1-CT-2006-039113 in 2007--2008, when this project began, and is currently supported by an NSERC Discovery Grant.

\section{Background material}
\label{nc2}

\subsection{\texorpdfstring{$\Gt$-structures on 7-manifolds}{G₂-structures on 7-manifolds}}
\label{nc21}

We now introduce the holonomy group $\Gt$, following \cite[\S 10.1]{Joyc3}.

\begin{dfn} Let $(x_1,\dots,x_7)$ be the standard
coordinates on $\R^7$. Define a 3-form $\vp_0$ on 
$\R^7$ by
\ea
\vp_0 & = \d x_1\!\w\!\d x_2\!\w\!\d x_3 - \d x_1\!\w\!\d x_4\!\w\!\d x_5- \d x_1\!\w\!\d x_6\!\w\!\d x_7 - \d x_2\!\w\!\d x_4\!\w\!\d x_6
\nonumber\\
&\qquad+ \d x_2\!\w\!\d x_5\!\w\!\d x_7 - \d x_3\!\w\!\d x_4\!\w\!\d x_7- \d x_3\!\w\!\d x_5\!\w\!\d x_6.
\label{nc2eq1}
\ea
The subgroup of $\GL(7,\R)$ preserving $\vp_0$ is the {\it 
exceptional Lie group\/} $\Gt$. It is a compact, semisimple, 14-dimensional 
Lie group, a subgroup of $\SO(7)$. It also preserves the Euclidean metric 
\e
g_0=(\d x_1)^2+\cdots+(\d x_7)^2
\label{nc2eq2}
\e
on $\R^7$, the orientation on $\R^7$, and the Hodge dual 4-form 
\ea 
*\vp_0 & = \d x_4\!\w\!\d x_5\!\w\!\d x_6\w\!\d x_7 - \d x_2\!\w\!\d x_3\!\w\!\d x_6\w\!\d x_7- \d x_2\!\w\!\d x_3\!\w\!\d x_4\w\!\d x_5 
\nonumber\\
& \qquad {} - \d x_1\!\w\!\d x_3\!\w\!\d x_5\w\!\d x_7
+ \d x_1\!\w\!\d x_3\!\w\!\d x_4\w\!\d x_6 
\nonumber \\
& \qquad {}
- \d x_1\!\w\!\d x_2\!\w\!\d x_5\w\!\d x_6 - \d x_1\!\w\!\d x_2\!\w\!\d x_4\w\!\d x_7.
\label{nc2eq3}
\ea

A {\it $\Gt$-structure} on a 7-manifold $M$ is a principal subbundle $Q$ of the frame bundle of $M$, with structure group $\Gt$. Each $\Gt$-structure $Q$ gives rise to a 3-form $\vp$, a metric $g$, a 4-form $*\vp$, and an orientation on $M$, such that every tangent space of $M$ admits an isomorphism with $\R^7$ identifying $\vp,g,*\vp$ and the orientation with $\vp_0,g_0,*\vp_0$ and the standard orientation on $\R^7$ respectively. Here $*\vp$ comes from $\vp$ by the Hodge star determined by $g$ and the orientation on $M$. Conversely, $\vp$ determines $Q$ uniquely. By an abuse of notation, we will refer to the pair $(\vp,g)$ as a $\Gt$-{\it structure}. We call the triple $(M,\vp,g)$ a $\Gt$-{\it manifold}. (In contrast to \cite[Ch.~10]{Joyc3}, we do not require $(\vp,g)$ to be torsion-free in a $\Gt$-manifold $(M,\vp,g)$, in the sense below.)
\label{nc2def1}
\end{dfn}

Sometimes we write $g,*\vp$ as $g_\vp,*_\vp\vp$ to emphasize that they are determined by $\vp$. Note that $g_\vp,*_\vp\vp$ are {\it nonlinear\/} functions of $\vp$, since although $*_\vp$ is linear, it depends on $g_\vp$, and so on $\vp$. Here is \cite[Prop.s 10.1.3 \& 10.1.5]{Joyc3}:

\begin{thm} Let\/ $(M,\vp,g)$ be a $\Gt$-manifold. The following are equivalent:
\begin{itemize}
\setlength{\itemsep}{0pt}
\setlength{\parsep}{0pt}
\item[{\rm(i)}] $\Hol(g)\subseteq \Gt,$ and\/ $\vp$ is the induced\/ $3$-form,
\item[{\rm(ii)}] $\nabla\vp=0$ on $M,$ where $\nabla$ is the
Levi-Civita connection of\/ $g,$ and
\item[{\rm(iii)}] $\d\vp=0$ and $\d(*_\vp\vp)=0$ on $M$.
\end{itemize}
If these hold then $g$ is Ricci-flat.
\label{nc2thm1}
\end{thm}

We call $\nabla\vp$ the {\it torsion} of the $\Gt$-structure $(\vp,g)$. If $(M,\vp,g)$ is a $\Gt$-manifold with $(\vp,g)$ torsion-free, we call $(M,\vp,g)$ a {\it torsion-free $\Gt$-manifold}. Here is \cite[Prop.~10.2.2 \& Th.~10.4.4]{Joyc3}:

\begin{thm} Let\/ $(M,\vp,g)$ be a compact torsion-free $\Gt$-manifold. Then $\Hol(g)=\Gt$ if and only if\/ $\pi_1(M)$ is finite. In this case the moduli space of metrics with holonomy $\Gt$ on $M,$ up to diffeomorphisms isotopic to the identity, is a smooth manifold of dimension~$b^3(M)$.
\label{nc2thm2}
\end{thm}

\subsection{\texorpdfstring{Exterior forms on $\Gt$-manifolds}{Exterior forms on G₂-manifolds}}
\label{nc22}

Exterior forms on a $\Gt$-manifold have a natural decomposition:

\begin{prop} Let\/ $(M,\vp,g)$ be a $\Gt$-manifold. Then\/ $\La^kT^*M$ has a natural orthogonal splitting into components as follows, where\/ $\La^k_l$ is a vector subbundle of rank\/ $l$ corresponding to an irreducible representation of\/~{\rm$\Gt$:}
\begin{equation*}
\begin{array}{llll}
{\rm(i)} & \La^1T^*M=\La^1_7,
&{\rm(ii)} &\La^2T^*M=\La^2_7\op\La^2_{14}, \\
{\rm(iii)} &\La^3T^*M=\La^3_1\op\La^3_7\op \La^3_{27},
& {\rm(iv)} &\La^4T^*M=\La^4_1\op\La^4_7\op\La^4_{27},\\
{\rm(v)} &\La^5T^*M=\La^5_7\op\La^5_{14}, 
& {\rm(vi)} &\La^6T^*M=\La^6_7.
\end{array}
\end{equation*}
The Hodge star\/ $*$ of\/ $g$ gives an isometry between\/ $\La^k_l$ and $\La^{7-k}_l$. Here
\begin{align*}
\La^2_{14} & = \Ker\bigl((*\vp \w \cdot) :\La^2T^*M \to \La^6T^*M\bigr),\quad \La^3_1 = \an{\vp}, \quad \La^4_1 = \an{*\vp}, \\
\La^4_7 & = \Im\bigl( (\vp \w \cdot) :T^*M \to \La^4T^*M\bigr), \quad \La^5_7 = \Im\bigl( (*\vp \w \cdot):T^*M \to \La^5T^*M\bigr).
\end{align*}

\label{nc2prop1}
\end{prop}

Write $\Om^k=\Ga^\iy(\La^kT^*M)$ and $\Om^k_l=\Ga^\iy(\La^k_l)$ for the vector spaces of smooth sections, so that we have splittings $\Om^2=\Om^2_7\op\Om^2_{14}$, $\Om^3=\Om^3_1\op\Om^3_7\op\Om^3_{27}$, and so on. Write $\pi_l:\Om^k\ra\Om^k_l$ for the projections to the factors in these splittings. So, for instance, if $\xi\in\Om^2$ is a 2-form on $M$ then~$\xi=\pi_7(\xi)+\pi_{14}(\xi)$.

These splittings are crucial in understanding the linearization of the map $\Theta$ that takes a $\Gt$-structure $\vp$ to the dual $4$-form $*_{\vp} \vp$, which we now discuss. 

\begin{dfn} Let $M$ be a 7-manifold. We call a 3-form $\vp$ on $M$ a {\it positive\/ $3$-form\/} if for every $p\in M$, there exists an isomorphism between $T_pM$ and $\R^7$ that identifies $\vp\vert_p$ and the 3-form $\vp_0$ of \eq{nc2eq1}. Each positive 3-form $\vp$ determines a unique $\Gt$-structure $(\vp,g)$ on $M$. Similarly, call a 4-form $\psi$ on $M$ a {\it positive\/ $4$-form\/} if for every $p\in M$, there exists an isomorphism between $T_pM$ and $\R^7$ that identifies $\psi\vert_p$ and the 4-form $*\vp_0$ of \eq{nc2eq3}.

Write points of $\La^kT^*M$ as $(x,\al)$ for $x\in M$ and $\al\in\La^kT_x^*M$. Write $\La^3_+T^*M$ and $\La^4_+T^*M$ for the subsets of $(x,\al)$ in $\La^3T^*M$, $\La^4T^*M$ with $\al$ positive (that is, $\al$ is identified with $\vp_0,*\vp_0$ in \eq{nc2eq1}, \eq{nc2eq3} by some isomorphism $T_xM\cong\R^7$). Then $\La^3_+T^*M\ra M$, $\La^4_+T^*M\ra M$ are fibre bundles (but not vector bundles), whose sections are positive 3- and 4-forms.

As the stabilizer group of $\vp_0$ in $\GL(7,\R)$ is $\Gt$, we see that $\La^3_+T^*M\ra M$ is a bundle with fibre $\GL(7,\R)/\Gt$, which has dimension $49-14=35$. Since $\La^3T^*M$ has rank $\binom{7}{3}=35$, we see that $\La^3_+T^*M$ is an open submanifold in $\La^3T^*M$ (note that it is not a vector subbundle). The stabilizer group of $*\vp_0$ in $\GL(7,\R)$ is $\Gt\t\{\pm 1\}$, so again $\La^4_+T^*M$ is open in~$\La^4T^*M$.

Define $\Th:\Ga^\iy(\La^3_+T^*M)\ra\Ga^\iy(\La^4_+T^*M)$ by $\Th(\vp)=*_\vp\vp$, where $*_\vp\vp$ is the Hodge dual form defined using the metric $g_\vp$ and orientation induced by the $\Gt$-structure $(\vp,g_\vp)$ associated to $\vp$. Then $\Th$ is a smooth, nonlinear map.
\label{nc2def2}
\end{dfn}

Here is \cite[Prop.~10.3.5]{Joyc3}, slightly rewritten.

\begin{prop} There exist universal constants $\ep,C>0$ such that the following holds. Let\/ $(M,\vp,g)$ be a $\Gt$-manifold, and\/ $\xi$ be a $3$-form on $M$ with $\nm{\xi}_{C^0}\le\ep$. Then\/ $\vp+\xi$ is a positive $3$-form, and
\e 
\Th(\vp+\xi)=*_\vp\vp+*_\vp\ts\bigl(\frac{4}{3}\pi_1(\xi)+\pi_7(\xi)-\pi_{27}(\xi)\bigr)+F_\vp(\xi),
\label{nc2eq4}
\e
where the nonlinear function $F_\vp:\bigl\{\xi\in\Om^3:\nm{\xi}_{C^0}\le\ep\bigr\}\ra\Om^4$ satisfies
\begin{equation} 
F_\vp(0)=0, \quad \bmd{F_\vp(\xi)} \le C\ms{\xi}, \quad \bmd{\nabla F_\vp (\xi)} \le C\bigl(\ms{\xi}\md{\nabla\vp} + \md{\xi}\md{\nabla \xi}\bigr).
\label{nc2eq5}
\end{equation}
Here all norms and the covariant derivatives are taken with respect to\/~$g$.
\label{nc2prop2}
\end{prop}

Equation \eq{nc2eq4} computes the derivative $D_\vp\Th:\Om^3\ra\Om^4$ of $\Th$ at $\vp$: we have
\e
D_\vp\Th = *_\vp\ci\ts\bigl(\frac{4}{3}\pi_1+\pi_7-\pi_{27}\bigr).
\label{nc2eq6}
\e

\subsection{\texorpdfstring{An existence theorem for torsion-free $\Gt$-manifolds}{An existence theorem for torsion-free G₂-manifolds}}
\label{nc23}

An important technical tool that we will require is the following theorem of the first author, which says that if one can find a $\Gt$-structure $(\vp,g)$ with $\d\vp=0$ on a compact 7-manifold $M$, whose torsion is sufficiently small in a certain sense, then there exists a {\it torsion-free\/} $\Gt$-structure $(\ti \vp,\ti g)$ on $M$ which is close to $(\vp,g)$, and in the same de Rham cohomology class. It was first used in \cite{Joyc1,Joyc2} to construct the first compact examples of manifolds with $\Gt$-holonomy.

\begin{thm}[Joyce {\cite[Th.~11.6.1]{Joyc3}}] Let\/ $\al,K_1,K_2,$ and\/ $K_3$ be any positive constants. Then there exist $\ep\in(0,1]$ and\/ $K_4>0,$ such that whenever $0<t \le\ep,$ the following holds.

Let\/ $M$ be a compact\/ $7$-manifold, with a\/ $\Gt$-structure $(\vp,g)$ satisfying $\d\vp=0$. Suppose there is a closed\/  $4$-form $\psi$ on $M$ such that:
\begin{itemize}
\setlength{\itemsep}{0pt}
\setlength{\parsep}{0pt}
\item[{\rm(i)}] $\bnm{\Th(\vp)-\psi}{}_{C^0} \le K_1 \, t^{\al},$ $\bnm{\Th(\vp)-\psi}{}_{L^2}\le K_1 \, t^{\frac{7}{2} + \al},$ and \\
$\bnm{\d(\Th(\vp)-\psi)}{}_{L^{14}} \le K_1 \, t^{-\frac{1}{2} + \al}$.
\item[{\rm(ii)}] the \begin{bfseries}injectivity radius\end{bfseries} $\mathrm{inj}$ of\/ $g$ satisfies $\mathrm{inj} \ge K_2 \, t$.
\item[{\rm(iii)}] the \begin{bfseries}Riemann curvature\end{bfseries} tensor $\mathrm{Rm}$ of\/ $g$ satisfies $\nm{\mathrm{Rm}}_{C^0} \le K_3 \, t^{-2}$.
\end{itemize}
Then there exists a smooth, \begin{bfseries}torsion-free\end{bfseries} $\Gt$-structure $(\ti \vp,\ti g)$ on $M$ such that\/ $\nm{\ti\vp-\vp}_{C^0}\le K_4 \,t^{\al}$ and\/ $[\ti\vp]=[\vp]$ in $H^3(M,\R)$. Here all norms are computed using the original metric $g$.
\label{nc2thm3}
\end{thm}

\begin{rem} We have rewritten the theorem slightly: \cite[Th.~11.6.1]{Joyc3} is expressed in terms of the 3-form $\chi = \vp - *_g \psi$ rather than the 4-form $\Th(\vp)-\psi$, and also \cite[Th.~11.6.1]{Joyc3} fixes $\al=\ha$, which is sufficient for the applications in \cite{Joyc3}, but \cite[bottom of p.~296]{Joyc3} explains how to generalize to all~$\al>0$.
\label{nc2rem1}
\end{rem}

\subsection{\texorpdfstring{$\Gt$-manifolds and hyperK\"ahler 4-manifolds}{G₂-manifolds and hyperKähler 4-manifolds}}
\label{nc24}

Let us identify $\R^7$ with coordinates $(x_1,\ldots,x_7)$ with $\R^3\op\H$ with coordinates $((x_1,x_2,x_3),(y_1,\ldots,y_4))$, where $\H=\R^4$ is the quaternions,~by
\begin{equation*}
(x_1,\ldots,x_7)\cong((x_1,x_2,x_3),(y_1,y_2,y_3,y_4))=((x_1,x_2,x_3),(x_4,x_5,x_6,x_7)).
\end{equation*}
Then in equations \eq{nc2eq1}--\eq{nc2eq3} we have 
\e
\begin{aligned}
\vp_0&=\d x_1\w\d x_2\w\d x_3-\d x_1\w\om_0^I-\d x_2\w\om_0^J-\d x_3\w\om_0^K,\\
g_0&=(\d x_1)^2+(\d x_2)^2+(\d x_3)^2+h_0,\\ 
*\vp_0&= \vol_{\H} - \d x_2\w\d x_3\w\om_0^I-\d x_3\w\d x_1\w\om_0^J-\d x_1\w\d x_2\w\om_0^K,\\
\end{aligned}
\label{nc2eq7}
\e
where
\ea
\om_0^I &= \d y_1\w\d y_2+\d y_3\w\d y_4, \qquad \om_0^J = \d y_1\w\d y_3+\d y_4\w\d y_2, 
\nonumber\\
\om_0^K &= \d y_1\w\d y_4+\d y_2\w\d y_3, \qquad h_0 = (\d y_1)^2+(\d y_2)^2+(\d y_3)^2+(\d y_4)^2, 
\nonumber\\
\vol_\H &= \d y_1\!\w\!\d y_2\!\w\!\d y_3\!\w\!\d y_4 = \ha\om_0^I\!\w\!\om_0^I = \ha\om_0^J\!\w\!\om_0^J \!=\! \ha\om_0^K\!\w\!\om_0^K.
\label{nc2eq8}
\ea
Here $h_0$ is the Euclidean metric on $\H$, with volume form $\vol_\H$. As in \cite[Ch.~7]{Joyc3}, the metric $h_0$ is {\it hyperK\"ahler}, so it is K\"ahler with respect to three different complex structures $I,J,K$ on $\H$. Here $\om_0^I,\om_0^J,\om_0^K$ are the K\"ahler forms of $I,J,K$ for $h_0$. Explicitly, we have $\om_0^I (v,w) = h_0(Iv, w)$ and similarly for $J,K$. It is clear that the three K\"ahler forms are all \emph{self-dual} with respect to $h_0$ and the standard orientation. The three complex structures satisfy the quaternion multiplication relations when acting on tangent vector fields:
\e
\begin{aligned}
& I^2=J^2=K^2=-\mathbf{1}, \qquad IJ=-JI=K, \\
& JK=-KJ=I, \qquad KI=-IK=J.
\end{aligned}
\label{nc2eq9}
\e
In fact, $\H$ has an entire $\cS^2$ family of complex structures, with respect to each of which $h_0$ is a K\"ahler metric. Explicitly, $J_{\bf u} = u_1 I + u_2 J + u_3 K$ is such a complex structure, for any $\mathbf{u} = (u_1, u_2, u_3) \in \R^3$ satisfying $u_1^2 + u_2^2 + u_3^2 = 1$, with K\"ahler form~$\om_0^{J_{\bf u}}=u_1\om_0^I+u_2\om_0^J+u_3\om_0^K$.

Given a complex structure $I$ acting on tangent vectors, we define its dual map acting on 1-forms, also denoted $I$, by $(I\alpha)(v) = \alpha(Iv)$. In particular, with this definition it follows that when acting on 1-forms, we have:
\e
\begin{aligned}
& I^2=J^2=K^2=-\mathbf{1}, \qquad IJ=-JI=-K, \\
& JK=-KJ=-I, \qquad KI=-IK=-J.
\label{nc2eq10}
\end{aligned}
\e
Explicitly, here we have
\ea
I(\d y_1)&=-\d y_2, & \!I(\d y_2)&=\d y_1, & I(\d y_3)&=-\d y_4, & \!I(\d y_4)&=\d y_3, 
\nonumber\\
J(\d y_1)&=-\d y_3, & \!J(\d y_2)&=\d y_4, & J(\d y_3)&=\d y_1, & \!J(\d y_4)&=-\d y_2, 
\nonumber\\
K(\d y_1)&=-\d y_4, & \!K(\d y_2)&=-\d y_3, & K(\d y_3)&=\d y_2, & \!K(\d y_4)&=\d y_1.\!\!
\label{nc2eq11}
\ea
We note here for future use that, from the orthogonality of $I$, the self-duality $\ast \omega_I = \omega_I$, and the relation $\om_0^I (v,w) = h_0(Iv, w)$ it follows that
\e
\ast ( \alpha \wedge \omega_I) = - I \alpha \qquad \text{for any 1-form $\alpha$},
\label{nc2eq12}
\e
and similarly for $J,K$. Equation \eq{nc2eq12} is a purely linear algebraic fact that holds on any hyperK\"ahler $4$-manifold.

The Lie subgroup of $\GL(4,\R)$ acting on $\H=\R^4$ preserving $\om_0^I,\om_0^J,\om_0^K$ is $\SU(2)\!\cong\!\Sp(1)$. So the action of $\SU(2)$ on $\R^7\cong\R^3\op\H$ which is trivial on $\R^3$ and left multiplication on $\H$ preserves $\vp_0,*\vp_0$ by \eq{nc2eq7}--\eq{nc2eq8}. This induces an embedding of Lie subgroups $\SU(2)\hookra \Gt$, since $\Gt$ is the subgroup of $\GL(7,\R)$ fixing~$\vp_0$. 

Riemannian 4-manifolds with holonomy $\SU(2)$ are known as {\it Calabi--Yau\/ $2$-folds}, or {\it hyperK\"ahler $4$-manifolds}, as in \cite[Ch.s 6 \& 7]{Joyc3} for instance. By the general theory of Riemannian holonomy, the inclusion $\SU(2)\hookra \Gt$ above means that if $X$ is a hyperK\"ahler $4$-manifold, with metric $h$, volume form $\vol_X$, complex structures $I,J,K$, and K\"ahler form $\om^I,\om^J,\om^K$, then the product 7-manifold $M=\R^3\t X$ is a torsion-free $\Gt$-manifold, with 3-form $\vp$, metric $g$ and 4-form $\psi$ given by
\e
\begin{split}
\vp&=\d x_1\w\d x_2\w\d x_3-\d x_1\w\om^I-\d x_2\w\om^J-\d x_3\w\om^K,\\
g&=(\d x_1)^2+(\d x_2)^2+(\d x_3)^2+h,\\ 
\psi&=\vol_X-\d x_2\w\d x_3\w\om^I-\d x_3\w\d x_1\w\om^J-\d x_1\w\d x_2\w\om^K,
\end{split}
\label{nc2eq13}
\e
as in \eq{nc2eq7}. The $\Gt$-manifolds of this type have holonomy $\SU(2)\subset \Gt$.

\subsection{The Eguchi--Hanson space}
\label{nc25}

Following our discussion of hyperK\"ahler 4-manifolds in \S\ref{nc24}, we now describe a particular example of a hyperK\"ahler 4-manifold that plays a crucial role in our construction. The Eguchi--Hanson space was originally discovered in \cite{EgHa1, EgHa2} but it can be described from several points of view. These descriptions include: the resolution (blow-up) of $\C^2/\Z_2$ at the origin, which describes it as a complex manifold; the simplest case of both the Calabi construction \cite{Cala} of hyperK\"ahler metrics on $T^* \CP^n$ and the Stenzel construction \cite{Sten} of Calabi--Yau metrics on $T^* S^n$; and the hyperK\"ahler quotient construction of Eguchi--Hanson space. A good survey of hyperK\"ahler manifolds in general and the Eguchi--Hanson space in particular, discussing various different approaches, can be found in~\cite{Danc}.

Write the Eguchi--Hanson space as $X$, with blow-up map $B:X\ra\C^2/\{\pm 1\}$, and complex structure $I$ on $X$, $\C^2$, and $\C^2/\{\pm 1\}$. The exceptional divisor of the blow-up is written $Y=B^{-1}(0)$, where $Y\cong\CP^1$. Write $r:\C^2\ra[0,\iy)$ for the radius function $r(z_1,z_2)=(\ms{z_1}+\ms{z_2})^{1/2}$. Also write $r:\C^2/\{\pm 1\}\ra[0,\iy)$ for the function descending from $r:\C^2\ra[0,\iy)$, and write $r:X\ra[0,\iy)$ in place of $r\ci B$, so that~$Y=r^{-1}(0)\subset X$.

\begin{rem} The map $B:X\ra \C^2/\{\pm 1\}$ is {\it not a smooth map\/} in the sense of orbifolds. This means that given some smooth function or tensor $T$ on $\C^2$ which is invariant under $\{\pm 1\}$, then $T$ descends to a smooth function or tensor on $\C^2/\{\pm 1\}$ in the orbifold sense, but the pullback $B^*(T)$ {\it may not be smooth (or defined) on the exceptional divisor\/} $Y=B^{-1}(0)$.

However, since $B$ is defined using complex geometry, holomorphic objects on $\C^2/\{\pm 1\}$ do pull back to holomorphic (and hence smooth) objects on $X$. For example, the holomorphic functions $z_1^2$, $z_2^2$, and $z_1z_2$ on $\C^2/\{\pm 1\}$ pull back to holomorphic functions on $X$. Although $r^2:\C^2/\{\pm 1\}\ra\R$ is smooth in the orbifold sense, its pullback $r^2:X\ra\R$ is not smooth on $X$. But $r^4:X\ra\R$ is smooth, since 
\begin{equation*}
r^4=\ms{z_1^2}+\ms{z_2^2}+2\ms{z_1z_2},
\end{equation*}
where $z_1^2$, $z_2^2$, and $z_1z_2$ are holomorphic, and hence smooth, on $X$. Near $Y$, we can think of $r^4$ as the squared distance to $Y$ in $X$.

In \S\ref{nc4} we will work with a bundle $\si:P\ra L$ with Eguchi--Hanson space fibres $X_x$ for $x\in L$. We will have to be careful that the metrics and forms we define on $P$ extend smoothly over the exceptional divisors $Y_x$ in the fibres~$X_x$.
\label{nc2rem2}
\end{rem}

On $\R^4=\C^2$ use real coordinates $(y_1,y_2,y_3,y_4)$ and complex coordinates $(z_1,z_2)$ with $z_1=y_1+iy_2$, $z_2=y_3+iy_4$, so that $r^2=y_1^2+y_2^2+y_3^2+y_4^2=\ms{z_1}+\ms{z_2}$. We have a Euclidean metric $h_0$, complex structures $I,J,K$ where $z_1,z_2$ are holomorphic with respect to $I$, and K\"ahler forms $\om^I_0,\om^J_0,\om^K_0$, where
\begin{equation*}
\begin{aligned}
h_0&=\d y_1^2+\d y_2^2+\d y_3^2+\d y_4^2=\ms{\d z_1}+\ms{\d z_2},\\
\om^I_0&=\d y_1\w\d y_2+\d y_3\w\d y_4=\ts\frac{i}{2}(\d z_1\w\d\bar z_1+\d z_2\w\d\bar z_2),\\
\end{aligned}
\end{equation*}
and
\begin{equation*}
\begin{aligned}
& \om^J_0=\d y_1\w\d y_3+\d y_4\w\d y_2, \qquad \om^K_0=\d y_1\w\d y_4+\d y_2\w\d y_3, \\
& \om^J_0+i\om^K_0=\d z_1\w\d z_2.
\end{aligned}
\end{equation*}
Since $I(\d z_1)=i\d z_1$ and $I(\d z_2)=i\d z_2$, we have
\begin{equation*}
I\d y_1=-\d y_2,\quad I\d y_2=\d y_1,\quad I\d y_3=-\d y_4,\quad I\d y_4=\d y_3.	
\end{equation*}
Then we may also write
\e
\om^I_0=\ts-\frac{1}{4}\d[I (\d (\ts r^2))],
\label{nc2eq14}
\e
where $r^2=y_1^2+y_2^2+y_3^2+y_4^2:\C^2\ra\R$ is the K\"ahler potential of $\om_0$ on $(\C^2,I)$.

For $a>0$, define $f_a:(0,\iy)\ra\R$ by
\e
\begin{split}
f_a(r)&=\sqrt{r^4+a^2}-a\log \biggl(\frac{\sqrt{r^4+a^2}+a}{r^2}\biggr)\\
&=\sqrt{r^4+a^2}+2a\log r-a\log \Bigl(\sqrt{r^4+a^2}+a\Bigr).
\end{split}
\label{nc2eq15}
\e
Then $f_a(r)=af_1(a^{-1/2}r)$. There is a unique Ricci-flat K\"ahler metric $h_a$ on $(X,I)$ with K\"ahler form $\om^I_a$, where analogous to \eq{nc2eq14}, on $X\sm Y$ we have
\e
\om^I_a\vert_{X\sm Y}=\ts-\frac{1}{4}\d[I (\d f_a(r))],
\label{nc2eq16}
\e

The function $r^2$ is the K\"ahler potential of the Euclidean metric $h_0$ on $\C^2/\{\pm 1\}$. We need to understand the difference between the metrics $h_a$ and $h_0$ near infinity.  Define a smooth function $G_a:(0,\iy)\ra\R$ by
\e
G_a(r)=f_a(r)-r^2.
\label{nc2eq17}
\e
Then from the first line of \eq{nc2eq15} we can show that for large $r$ we have 
\e
\nabla^kG_a(r)=O(r^{-2-k}) \;\>\text{as $r\ra\iy$, for all $k\ge 0$,}
\label{nc2eq18}
\e
where $\nabla$ is the Levi-Civita connection of $h_0$. Equation \eq{nc2eq18} when $k=2$ implies that $h_a=h_0+O(r^{-4})$ as $r\ra\iy$. That is, $h_a$ is {\it Asymptotically Locally Euclidean} (ALE). See \cite[Chapters 7--8]{Joyc3} for more background on ALE metrics.

The expression \eq{nc2eq16} defines $\om^I_a$ on $X \sm Y$. We will show how it in fact extends smoothly over $Y$ in a unique way. From the second line of \eq{nc2eq15}, write
\e
f_a(r)=2a\log r+H_a(r^4),
\label{nc2eq19}
\e
where $H_a:(-a^2,\iy)\ra\R$ is the smooth function given by
\begin{equation*}
H_a(u)=\sqrt{u+a^2}-a\log\Bigl(\sqrt{u+a^2}+a\Bigr).
\end{equation*}
Note that as $r^4:X\ra\R$ is smooth by Remark \ref{nc2rem2} and takes values in $[0,\iy)\subset (-a^2,\iy)$, this implies that $H_a(r^4)$ is a smooth function on all of $X$, including~$Y$.

Define 2-forms $\om^J,\om^K$ on $X$ by $\om^J+i\om^K=\pi^*(\om^J_0+i\om^K_0)$. These are smooth by Remark \ref{nc2rem2}, because $\om^J_0+i\om^K_0$ is holomorphic on $\C^2/\{\pm 1\}$. There are unique complex structures $J_a,K_a$ on $X$ such that $h_a$ is K\"ahler with respect to $J_a,K_a$ with K\"ahler forms $\om^J,\om^K$. Then $(X,h_a)$ is hyperK\"ahler, with complex structures $(I,J_a,K_a)$ and K\"ahler forms $(\om^I_a,\om^J,\om^K)$. Note that $h_a,J_a,K_a,\om^I_a$ depend on $a\in(0,\iy)$, but $I,\om^J,\om^K$ do not. As $a\ra 0$, these $h_a$, $(I,J_a,K_a)$, $(\om^I_a,\om^J,\om^K)$ approach the Euclidean hyperK\"ahler structure on~$\C^2/\{\pm 1\}$.

There is a natural holomorphic map $\pi:X\ra Y=\CP^1$, acting by $\pi(x)=[z_1,z_2]$ if $B(x)=(z_1,z_2)$ for $x\in X\sm Y$, which realizes $X$ as the total space of the holomorphic line bundle $\cO (-2) = T^*\CP^1$ over $\CP^1$. On $X\sm Y$ we have
\e
-\ts\frac{1}{4}\d[I (\d (2\log r))]=\pi^*(\om_{\CP^1})\vert_{X\sm Y},
\label{nc2eq20}
\e
where $\om_{\CP^1}$ is the K\"ahler form of the usual Fubini--Study metric on $\CP^1$. Combining \eq{nc2eq16}, \eq{nc2eq19} and \eq{nc2eq20} shows that
\e
\om^I_a=a\,\pi^*(\om_{\CP^1})\ts-\frac{1}{4}\d[I (\d H_a(r^4))].
\label{nc2eq21}
\e
Note that although \eq{nc2eq16} only makes sense on $X\sm Y$ since $f_a(r(x))\ra -\iy$ as $x\ra y \in Y$ in $X$, equation \eq{nc2eq21} is valid on all of $X$, because both $\pi^*(\om_{\CP^1})$ and $H_a(r^4)$ are smooth on~$X$.

\begin{rem} The submanifold $Y$ of $X$, corresponding to the zero section $\CP^1 \cong \cS^2$ in $X = T^* \CP^1$, is sometimes called the \emph{bolt} of $X$. It is well known that the Fubini--Study metric on $\CP^1$ is isometric to the round metric on $S^2$ of radius $\frac{1}{2}$. It follows from \eq{nc2eq21} that the area of the bolt is $\pi a$.
\label{nc2rem3}
\end{rem}

By the material of \S\ref{nc24}, we see that we have a torsion-free $\Gt$-structure $(\vp_a,g_a)$ on $\R^3\t X$ for $a>0$, with 4-form $\psi_a$, given by
\begin{align*}
\vp_a&=\d x_1\w\d x_2\w\d x_3-\d x_1\w\om_a^I-\d x_2\w\om^J-\d x_3\w\om^K,\\
g_a&=(\d x_1)^2+(\d x_2)^2+(\d x_3)^2+h_a,\\ 
\psi_a&=\vol_X-\d x_2\w\d x_3\w\om_a^I-\d x_3\w\d x_1\w\om^J-\d x_1\w\d x_2\w\om^K.
\end{align*}

We can generalize this by rescaling distances in $X$ by $t>0$, replacing $h,\om_a^I,\om^J,\om^K,\vol_\H$ by $t^2h,t^2\om_a^I,t^2\om^J,t^2\om^K,t^4\vol_\H$, respectively, giving a family of torsion-free $\Gt$-structures $(\vp_{a,t},g_{a,t})$ on $\R^3\t X$ for $a,t>0$ with 4-forms~$\psi_{a,t}$:
\ea
\vp_{a,t}&=\d x_1\w\d x_2\w\d x_3-t^2\d x_1\w\om_a^I-t^2\d x_2\w\om^J-t^2\d x_3\w\om^K,
\nonumber\\
g_{a,t}&=(\d x_1)^2+(\d x_2)^2+(\d x_3)^2+t^2h_a,
\label{nc2eq22}\\ 
\psi_{a,t}&=t^4\vol_X-t^2\d x_2\w\d x_3\w\om_a^I-t^2\d x_3\w\d x_1\w\om^J-t^2\d x_1\w\d x_2\w\om^K.
\nonumber
\ea
These $(\vp_{a,t},g_{a,t}),\psi_{a,t}$ will be our local models in \S\ref{nc4}.

\subsection{\texorpdfstring{$\Gt$-manifolds and Calabi--Yau 3-folds}{G₂-manifolds and Calabi--Yau 3-folds}}
\label{nc26}

Let us identify $\R^7$ with coordinates $(x_1,\ldots,x_7)$ with $\R\op\C^3$ with coordinates $(x,(z_1,z_2,z_3))$ by
\begin{equation*}
(x_1,\ldots,x_7)\cong(x,(z_1,z_2,z_3))=(x_4,(x_1+ix_5,x_2+ix_6,x_3+ix_7)).
\end{equation*}
Then in equations \eq{nc2eq1}--\eq{nc2eq3} we have 
\begin{gather}
\vp_0=\d x\!\w\om_0 +\Re\Om_0, \quad g=\d x^2+h_0, \quad *\vp_0=-\d x\w\Im\Om_0+\ha \om_0\w\om_0,
\nonumber\\
\text{where}\quad
\om_0=\ts\frac{i}{2}\bigl(\d z_1\w\d\bar z_1+\d z_2\w\d\bar z_2+\d z_3\w\d\bar z_3\bigr),
\nonumber\\
\Om_0=\d z_1\w\d z_2\w\d z_3,\quad\text{and}\quad h_0=\ms{\d z_1}+\md{\d z_2}+\ms{\d z_3},
\label{nc2eq23}
\end{gather}
so that $\om_0$ is the standard K\"ahler form, $\Om_0$ the holomorphic volume form, and $h_0$ the K\"ahler metric on $\C^3$, normalized by the equation $\vol_{\C^3}=\frac{1}{6}\om_0^3=\frac{i}{8}\,\Om_0\w\bar\Om_0$. The Lie subgroup of $\GL(6,\R)$ acting on $\R^6\cong\C^3$ preserving both $\om_0$ and $\Om_0$ is $\SU(3)$. So the action of $\SU(3)$ on $\R^7\cong\R\op\C^3$ which is trivial on $\R$ and as usual on $\C^3$ preserves $\vp_0,*\vp_0$ by \eq{nc2eq23}. This induces an embedding of Lie subgroups $\SU(3)\hookra \Gt$, since $\Gt$ is the subgroup of $\GL(7,\R)$ fixing~$\vp_0$. 

Riemannian 6-manifolds with holonomy $\SU(3)$ are known as {\it Calabi--Yau\/ $3$-folds}, as in \cite[Ch.~6]{Joyc3} for instance. By the general theory of Riemannian holonomy, the inclusion $\SU(3)\hookra \Gt$ above means that if $Y$ is a Calabi--Yau 3-fold, with Ricci-flat K\"ahler metric $h$, complex structure $J$, K\"ahler form $\om$, and holomorphic $(3,0)$-form $\Om$, normalized by $\vol_Y=\frac{1}{6}\om^3=\frac{i}{8}\,\Om\w\bar\Om$ and 
\begin{equation*} 
h(v,w)=\om(v,Jw)
\end{equation*}
for vector fields $v$ and $w$ on $Y$, and $\cS^1=\R/\Z$ is the circle with coordinate $x\in\R/\Z$ and metric $\d x^2$, then the product 7-manifold $M=\cS^1\t Y$ is a torsion-free $\Gt$-manifold, with 3-form $\vp$, metric $g$ and 4-form $\psi$ given by
\begin{equation} 
\vp=\Re\Om+\d x\w\om, \qquad g= \d x^2+h,\qquad \psi=-\d x\w\Im\Om+\ha\om\w\om,
\label{nc2eq24}
\end{equation}
as in \eq{nc2eq23}. The $\Gt$-manifolds of this type have holonomy $\SU(3)\subset \Gt$.

\subsection{\texorpdfstring{$\Gt$-involutions and their fixed point sets}{G₂-involutions and their fixed point sets}}
\label{nc27} 
 
We now define what we mean by a $\Gt$-involution $\iota$, which we need to construct the orbifold $M/\an{\iota}$.

\begin{dfn} Let $(M,\vp)$ be a $\Gt$-manifold, and $\iota : M \to M$ be a diffeomorphism of $M$. We call $\iota$ a \emph{$\Gt$-involution} if $\iota$ is an involution which preserves the $\Gt$-structure. This means $\iota^2 = \mathbf{1}$, where $\mathbf{1}$ is the identity map of $M$, and $\iota^*(\vp) = \vp$. Such a map $\iota$ also clearly satisfies $\iota^*(g) = g$ and $\iota^*(\psi) = \psi$.
\label{nc2def3}
\end{dfn}

\begin{ex} Suppose that $M = \cS^1 \times Y$, with $Y$ a Calabi--Yau $3$-fold, as described in \S\ref{nc26}, and $\cS^1=\R/\Z$. An \emph{anti-holomorphic isometric involution} $\tau$ of $Y$ is a diffeomorphism satisfying $\tau^2 = \mathbf{1}$ (involution), $\tau^*(g) = g$ (isometry), and $\tau^*(J) = - J$ (anti-holomorphic). It follows that $\tau^*(\omega) = -\omega$ and $\tau^*(\Omega) = \bar \Omega$. Define an involution on $M=\cS^1 \times Y$ by $\iota (x,y) = (-x, \tau(y))$. From equation \eq{nc2eq24} we see that
\begin{equation*}
\iota^*(\vp) \!=\! \Re(\tau^*(\Omega)) \!+\! (-\d x) \!\wedge\! \tau^*(\omega)\! =\! \Re(\bar \Omega) \!+\! (-\d x) \wedge (-\omega) \!=\! \Re(\Omega) \!+\! \d x  \wedge \omega \!=\! \vp,
\end{equation*}
so that such an $\iota$ is a $\Gt$-involution.
\label{nc2ex1}
\end{ex}

We will need to understand the fixed point set of a $\Gt$-involution. The following is \cite[Prop.~10.8.1]{Joyc3}, except for $L$ totally geodesic, which is obvious.

\begin{prop} Let\/ $(M,\vp)$ be a connected, compact\/ $\Gt$-manifold, and let\/ $\iota$ be a $\Gt$-involution of\/ $M$. Let\/ $L$ be the fixed point set of\/ $\iota$. That is, $L = \{ p \in M:\, \iota(p) = p\}$. Suppose that\/ $\iota$ is not the identity ($L\neq M$), and that\/ $\iota$ has at least one fixed point\/ ($L\neq\es$). Then $L$ is a smooth, orientable $3$-dimensional compact submanifold of\/ $M$ which is totally geodesic, and, with respect to a canonical orientation, is \begin{bfseries}associative\end{bfseries}.
\label{nc2prop3}
\end{prop}

\begin{rem} An associative submanifold is calibrated with respect to $\vp$, in the sense of Harvey--Lawson~\cite{HaLa}. In particular, we note here for future reference the important fact that the normal bundle $\nu$ of an associative submanifold $L$ of $(M, \vp)$ is always topologically trivial. This fact is well-known, but hard to find in the literature, so we outline the argument here. We use the algebraic properties of the octonions (see~\cite{HaLa}). The fibre of the normal bundle $\nu_p$ over a point $p \in L$ is a coassociative $4$-plane in $T_p M$, and thus the cross product $s_p \times X_p$ of a normal vector $s_p \in \nu_p$ with a tangent vector $X_p \in T_p L$ always lies in the normal space $\nu_p$. Let $\{ e_1, e_2, e_3 \}$ be a global frame for $L$, which exists since all $3$-manifolds are parallelizable. Let $s$ be a non-vanishing section of $\nu$, which exists because $\nu$ is an $\R^4$-bundle over a $3$-manifold. Then the set $\{ s, s \times e_1, s \times e_2, s \times e_3 \}$ is a global frame for $\nu$, so $\nu$ is topologically trivial.
\label{assocrmk}
\end{rem}

\begin{ex} Consider the $\Gt$-involution on $\cS^1 \times Y$ of Example \ref{nc2ex1}. If $(x,y)$ is a fixed point of $\iota$, then $x=-x$ in $\R/\Z$ and $\tau(y)=y$. This means $x=0$ or $x=\ha$, and $y$ is a fixed point of $\tau$. Since $\tau^*(\omega) = - \omega$ and $\tau^*(\Im(\Omega)) = - \Im(\Omega)$, we see that on the fixed point set $L$ of $\tau$, which by Proposition \ref{nc2prop3} is a $3$-dimensional submanifold of $Y$, the forms $\omega$ and $\Im(\Omega)$ both vanish. Thus $L$ is a special Lagrangian submanifold of $Y$, and the fixed point set of $\iota$ in $M$ is two copies of $L$, namely $\{0 \} \times L$ and~$\{\ha\} \times L$.
\label{nc2ex2}
\end{ex}

\section{\texorpdfstring{$\Gt$-structures on the normal bundle $\nu$ of $L$ in~$M$}{G₂-structures on the normal bundle ν of L in M}}
\label{nc3}

From here until the end of \S\ref{nc64} we will work in the following situation.

\begin{ass} Suppose we are given a compact, torsion-free $\Gt$-manifold\/ $(M,\vp,g),$ and a nontrivial involution $\io:M\ra M$ with\/ $\io^*(\vp)=\vp,$ so that\/ $\io^*(g)=g$. Write $L$ for the fixed locus of\/ $\io$ in $M$. Then $L$ is an associative $3$-fold in $(M,\vp,g)$ by Proposition\/ {\rm\ref{nc2prop3},} and is compact as $M$ is. 

We suppose $L$ is nonempty, but not necessarily connected. Suppose we are given a closed, coclosed, nonvanishing $1$-form $\la$ on $L$. That is, $\la\in\Om^1(L)$ with\/ $\d\la=\d^*\la=0,$ where $\d^*$ is defined using $g\vert_L,$ and\/ $\la\vert_x\ne 0$ in $T_x^*L$ for all\/~$x\in L$. (See Remark~\ref{nc6rem} for a justification of the necessity of this assumption.)
\label{nc3ass}	
\end{ass}

\subsection{\texorpdfstring{Choosing a tubular neighbourhood of $L$}{Choosing a tubular neighbourhood of L}}
\label{nc31}

Let $(M,\vp,g),\io,L$ be as in Assumption \ref{nc3ass}. Write $\nu\ra L$ for the normal bundle of $L$ in $M$, a rank 4 real vector bundle in the exact sequence
\e
\xymatrix@C=30pt{ 0 \ar[r] & TL \ar[r] & TM\vert_L \ar[r] & \nu \ar[r] & 0. }
\label{nc3eq1}
\e
There is a unique vector bundle isomorphism
\e
TM\vert_L\cong \nu\op TL
\label{nc3eq2}
\e
compatible with \eq{nc3eq1} such that the subbundles $\nu$, $TL$ in $TM\vert_L$ are orthogonal with respect to $g$. Write $g_L\in\Ga^\iy(S^2T^*L)$ for the restriction of $g$ to $L$, and $h_\nu\in\Ga^\iy(S^2\nu^*)$ for the restriction of $g$ to the factor $\nu$ in \eq{nc3eq2}, so that
\begin{equation*}
g\vert_L=h_\nu\op g_L\qquad\text{in $\Ga^\iy(S^2T^*M\vert_L)$.}
\end{equation*}

Write $\nabla^g$ for the Levi-Civita connection of $g$, and $\nabla^\nu$ for the connection on $\nu\ra L$ induced from the restriction of $\nabla^g$ to $TM\vert_L$ using \eq{nc3eq2}, and $\nabla^{g_L}$ for the Levi-Civita connection of $g_L$. Recall by Proposition \ref{nc2prop3} that $L$ is totally geodesic in $(M,g)$, hence the second fundamental form of $L$ in $M$ vanishes. Consequently, suppose $v,w\in\Ga^\iy(TL)$, and $\al\in\Ga^\iy(\nu)$. Choose vector fields $\ti v,\ti w\in\Ga^\iy(TM)$ with $\ti v\vert_L\cong 0\op v$, $\ti w\vert_L\cong \al\op w$ under the isomorphism \eq{nc3eq2}. Then under \eq{nc3eq2} we have
\e
(\nabla^g_{\ti v}\ti w)\vert_L\cong (\nabla^\nu_v\al)\op(\nabla^{g_L}_vw).
\label{nc3eq3}
\e

We will choose data $U_R\subset\nu$ and $\Up:U_R\ra M$ as in the next definition:

\begin{dfn} Write points of $\nu$ as $(x,\al)$ for $x\in L$ and $\al\in\nu_x$. Fix $R>0$, and define
\e
U_R=\bigl\{(x,\al)\in\nu:\md{\al}_{h_\nu}<R\bigr\},
\label{nc3eq4}
\e
so that $U_R$ is a tubular neighbourhood of the zero section in $\nu$. Write $\pi:U_R\ra L$ for the projection $\pi:(x,\al)\mapsto x$. We will choose a smooth map $\Up:U_R\ra M$ satisfying the following conditions:
\begin{itemize}
\setlength{\itemsep}{0pt}
\setlength{\parsep}{0pt}
\item[(i)] $\Up$ is a diffeomorphism with an open neighbourhood of $L$ in $M$.
\item[(ii)] $\Up(x,0)=x$ for $x\in L$.
\item[(iii)] $\Up(x,-\al)=\io\ci\Up(x,\al)$ for all $(x,\al)\in U_R$.  (Compatibility with the involution.)
\item[(iv)] Identify $L$ with the zero section $\bigl\{(x,0):x\in L\bigr\}$ in $U_R\subset\nu$. Then there is a natural identification $TU_R\vert_L\cong \nu\op TL$. Also (i),(ii) imply that $\Up_*\vert_L:TU_R\vert_L\ra TM\vert_L$ is a vector bundle isomorphism. The composition of these isomorphisms $\nu\op TL\cong TU_R\vert_L\cong TM\vert_L$ must agree with \eq{nc3eq2}. Another way to say this is that the induced pushforward $\Up_* : T U_R \to T M$ restricted to the zero section of $T U_R$ is the identity map on $T_x L \oplus \nu_x$.
\end{itemize}
\end{dfn}

One way to define such a map $\Up$ is using {\it exponential normal coordinates} along $L$: for each $(x,\al)$ in $U_R$ there is a unique geodesic $\ga:[0,1]\ra M$ in $(M,g)$ with $\ga(0)=x$, $\dot\ga(0)=\al$ and length $\md{\al}_{h_\nu}$, and we set $\Up(x,\al)=\ga(1)$. Then provided $R>0$ is small enough, conditions (i)--(iv) hold. However, in Proposition \ref{nc3prop1} we will choose $U_R,\Up$ to satisfy an extra condition, which exponential normal coordinates may not satisfy.
\label{nc3def1}	

\subsection{\texorpdfstring{Power series decomposition of $\vp,*\vp,g$ on $\nu$}{Power series decomposition of φ,*φ,g on ν}}
\label{nc32}

Use the notation of \S\ref{nc31}. Let $V$ denote the vertical subbundle of $T \nu$. We have $V \cong \pi^* \nu$ as a bundle over $\nu$, where $\pi : \nu \to L$ is the projection. The total space of $\nu$ admits a 1-parameter family of diffeomorphisms given by dilation in the fibres. Explicitly, given $t \in \R$, the dilation map $t : \nu \to \nu$ given by $t (x, \alpha) = (x, t \alpha)$ is a diffeomorphism if $t \neq 0$. The vector field $\de$ on $\nu$ whose flow is this family is called the {\it dilation vector field} on $\nu$. It is a vertical vector field, and with respect to the isomorphism $V \cong \pi^* \nu$ we have $\de\vert_{(x,\al)}\cong\al$ at each $(x,\al)\in\nu$. 

Define smooth maps $\Up_t$ for $t\in\R$ by
\e
\begin{aligned}
\Up_t&:U_{\md{t}^{-1}R}\to M, & \Up_t&:(x,\al)\mapsto \Up(x,t\al), & 0&\ne t\in\R,\\
\Up_0&:\nu\to M, & \Up_0&:(x,\al)\mapsto \Up(x,0\al)=x, & t&=0,	
\end{aligned}
\label{nc3eq5}
\e
where $U_{\md{t}^{-1}R}$ is as in \eq{nc3eq4} with $\md{t}^{-1}R$ in place of $R$, so that $\Up_1=\Up$. Note that $\Up_t = \Up \circ t$ is the composition of $\Up$ with dilation by $t$.

Consider the pullbacks $\Up_t^*(\vp)$, $\Up_t^*(*\vp)$, $\Up_t^*(g)$. These are defined on $\nu$ for $t=0$ and on $U_{\md{t}^{-1}R}$ for $t\ne 0$, where $\lim_{t\ra 0}U_{\md{t}^{-1}R}=\nu$. Since $\Up_t$ depends smoothly on $t\in\R$, so do $\Up_t^*(\vp)$, $\Up_t^*(*\vp)$, $\Up_t^*(g)$. Thus we can consider the Taylor series of $\Up_t^*(\vp)$, $\ab\Up_t^*(*\vp)$, $\ab\Up_t^*(g)$ in $t$ at~$t=0$. 

From Definition \ref{nc3def1}(iii) and $\io^*(\vp)=\vp$ we see that $\Up_{-t}^*(\vp)=\Up_t^*(\vp)$, and similarly for $*\vp,g$, so there are no odd powers of $t$ in the Taylor series. Hence we may write
\ea
\Up_t^*(\vp)&\sim \ts\sum\limits_{n=0}^\iy t^{2n}\vp^{2n},
\label{nc3eq6}\\
\Up_t^*(*\vp)&\sim \ts\sum\limits_{n=0}^\iy t^{2n}\psi^{2n},
\label{nc3eq7}\\
\Up_t^*(g)&\sim \ts\sum\limits_{n=0}^\iy t^{2n}g^{2n}.
\label{nc3eq8}
\ea

Here $\vp^{2n},\psi^{2n},g^{2n}$ are defined and smooth on all of $\nu$, and `$\sim$' in \eq{nc3eq6} means that on any compact subset $S\subseteq\nu$, so that $S\subseteq U_{\md{t}^{-1}R}$ for sufficiently small $t$, we have
\begin{equation*}
\sup\nolimits_S\Big\vert \Up_t^*(\vp)-\ts\sum\limits_{n=0}^k t^{2n}\vp^{2n}\Big\vert=o(t^{2k}) \qquad \text{as $t\ra 0$ for all $k=0,1,\ldots,$}	
\end{equation*}
and similarly for \eq{nc3eq7}--\eq{nc3eq8}.

\begin{rem} In fact as $\vp$ satisfies an elliptic equation, there is a unique real analytic structure on $M$ such that $\vp,*\vp,g$ are real analytic. If we choose $\Up$ also to be real analytic, then on any compact subset $S\subseteq\nu$, the sums in \eq{nc3eq6}--\eq{nc3eq8} converge absolutely for small enough $t$, and the equations hold exactly. But we will not need this.
\label{nc3rem1}
\end{rem}

Note that $\vp^{2n},\psi^{2n},g^{2n}$ depend on the choice of map $\Up$.

Since $\d\vp=0$ and $\d(*\vp)=0$ we see from \eq{nc3eq6}--\eq{nc3eq7} that
\e
\d\vp^{2n} = 0 \text{\, and \,} \d\psi^{2n} = 0 \qquad\text{for all $n=0,1,\ldots.$}
\label{nc3eq9}
\e

Let $s \neq 0$ and consider the dilation $s : \nu \to \nu$. Then $\Up_{ts} = \Up_t \circ s$, so $s^* \Up_t^* = \Up_{st}^*$. It follows that
\begin{equation*}
\sum_{n=0}^{\iy} t^{2n} s^* \vp^{2n} \sim s^* \Up_t^*(\vp) = \Up_{ts}^*(\vp) \sim \sum_{n=0}^{\iy} t^{2n} s^{2n} \vp^{2n}
\end{equation*}
and similarly for $*\vp$ and $g$. Thus $\vp^{2n},\psi^{2n},g^{2n}$ are homogeneous of order $2n$ under dilations in $\nu$, so that
\begin{equation*}
\cL_\de\vp^{2n}=2n\,\vp^{2n},\quad \cL_\de\psi^{2n}=2n\,\psi^{2n},\quad \cL_\de g^{2n}=2n\,g^{2n},
\end{equation*}
where $\cL_\de$ is the Lie derivative. Using Cartan's formula $\cL_v\be=\d(v\cdot\be)+v\cdot(\d\be)$ for a vector field $v$ and exterior form $\be$, and equation \eq{nc3eq9}, we see that
\e
\vp^{2n}=\frac{1}{2n}\d(\de\cdot\vp^{2n}),\quad \psi^{2n}=\frac{1}{2n}\d(\de\cdot\psi^{2n})\quad \text{for $n=1,2,\ldots,$}
\label{nc3eq10}
\e
so the $\vp^{2n}$ and $\psi^{2n}$ are exact for $n>0$.

Now let $\ti\nabla^\nu$ be a connection on $\nu$. This could be the connection $\nabla^\nu$ defined in \S\ref{nc31}, but later we will want the freedom to choose a different connection. On $\nu$, the connection $\ti\nabla^\nu$ induces a splitting of vector bundles
\e
T\nu= V\op H,\quad\text{where $V\cong \pi^*(\nu)$ and $H\cong \pi^*(TL)$.}
\label{nc3eq11}
\e
Here $V$ and $H$ are the vertical and horizontal subbundles of the connection, respectively. The exact sequence
\begin{equation*}
\xymatrix@C=30pt{ 0 \ar[r] & \pi^*(\nu)\cong V \ar[r] & T\nu \ar[r] & \pi^*(TL)\cong H \ar[r] & 0 }
\end{equation*}
is independent of choices, but the embedding $H\subset T\nu$ splitting this exact sequence depends on the choice of $\ti\nabla^\nu$. Dual to \eq{nc3eq11} we have a splitting $T^*\nu=V^*\op H^*$, and hence a splitting
\e
\La^kT^*\nu= \bigop_{i+j=k, \;
0\le i\le 4,\; 0\le j\le 3}\La^iV^*\ot \La^jH^*.
\label{nc3eq12}
\e

Write $\vp^{2n}_{i,j},\psi^{2n}_{i,j}$ for the components of $\vp^{2n},\psi^{2n}$ in $\La^iV^*\ot \La^jH^*$ in the splitting \eq{nc3eq12}. We will call such forms {\it of type\/} $(i,j)$. Hence we have
\e
\vp^{2n}=\vp^{2n}_{0,3}+\vp^{2n}_{1,2}+\vp^{2n}_{2,1}+\vp^{2n}_{3,0},\quad
\psi^{2n}=\psi^{2n}_{1,3}+\psi^{2n}_{2,2}+\psi^{2n}_{3,1}+\psi^{2n}_{4,0}.
\label{nc3eq13}
\e
Note that the components $\vp^{2n}_{i,j},\psi^{2n}_{i,j}$ depend on the choice of connection $\ti\nabla^\nu$, although $\vp^{2n},\psi^{2n}$ do not.

Similarly we decompose $g^{2n}=g^{2n}_{0,2}+g^{2n}_{1,1}+g^{2n}_{2,0}$, with $g^{2n}_{0,2}\in S^2H^*$, $g^{2n}_{1,1}\in V^*\ot H^*$ and~$g^{2n}_{2,0}\in S^2V^*$.

In fact, for $n>0$ we can decompose $\vp^{2n},\psi^{2n}$ further. As the dilation vector field $\de$ lies in the $V$ factor in \eq{nc3eq11}, we see that $\de\cdot\vp^{2n}_{i,j}$ is of type $(i-1,j)$, and similarly for $\de\cdot\psi^{2n}_{i,j}$. Since closed 1-forms on $L$ pull back to closed sections of $H^*$,  we can show that if $\be=\be_{i,j}$ is a form of type $(i,j)$ then $\d\be=(\d\be)_{i+1,j}+(\d\be)_{i,j+1}+(\d\be)_{i-1,j+2}$ is the sum of terms of type $(i+1,j)$ and $(i,j+1)$ and $(i-1,j+2)$ only. That is, by taking the exterior derivative, the number of horizontal components can never decrease, and can never increase by more than two. Define forms $\dot\vp^{2n}_{i,j},\ddot\vp^{2n}_{i,j},\dddot\vp^{2n}_{i,j}$, and $\dot\psi^{2n}_{i,j},\ddot\psi^{2n}_{i,j},\dddot\psi^{2n}_{i,j}$ of type $(i,j)$ for $n>0$ by
\e
\begin{split}
\frac{1}{2n}\d(\de\cdot\vp^{2n}_{i,j})&=\dot\vp^{2n}_{i,j}+\ddot\vp^{2n}_{i-1,j+1}+\dddot\vp^{2n}_{i-2,j+2},\\
\frac{1}{2n}\d(\de\cdot\psi^{2n}_{i,j})&=\dot\psi^{2n}_{i,j}+\ddot\psi^{2n}_{i-1,j+1}+\dddot\psi^{2n}_{i-2,j+2}.
\end{split}
\label{nc3eq14}
\e
Then \eq{nc3eq10} implies that
\e
\vp^{2n}_{i,j}=\dot\vp^{2n}_{i,j}+\ddot\vp^{2n}_{i,j}+\dddot\vp^{2n}_{i,j},\qquad \psi^{2n}_{i,j}=\dot\psi^{2n}_{i,j}+\ddot\psi^{2n}_{i,j}+\dddot\psi^{2n}_{i,j}\quad\text{for $n>0$.}
\label{nc3eq15}
\e
By considering equation \eq{nc3eq14} for various values of $(i,j)$, we derive the following conditions:
\ea
\ddot\vp^{2n}_{3,0} & = \dddot\vp^{2n}_{2,1} = \dddot\vp^{2n}_{3,0} = 0, \quad \ddot\psi^{2n}_{4,0} = \dddot\psi^{2n}_{3,1} = \dddot\psi^{2n}_{4,0} = 0 & & \text{for $n>0$ ($j = -1$ or $-2$)}, 
\nonumber\\
\dot\vp^{2n}_{0,3} & = 0 \qquad \text{for $n>0$ ($i=0$, $j=3$)}.
\label{nc3eq16}
\ea

\subsection{Understanding the leading terms in the power series}
\label{nc33}

On the fibres of the rank 7 vector bundle $TM\vert_L$ over $L$ we have a $\Gt$-structure $(\vp,g)\vert_L$. Pulling back by $\pi:\nu\ra L$ gives a $\Gt$-structure $\pi^*((\vp,g)\vert_L)$ on $\pi^*(TM\vert_L)\ra \nu$. But \eq{nc3eq2} gives an isomorphism
\begin{equation*}
\pi^*(TM\vert_L)\cong \pi^*(\nu)\op\pi^*(TL),
\end{equation*}
and \eq{nc3eq11} an isomorphism 
\begin{equation*}
T\nu= V\op H\cong \pi^*(\nu)\op\pi^*(TL).
\end{equation*}
Combining these gives an isomorphism $T\nu\cong \pi^*(TM\vert_L)$. Let $(\vp^\nu,g^\nu)$ be the $\Gt$-structure on the fibres of $T\nu$ (that is, the $\Gt$-structure on $\nu$) identified with $\pi^*((\vp,g)\vert_L)$ by this isomorphism, and $\psi^\nu$ the corresponding 4-form $*_{g^\nu}(\vp^\nu)$. One can think of $(\vp^{\nu}, g^{\nu})$ as a $\Gt$-structure on $\nu$ that is ``constant'' on the fibres, since at each point $(x, \alpha) \in \nu$ it corresponds to the $\Gt$-structure $\vp(x)$ at $(x, 0) \in \nu$ via the canonical isomorphism $T_{(x, \alpha)} (\nu_x) \cong \nu_x$ between the tangent space of a vector space and the vector space itself.

Thus $(\nu, \vp^\nu, g^\nu)$ is a $\Gt$-manifold, which is generally not torsion-free, as we will see shortly. We will use the metric $g^\nu$ to measure the size of tensors on $\nu$. 

Write $\vp^\nu_{i,j},\psi^\nu_{i,j},g^\nu_{i,j}$ for the components of $\vp^\nu,\psi^\nu,g^\nu$ of type $(i,j)$. As $\vp^\nu,\psi^\nu,\ab g^\nu$ are unchanged under $-1:\nu\ra \nu$, the components with $i$ odd are zero, giving
\begin{equation*}
\vp^\nu=\vp^\nu_{0,3}+\vp^\nu_{2,1},\quad \psi^\nu=\psi^\nu_{2,2}+\psi^\nu_{4,0},\quad
g^\nu=g^\nu_{0,2}+g^\nu_{2,0}.
\end{equation*}
The fact that $g^{\nu}_{1,1} = 0$ says that $H$ is orthogonal to $V$. Since dilations rescale the $V$ factor in \eq{nc3eq11} but fix the $H$ factor, the effect of Lie derivative by the dilation vector field $\de$ is
\e
\cL_\de\vp^\nu_{i,j}=i\,\vp^\nu_{i,j},\quad
\cL_\de\psi^\nu_{i,j}=i\,\psi^\nu_{i,j},\quad
\cL_\de g^\nu_{i,j}=i\,g^\nu_{i,j}.
\label{nc3eq17}
\e

Return now to $\vp^{2n},\psi^{2n},g^{2n}$ defined in \S\ref{nc32}. Observe that $\md{\vp^{2n}_{i,j}}_{g^\nu}$, $\md{\psi^{2n}_{i,j}}_{g^\nu}$, $\md{g^{2n}_{i,j}}_{g^\nu}$ are continuous functions on $\nu$ which are homogeneous of some degree under dilations. Explicitly, since $t^* \vp^{2n}_{i,j} = t^{2n} \vp_{i,j}$, the function $\md{\vp^{2n}_{i,j}}_{g^{\nu}}$ is homogeneous of order $2n - i$ with respect to dilations, because the $g^{\nu}$-norm on vertical $k$-forms is homogeneous of degree $-k$ by \eq{nc3eq17}. Similarly $\md{\psi^{2n}_{i,j}}_{g^\nu}$ and $\md{g^{2n}_{i,j}}_{g^\nu}$ are also homogeneous of degree $2n - i$.

Write $r:\nu\ra[0,\iy)$ for the radius function, where $r(x,\al) = \md{\al}_{g^\nu}$. When restricted to the unit sphere bundle of $\nu$, which is compact, the continuous functions $\md{\vp^{2n}_{i,j}}_{g^\nu}$, $\md{\psi^{2n}_{i,j}}_{g^\nu}$, $\md{g^{2n}_{i,j}}_{g^\nu}$ are bounded. Thus, by their homogeneity, we deduce that
\e
\bmd{\vp^{2n}_{i,j}}_{g^\nu}=O(r^{2n-i}), \quad \bmd{\psi^{2n}_{i,j}}_{g^\nu}=O(r^{2n-i}), \quad \bmd{g^{2n}_{i,j}}_{g^\nu}=O(r^{2n-i}).
\label{nc3eq18}
\e
As the degree of homogeneity cannot be negative for $\vp^{2n}_{i,j},\psi^{2n}_{i,j},g^{2n}_{i,j}\ne 0$, by continuity at $r=0$, we have
\e
\vp^{2n}_{i,j}=0,\quad \psi^{2n}_{i,j}=0,\quad g^{2n}_{i,j}=0\quad\text{if $i>2n$.}
\label{nc3eq19}
\e
An immediate consequence of combining \eq{nc3eq19} and \eq{nc3eq14} is
\e
\begin{aligned}
\dot\vp^{2n}_{i,j} &= \ddot\vp^{2n}_{i-1,j+1} = \dddot\vp^{2n}_{i-2,j+2} = 0, && \text{for $n>0$, ($i> 2n$)},
\\
\dot\psi^{2n}_{i,j} & = \ddot\psi^{2n}_{i-1,j+1} = \dddot\psi^{2n}_{i-2,j+2} = 0, && \text{for $n>0$, ($i> 2n$)}.
\end{aligned}
\label{nc3eq19b}
\e
The terms in \eq{nc3eq18} which are homogeneous of $O(r^0)$ are those which may be nonzero at the zero section $L\subset\nu$. But at the zero section, $\vp^\nu,\psi^\nu,g^\nu$ agree with $\Up^*(\vp),\Up^*(\psi),\Up^*(g)$, that is, with $\Up_1^*(\vp),\Up_1^*(\psi),\Up_1^*(g)$. So \eq{nc3eq6}--\eq{nc3eq8} yield
\begin{equation*}
\vp^\nu_{i,j}=\vp^i_{i,j}, \quad \psi^\nu_{i,j}=\psi^i_{i,j}, \quad g^\nu_{i,j}=g^i_{i,j}, \quad \text{ (no sum over $i$).}
\end{equation*}

If $i>0$ then $\vp^i_{i,j}=\dot\vp^i_{i,j}+\ddot\vp^i_{i,j}+\dddot\vp^i_{i,j}$, but $\ddot\vp^i_{i,j}$, $\dddot\vp^i_{i,j}$ are respectively components of $\frac{1}{i}\d(\de\cdot\vp^i_{i+1,j-1})$ and $\frac{1}{i}\d(\de\cdot\vp^i_{i+2,j-2})$, which are zero by \eq{nc3eq19}. So $\ddot\vp^i_{i,j}=\dddot\vp^i_{i,j}=0$, and similarly $\ddot\psi^i_{i,j}=\dddot\psi^i_{i,j}=0$. Hence we have
\e
\vp^\nu=\vp^0_{0,3}+\dot\vp^2_{2,1},\qquad \psi^\nu=\dot\psi^2_{2,2}+\dot\psi^4_{4,0}.
\label{nc3eq20}
\e
Here $\vp^0_{0,3}=\pi^*(\vol_L)$ is the pullback to $\nu$ of the volume form $\vol_L$ on $L$, and is a closed 3-form on~$\nu$. 

We also find it useful to define $\vp^\nu_t,\psi^\nu_t,g^\nu_t$ on $\nu$ to be the pullbacks of $\vp^\nu,\psi^\nu,g^\nu$ under $t : \nu\ra\nu$ for $t>0$. This scales each component by the powers of $t$ in \eq{nc3eq6}--\eq{nc3eq8}, so that by \eq{nc3eq20} we have
\e
\vp^\nu_t=\vp^0_{0,3}+t^2\dot\vp^2_{2,1},\quad \psi_t^\nu=t^2\dot\psi^2_{2,2}+t^4\dot\psi^4_{4,0},\quad
g_t^\nu=g^0_{0,2}+t^2g^2_{2,0}.
\label{nc3eq21}
\e
Then $(\vp^\nu_t,g^\nu_t)$ is a $\Gt$-structure on $\nu$, with 4-form $\psi_t^\nu=*_{g^\nu_t}(\vp^\nu_t) = \Theta(\vp^\nu_t)$, and $\vp^\nu_t,\psi^\nu_t,g^\nu_t$ are the leading order approximations to $\Up_t^*(\vp),\Up_t^*(*\vp),\Up_t^*(g)$ near the zero section $L\subset\nu$, just as $\vp^\nu,\psi^\nu,g^\nu$ approximate $\Up^*(\vp),\ab\Up^*(*\vp),\ab\Up^*(g)$. If we estimate the size of tensors using $g^\nu_t$ rather than $g^\nu$, the answers differ by powers of $t$. For example, if $\be_{i,j}$ is a form on $\nu$ of type $(i,j)$ then
\e
\bmd{\be_{i,j}}_{g^\nu_t}=t^{-i} \, \bmd{\be_{i,j}}_{g^\nu}.
\label{nc3eq22}
\e

We may summarize all the work above, also using \eq{nc3eq19b}, with the following pair of equations:
\ea
& \Up_t^*(\vp) \sim \nonumber \\  
&\begin{aligned}
& & \vp^0_{0,3}{} + {} & t^2\dot\vp^2_{2,1} &\!\!\!\!\!\!\!\!\!\!\!\!\!\!\!\!\!\!\! \bigr] \; \, \vp^\nu_t = & O(1) \\
& & {}+{} & t^2\ddot\vp^2_{1,2} + t^2\dot\vp^2_{1,2} + t^4\dot\vp^4_{3,0} & \bigr] \; \; & O(tr) \\
& & {}+{} & t^2\dddot\vp^2_{0,3} + t^2\ddot\vp^2_{0,3} + t^4\ddot\vp^4_{2,1} + t^2\dot\vp^2_{0,3} + t^4\dot\vp^4_{2,1} & \bigr] \; \; & O(t^2r^2) \\
& & & {}\phantom{t^2\dddot\vp^2_{0,3} + t^2\ddot\vp^2_{0,3}} + t^4\dddot\vp^4_{1,2}\phantom{+t^2\dot\vp^2_{0,3}\,\,\,} + t^4\ddot\vp^4_{1,2} + t^4\dot\vp^4_{1,2} + t^6\dot\vp^6_{3,0} & \bigr] \; \; & O(t^3r^3)
\end{aligned}
\nonumber\\
& \qquad \qquad \, \, \, {}+{} \cdots,
\label{nc3eq23}\\
& \Up_t^*(*\vp) \sim \nonumber \\
&\begin{aligned}
& t^2\dot\psi^2_{2,2} + t^4\dot\psi^4_{4,0} + & \!\!\!\!\!\!\!\!\!\!\!\!\!\!\!\!\!\!\! \bigr] \; \; \psi^\nu_t = & O(1) \\
& t^2\ddot\psi^2_{1,3} + t^4\ddot\psi^4_{3,1} + t^2\dot\psi^2_{1,3} + t^4\dot\psi^4_{3,1} & \bigr] \; \; & O(tr) \\
& {}\phantom{t^2\ddot\psi^2_{1,3}} + t^4\dddot\psi^4_{2,2} \phantom{+t^2\dot\psi^2_{1,3}\,\,\,} + t^4\ddot\psi^4_{2,2} + t^4\dot\psi^4_{2,2} + t^6\dot\psi^6_{4,0} & \bigr] \; \; & O(t^2r^2) \\
& {}\phantom{t^2\ddot\psi^2_{1,3} + t^4\dddot\psi^4_{2,2} + t^2\dot\psi^2_{1,3}\kern.07em} + t^4\dddot\psi^4_{1,3} + t^4\ddot\psi^4_{1,3} + t^6\ddot\psi^6_{3,1} + t^4\dot\psi^4_{1,3} + t^6\dot\psi^6_{3,1} & \bigr] \; \; & O(t^3r^3) \\
\end{aligned}
\nonumber\\
& \qquad \qquad \, \, \, {}+{} \cdots,
\label{nc3eq24}
\ea
where the $O(\cdots)$ are measured using $g^\nu_t$. Explicitly, using \eq{nc3eq18} and \eq{nc3eq22} we have
\begin{equation*}
\md{t^a \al^{2n}_{i,j}}_{g^{\nu}_t} = t^a \md{\al^{2n}_{i,j}}_{g^{\nu}_t} = t^{a - i} \md{\al^{2n}_{i,j}}_{g^{\nu}} = t^{a-i} O(r^{2n - i}),
\end{equation*}
so the terms aligned horizontally in \eq{nc3eq23} and \eq{nc3eq24} have the same order $O(t^kr^k)$ when measured using $\md{\,.\,}_{g^\nu_t}$. The terms aligned vertically sum to a closed form, by \eq{nc3eq14}.

\subsection{\texorpdfstring{Choosing $\Up$ and $\ti\nabla^\nu$ to eliminate the $O(tr)$ error terms}{Choosing Υ,∇ν to eliminate the O(r) error terms}}
\label{nc34}

Taking $t=1$, it will be important later to eliminate the $O(tr)$ terms in \eq{nc3eq23}--\eq{nc3eq24}, as they would cause too large an error to apply Theorem \ref{nc2thm3}.

Recall that $\vp^{2n},\psi^{2n}$ depend on the choice of the map $\Up$, and the further decomposition into types $\vp^{2n}_{i,j}$ and $\psi^{2n}_{i,j}$ depend on the choice of connection $\tilde \nabla^{\nu}$ on $\nu$. In this section we will modify $\Up$ and $\tilde \nabla^{\nu}$ to eliminate the $O(tr)$ terms in~\eq{nc3eq23}--\eq{nc3eq24}.

In this section we will need several times to make use of a one-to-one correspondence between tensors on the total space $\nu$ that are homogeneous of a certain degree and sections of certain bundles over $L$. This correspondence is more general but we will only describe two particular cases we will need. Recall that we have canonical isomorphisms $V \cong \pi^* (\nu)$ and $H^* \cong \pi^*( T^* L)$ and, given a connection $\tilde \nabla^{\nu}$ on $\nu$, we also have isomorphisms $H \cong \pi^* (TL)$ and $V^* \cong \pi^* (\nu^*)$. Let $S^{k} (\nu^*)$ be the $k^{\text{th}}$ symmetric power of $\nu^*$, so sections of $S^{k} (\nu^*)$ are functions on the total space $\nu$ that are homogeneous of degree $k$ under dilations.
\smallskip

\noindent{\bf Case 1:} Let $\bar \al$ be a section of the bundle $S^{2n - i} (\nu^*) \otimes \Lambda^i (\nu^*) \otimes TL$. Then under these identifications $\bar \al$ corresponds to a section $v_{\bar \al}$ of $\Lambda^i (V^*) \otimes H$ which is homogeneous of degree $2n - i$. This correspondence can be seen explicitly as follows. Let $v$ be a vector field on $L$ and let $f_1, \ldots, f_i$ be sections of $\nu^*$. Then
\begin{equation*}
\bar \al= h \otimes (f_{1} \w \cdots \w f_{i}) \otimes v \longleftrightarrow v_{\bar \al}= h \, (\pi^* f_{1}) \w \cdots \w (\pi^* f_{i}) \otimes \pi^* (v)
\end{equation*}
for $h \in S^{2n - i}(\nu^*)$.
\smallskip

\noindent{\bf Case 2:} Let $\bar \al^{2n}_{i,j}$ be a section of the bundle $S^{2n - i} (\nu^*) \otimes \Lambda^i (\nu^*) \otimes \Lambda^j (T^* L)$. Then under these identifications $\bar \al^{2n}_{i,j}$ corresponds to a section of $\Lambda^i (V^*) \otimes \Lambda^j (H^*)$, that is an $(i,j)$ form on $\nu$, which is homogeneous of degree $2n - i$. Explicitly, if $e_1, \ldots, e_j$ are $1$-forms on $L$ and $f_1, \ldots, f_i$ are sections of $\nu^*$, then
\begin{equation*}
\begin{matrix}
\bar \al^{2n}_{i,j} = & h \otimes (f_{1} \w \cdots \w f_{i}) \otimes (e_{1} \w \cdots \w e_{j}) \hspace{0.7in} \\[0.4em]
\updownarrow & \\[0.4em]
\al^{2n}_{i,j} = & h \, (\pi^* f_{1}) \w \cdots \w (\pi^* f_{i}) \otimes (\pi^* e_1) \w \cdots \w (\pi^* e_j)
\end{matrix}
\end{equation*}
for $h \in S^{2n - i}(\nu^*)$. Let $\d^v = \d^{(1,0)}$ be the vertical derivative, the part of $\d$ that takes an $(i, j)$ form to an $(i+1, j)$ form. In this case we also need to understand the correspondence between $\d^v$ on the right hand side and a canonical operation on the left hand side. The easiest way to see this is in local coordinates. Let $x_1, \ldots, x_3$ be local coordinates on $L$ and let $y_1, \ldots, y_4$ be fibre coordinates for $\nu^*$. Locally, a decomposable $(i,j)$ form that is homogeneous of degree $2n-i$ will be of the form $h \, \d y_{a_1} \w \cdots \d y_{a_i} \w \d x_{b_1} \cdots \w \d x_{b_j}$ where $h(x,y)$ is homogeneous of degree $2n-i$ in $y_1, \ldots, y_4$. Noting that $\d^v x_b = 0$ and $\d^v y_a = \d y_a$, it is easy to see that $\d^v$ corresponds under the identification to the {\it partial antisymmetrization} map
\begin{equation*}
\w : S^{2n - i}(\nu^*) \otimes \Lambda^i (\nu^*) \otimes \Lambda^j (T^* L) \ra S^{2n - i - 1}(\nu^*) \otimes \Lambda^{i+1} (\nu^*) \otimes \Lambda^j (T^* L).
\end{equation*}
We write this map explicitly in a local frame when $2n - i = 1$ or $2n-i = 2$. Let $\be$ be a section of $\Lambda^i (\nu^*)$ and $\ga$ a $j$-form on $L$. Then we have
\begin{align*}
f_a \otimes \be \otimes \ga & \mapsto (f_a \w \be) \otimes \ga, & & 2n - i = 1, \\
(f_a f_b) \otimes \be \otimes \ga & \mapsto [f_a \otimes (f_b \w \be) + f_b \otimes (f_a \w \be)] \otimes \ga, & & 2n - i = 2.
\end{align*}

With these preparations out of the way, we begin with the following lemma.

\begin{lem} In the situation above, we have
\e
(D_{\vp^\nu}\Th) (\ddot\vp^2_{1,2}+\dot\vp^2_{1,2}+\dot\vp^4_{3,0})=\ddot\psi^2_{1,3}+\ddot\psi^4_{3,1}+\dot\psi^2_{1,3}+\dot\psi^4_{3,1},
\label{nc3eq25}
\e
where\/ $D_{\vp^\nu}\Th$ is as in equation \eq{nc2eq6} in\/~{\rm\S\ref{nc22}}.
\label{nc3lem1}
\end{lem}

\begin{proof} We have $\Th(\vp^\nu)=\psi^\nu$ and $\Th(\Up^*(\vp))=\Up^*(*\vp)$, since $\Th(\vp)=*\vp$, and $\Th$ is an intrinsic object on any 7-manifold, and so commutes with pullback by the local diffeomorphism $\Up$. Equation \eq{nc3eq25} now follows from \eq{nc3eq23}--\eq{nc3eq24} for $t=1$ and small $r$, since $D_{\vp^\nu}\Th$ must map the leading term in $\Up^*(\vp)-\vp^\nu$ to the leading term in~$\Th(\Up^*(\vp))-\Th(\vp^\nu)$.	
\end{proof}

In \S\ref{nc31} we chose $\Up:U_R\ra M$. If we expand $\Up$ in a power series around $L\subset U_R$ in terms of order $O(r^k)$ for $k=0,1,\ldots,$ then the $O(r^0)$ and $O(r^1)$ terms are determined by Definition \ref{nc3def1}(ii),(iv), but the $O(r^2)$ and higher terms are essentially arbitrary, apart from the $\Z_2$-equivariance in Definition \ref{nc3def1}(iii).

The next lemma shows that by choosing the $O(r^2)$ terms in the expansion of $\Up$ correctly, we can ensure the vanishing of the fourth column in \eq{nc3eq23}. Note that the $O(r^2)$ terms in the expansion of $\Up$ determine the $O(r)$ terms in the expansion of $\Up^* (\vp)$, since the pullback $\Up^*$ involves one derivative of $\Up$.

\begin{lem} We can choose the map $\Up:U_R\ra M$ in {\rm\S\ref{nc31}} so that\/~$\dot\vp^4_{3,0} = \ddot\vp^4_{2,1} = \dddot\vp^4_{1,2} = 0$.
\label{nc3lem2}	
\end{lem}

\begin{proof} Let $R'$ and $\Up':U_{R'}\ra M$ be some fixed choice of $R,\Up$ in \S\ref{nc31}, and $\dot\vp^{\prime 4}_{3,0}$ the corresponding value for $\dot\vp^4_{3,0}$. Let $\bar\al$ be a smooth section of the rank 30 vector bundle $S^2\nu^*\ot TL\ra L$, to be chosen later. Under the identification ({\it Case 1}) described above with $n = 1$ and $i = 0$, the section $\bar\al$ corresponds to a vector field $v_{\bar\al}$ in $H\subset T\nu$ on $\nu$ which is homogeneous quadratic in the fibre directions. That is, $v_{\bar\al}$ is a horizontal vector field on $\nu$ such that the function $\md{v_{\bar\al}}_{g^{\nu}}$ is homogeneous of degree 2 under dilations.

For $0<R\le R'$ small, define $\Up:U_R\ra M$ by $\Up=\Up'\ci\exp(v_{\bar\al})$, where $\exp(v_{\bar\al}):U_R\ra\nu$ is the flow of the vector field $v_{\bar\al}$, and it maps $U_R$ diffeomorphically to an open subset of $U_{R'}$ provided $R$ is sufficiently small. As $v_{\bar\al}=O(r^2)$, we see that $\Up$ and $\Up'$ agree up to $O(r^2)$ near $L\subset U_R$, so Definition \ref{nc3def1}(ii),(iv) for $\Up'$ imply (ii),(iv) for $\Up$. Also $v_{\bar\al}$ is invariant under $-1:U_R\ra U_R$, so Definition \ref{nc3def1}(iii) for $\Up'$ implies (iii) for $\Up$. Hence $R,\Up$ satisfy Definition \ref{nc3def1}.

We now have a 3-form $\dot\vp^4_{3,0}$ on $\nu$ from the new $\Up$, so we can compare $\dot\vp^4_{3,0}$ and $\dot\vp^{\prime 4}_{3,0}$. They are precisely the $(3,0)$ components that are homogeneous of degree 1 of $\Up^* (\vp)$ and $\Up^{\prime*} (\vp)$, respectively. Because $\Up^*(\vp)=\exp(v_{\bar\al})^*(\Up^{\prime*}(\vp))$ with $v_{\bar\al}=O(r^2)$, we find that
\begin{align*}
\Up^*(\vp) & =\Up^{\prime*}(\vp) + \cL_{v_{\bar\al}}(\Up^{\prime*}(\vp)) + O(r^2) \\
& = \Up^{\prime*}(\vp)+ \d [v_{\bar\al} \cdot \Up^{\prime*}(\vp)] + O(r^2)
\end{align*}
since $\Up^{\prime*} (\vp)$ is closed. We take the type (3,0) components of both sides of the above equation and keep those terms that are homogeneous of degree 1. From \eq{nc3eq23} with $t=1$ the only component of $\Up^{\prime*} (\vp)$ for which $\d [v_{\bar\al} \cdot \Up^{\prime*}(\vp)]$ will have a $(3,0)$ component is $\dot\vp^2_{2,1}$, and indeed $[\d(v_{\bar\al}\cdot\dot\vp^2_{2,1})]_{3,0}$ is homogeneous of degree 1. Thus, we have
\e
\dot\vp^4_{3,0} = \dot\vp^{\prime 4}_{3,0} + [\d(v_{\bar\al}\cdot\dot\vp^2_{2,1})]_{3,0} = \dot\vp^{\prime 4}_{3,0} + \d^v (v_{\bar\al}\cdot\dot\vp^2_{2,1}).
\label{nc3eq26}
\e

Because $\dot\vp^4_{3,0},\dot\vp^{\prime 4}_{3,0}$ are sections of $\La^3V^*\ra\nu$ which are linear in the fibre directions, by the identification ({\it Case 2}) above we can identify $\dot\vp^4_{3,0},\dot\vp^{\prime 4}_{3,0}$ with sections $\bar{\dot\vp}{}^4_{3,0},\bar{\dot\vp}{}^{\prime 4}_{3,0}$ of the rank 16 vector bundle $\nu^*\ot\La^3\nu^*\ra L$. These sections $\bar{\dot\vp}{}^4_{3,0},\bar{\dot\vp}{}^{\prime 4}_{3,0}$ are not arbitrary, however. The fourth column of \eq{nc3eq23} is closed, and the only $(4,0)$ component of its exterior derivative arises from the vertical (fibre) derivative of the (3,0) term. That is, $\dot\vp^4_{3,0}$ and $\dot\vp^{\prime 4}_{3,0}$ are closed in the fibre directions. As explained above, under the identification this corresponds to $\bar{\dot\vp}{}^4_{3,0}$ and $\bar{\dot\vp}{}^{\prime 4}_{3,0}$ being in the kernel of $\w:\nu^*\ot\La^3\nu^*\ra\La^4\nu^*$, which is a rank 15 vector bundle on~$L$.

Translating via the identifications, equation \eq{nc3eq26} becomes
\begin{equation*}
\bar{\dot\vp}{}^4_{3,0}=\bar{\dot\vp}{}^{\prime 4}_{3,0}+\bar\al\cdot\bar{\dot\vp}{}^2_{2,1},
\end{equation*}
where $\bar{\dot\vp}{}^2_{2,1}$ is a section of $\La^2\nu^*\ot T^*L\ra L$, and `$\,\cdot\,$' is a natural bilinear product
\begin{equation*}
\bigl(S^2\nu^*\ot TL\bigr)\t \bigl(\La^2\nu^*\ot T^*L\bigr)\longra \nu^*\ot\La^3\nu^*,\end{equation*}
defined by combining the dual pairing of $TL,T^*L$ with the partial antisymmetrization map $S^2\nu^*\ot \La^2\nu^*\ra \nu^*\ot\La^3\nu^*$. A calculation shows that the map $\bar \al \mapsto \bar\al \cdot \bar{\dot\vp}{}^2_{2,1}$ is surjective onto the kernel of $\w:\nu^*\ot\La^3\nu^*\ra\La^4\nu^*$. (The details are given in Proposition \ref{ncAprop1} of Appendix \ref{ncA}.) Therefore we can choose $\bar\al$ such that $\bar\al\cdot\bar{\dot\vp}{}^2_{2,1}=-\bar{\dot\vp}{}^{\prime 4}_{3,0}$, so that $\bar{\dot\vp}{}^4_{3,0}=0$, and hence $\dot\vp^4_{3,0}=0$, as desired. Finally, since $\ddot\vp^4_{3,0}=\dddot\vp^4_{3,0}=0$ by \eq{nc3eq16} we have $\vp^4_{3,0}=0$ by \eq{nc3eq15}. Hence by \eq{nc3eq14} we deduce that $\ddot\vp^4_{2,1}=\dddot\vp^4_{1,2}=0$ as well.
\end{proof}

\begin{rem} Note that in the proof above $\bar\al$ lives in a rank 30 vector bundle, but $\bar{\dot\vp}{}^4_{3,0}$ lives in a rank 15 vector bundle, so we have used only half of the freedom in~$\bar\al$.
\label{nc3rem2}
\end{rem}

From now on, fix the map $\Up:U_R\ra M$ satisfying Lemma \ref{nc3lem2}. This fixes the decomposition of $\Up^*(\vp)$ into components $\vp^{2n}$ in \eq{nc3eq6}. In \S\ref{nc32} we chose a connection $\ti\nabla^\nu$ on $\nu\ra L$, which was used to define the splitting $T\nu=V\op H$, and hence the decompositions of $\vp^{2n}$ into components $\vp^{2n}_{i,j}$ and $\dot\vp^{2n}_{i,j},\ddot\vp^{2n}_{i,j},\dddot\vp^{2n}_{i,j}$ in \eq{nc3eq13} and \eq{nc3eq15}. The next lemma explores the effect of changing the choice of $\ti\nabla^\nu$. Note that changing $\ti\nabla^\nu$ {\it does not change the fact that\/  $\dot\vp^4_{3,0}=\ddot\vp^4_{2,1}=\dddot\vp^4_{1,2}=0$ in Lemma\/} \ref{nc3lem2}, since we see from \eq{nc3eq23} with $t=1$ that $\dot\vp^4_{3,0}=0$ if and only if $\md{\vp^4}_{g^\nu}=O(r^2)$, where this second condition is independent of $\ti\nabla^\nu$, and $\dot\vp^4_{3,0}=0$ implies $\dot\vp^4_{3,0}=\ddot\vp^4_{2,1}=\dddot\vp^4_{1,2}=0$ as above.

\begin{lem} In\/ {\rm\S\ref{nc32}} there is a unique choice of the connection $\ti\nabla^\nu$ on $\nu$ used to define the splitting $T\nu=V\op H,$ such that\/ $\ddot\vp^2_{1,2}=\dot\vp^2_{1,2}=\ddot\vp^2_{0,3}=0$.
\label{nc3lem3}	
\end{lem}

\begin{proof} Observe that $\ddot\vp^2_{1,2}+\dot\vp^2_{1,2}$ is a $(1,2)$ form on $\nu$ which is linear in the fibre directions. By the identification ({\it Case 2}) above we can identify $\ddot\vp^2_{1,2}+\dot\vp^2_{1,2}$ with a section $\bar{\ddot\vp}{}^2_{1,2}+\bar{\dot\vp}{}^2_{1,2}$ of $\nu^*\ot\nu^*\ot\La^2T^*L$ over $L$, where the first factor of $\nu^*$ is linear functions in the fibres of $\nu\ra L$, and the second factor is 1-forms on the fibres of $\nu\ra L$.

Suppose $\ti\nabla^\nu$ and $\ti\nabla^{\prime\nu}$ are two choices of $\ti\nabla^\nu$, yielding $\ddot\vp^2_{1,2}+\dot\vp^2_{1,2}$ and $\ddot\vp^{\prime 2}_{1,2}+\dot\vp^{\prime 2}_{1,2}$ above, and corresponding sections $\bar{\ddot\vp}{}^2_{1,2}+\bar{\dot\vp}{}^2_{1,2}$ and $\bar{\ddot\vp}{}^{\prime 2}_{1,2}+\bar{\dot\vp}{}^{\prime 2}_{1,2}$ of $\nu^*\ot\nu^*\ot\La^2T^*L$. Then we have $\ti\nabla^\nu=\ti\nabla^{\prime\nu}+\Ga$, for $\Ga$ a smooth section of $\nu\ot\nu^*\ot T^*L$ over $L$. In terms of local frames, if $e_1, \ldots, e_3$ is a local orthonormal frame for $T^* L$ and $f_1, \ldots, f_4$ is a local frame for $\nu^*$, then $\pi^* e_1, \ldots, \pi^* e_3$ is a local frame for $H^* = H^{\prime *}$, while the two connections $\ti \nabla^{\nu}$ and $\ti \nabla^{\prime \nu}$ induce different vertical frames, since $V^* \neq V^{\prime *}$. We denote these frames by $\pi^* f_1, \ldots, \pi^* f_4$ and $\pi^{\prime *} f_1, \ldots, \pi^{\prime *} f_4$, respectively, which is an abuse of notation as they are not actually pullbacks. The relation between these frames is
\begin{equation*}
{(\pi^{\prime *} f^p)} |_{(x, \al)} = {(\pi^{*} f^p)} |_{(x, \al)} + \sum_{k=1}^3 {\Ga^p_{qk}}(x) \al^q e_k
\end{equation*}
at the point $(x , \al^p f_p) \in \nu$. From \eq{nc3eq23}, we see that $\ddot\vp^2_{1,2}+\dot\vp^2_{1,2}$ is precisely the $(1,2)$ part of $\vp^2$, and $\dot \vp^2_{2,1}$ is the $(2,1)$ part of $\vp^2$. We expand $\vp^2 = \vp^2_{2,1} + \vp^2_{1,2} + \vp^2_{0,3} = \vp^{\prime 2}_{2,1} + \vp^{\prime 2}_{1,2} + \vp^{\prime 2}_{0,3}$ in terms of these frames, and apply the identification between forms on $\nu$ and sections of bundles over $L$. A short calculation gives
\e
\bar{\ddot\vp}{}^2_{1,2}+\bar{\dot\vp}{}^2_{1,2}=\bar{\ddot\vp}{}^{\prime 2}_{1,2}+\bar{\dot\vp}{}^{\prime 2}_{1,2}+\Ga\cdot \bar{\dot\vp}^2_{2,1},
\label{nc3eq27}
\e
where $\bar{\dot\vp}^2_{2,1}$ is the section of $\La^2\nu^*\ot T^*L$ corresponding to $\dot\vp^2_{2,1}$ under the identification of $\La^2V^*\ot H^*$ with $\La^2\nu^*\ot T^*L$. Here `$\,\cdot\,$' is the canonical bilinear pairing
\begin{equation*}
(\nu \otimes \nu^* \otimes T^* L) \times (\Lambda^2 \nu^* \otimes T^* L) \to \nu^* \otimes \nu^* \otimes \Lambda^2 T^* L
\end{equation*}
given by pairing the $\nu$ factor with the $\Lambda^2 \nu^*$ factor, and wedging the two $T^* L$ factors together.

A calculation shows that the map $\Ga \mapsto \Ga \cdot \bar{\dot\vp}{}^2_{2,1}$ is an isomorphism of vector bundles on $L$. (The details are given in Proposition \ref{ncAprop2} of Appendix \ref{ncA}.) Hence if we fix a reference connection $\ti\nabla^{\prime\nu}$ on $\nu$, giving some $\bar{\ddot\vp}{}^{\prime 2}_{1,2}+\bar{\dot\vp}{}^{\prime 2}_{1,2}$, there is a unique $\Ga$ such that \eq{nc3eq27} gives $\bar{\ddot\vp}{}^2_{1,2}+\bar{\dot\vp}{}^2_{1,2}=0$, and then $\ti\nabla^\nu=\ti\nabla^{\prime\nu}+\Ga$ is the unique choice of connection on $\nu$ with $\ddot\vp^2_{1,2}+\dot\vp^2_{1,2}=0$. 

By \eq{nc3eq19}, we have $\vp^2_{3,0} = 0$. Thus from \eq{nc3eq14} we get $\dddot\vp^2_{1,2} = 0$. Therefore by \eq{nc3eq15} we find that $\vp^2_{1,2}=\ddot\vp^2_{1,2}+\dot\vp^2_{1,2}$. Hence we conclude that in fact $\vp^2_{1,2}=0$, which by \eq{nc3eq14} again implies that $\dot\vp^2_{1,2}=\ddot\vp^2_{0,3}=0$. Thus $\ddot\vp^2_{1,2}+\dot\vp^2_{1,2}=0$ yields $\ddot\vp^2_{1,2}=0$, and the proof is complete.
\end{proof}

We now combine Lemmas \ref{nc3lem1}, \ref{nc3lem2}, and \ref{nc3lem3}:

\begin{prop} We can choose the data $\Up:U_R\ra M$ and\/ $\ti\nabla^\nu$ in {\rm\S\ref{nc31}--\S\ref{nc32}} such that
\e
\begin{split}
&\ddot\vp^2_{1,2}=\dot\vp^2_{1,2}=\dot\vp^4_{3,0}=\ddot\vp^2_{0,3}=\ddot\vp^4_{2,1}=\dddot\vp^4_{1,2}=0,\\
&\ddot\psi^2_{1,3}=\ddot\psi^4_{3,1}=\dot\psi^2_{1,3}=\dot\psi^4_{3,1}=\ddot\psi^4_{2,2}=\dddot\psi^4_{1,3}=0,\\
& (D_{\vp^\nu} \Th) (\dddot\vp^2_{0,3}+\dot\vp^2_{0,3}+\dot\vp^4_{2,1})=\dddot\psi^4_{2,2}+\dot\psi^4_{2,2}+\dot\psi^6_{4,0}.
\end{split}
\label{nc3eq28}
\e
\label{nc3prop1}	
\end{prop}

\begin{proof} Lemmas \ref{nc3lem2} and \ref{nc3lem3} prove that we may choose $\Up,\ti\nabla^\nu$ such that the first line of \eq{nc3eq28} holds. Hence Lemma \ref{nc3lem1} shows that
\e 
\ddot\psi^2_{1,3}+\ddot\psi^4_{3,1}+\dot\psi^2_{1,3}+\dot\psi^4_{3,1}=0.
\label{nc3eq29}
\e
Since $\psi^2_{3,1} = 0$ by \eq{nc3eq19}, we get from \eq{nc3eq14} that $\dddot\psi^2_{1,3}=0$. We also have $\dddot\psi^4_{3,1}=0$ by \eq{nc3eq16}. Thus, taking the $(1,3)$ and $(3,1)$ components of \eq{nc3eq29} and using \eq{nc3eq15} we find
\e 
\psi^2_{1,3}=\ddot\psi^2_{1,3}+\dot\psi^2_{1,3}=0,\qquad \psi^4_{3,1}=\ddot\psi^4_{3,1}+\dot\psi^4_{3,1}=0.
\label{nc3eq30}
\e
Applying \eq{nc3eq14} to $\psi^2_{1,3} = 0$ and $\psi^4_{3,1} = 0$ gives $\dot\psi^2_{1,3}=0$ and $\dot\psi^4_{3,1}=\ddot\psi^4_{2,2}=\dddot\psi^4_{1,3}=0$, and feeding these back into \eq{nc3eq30} then gives $\ddot\psi^2_{1,3} = 0$ and $\ddot\psi^4_{3,1} = 0$. We have thus proved the second line of \eq{nc3eq28}.

 Now in \eq{nc3eq23}--\eq{nc3eq24} with $t=1$, all the $O(r)$ terms are zero, so the leading error terms are the $O(r^2)$ terms. Some of these $O(r^2)$ terms vanish by the first two lines of \eq{nc3eq28}, and the rest are on the left and right hand sides of the third line of \eq{nc3eq28}. Therefore the argument of Lemma \ref{nc3lem1} applied to the leading $O(r^2)$ errors proves the third line.
\end{proof}

From now on we fix the data $\Up,\ti\nabla^\nu$ as in Proposition \ref{nc3prop1}. The proposition, combined with equations \eq{nc3eq20}, \eq{nc3eq21}, and \eq{nc3eq23}--\eq{nc3eq24} implies:

\begin{cor} In the situation above, $\ti\vp^\nu=\vp^\nu+\dddot\vp^2_{0,3}$ and\/ $\ti\psi^\nu=\psi^\nu+\dddot\psi^4_{2,2}$ are closed forms on $\nu,$ where $\md{\dddot\vp^2_{0,3}}_{g^\nu}=O(r^2)$ and\/ $\md{\dddot\psi^4_{2,2}}_{g^\nu}=O(r^2)$.

Similarly, $\ti\vp^\nu_t=\vp^\nu_t+t^2\dddot\vp^2_{0,3}$ and\/ $\ti\psi^\nu_t=\psi^\nu_t+t^4\dddot\psi^4_{2,2}$ are closed forms on $\nu,$ where $\md{t^2\dddot\vp^2_{0,3}}_{g^\nu_t}=O(t^2r^2)$ and\/ $\md{t^4\dddot\psi^4_{2,2}}_{g^\nu_t}=O(t^2r^2)$.
\label{nc3cor}
\end{cor}

Here is the point of all this. We have constructed a $\Gt$-structure $(\vp^\nu_t,g^\nu_t)$ on $\nu$ with 4-form $\psi^\nu_t = \Theta(\vp^{\nu}_t) =*_{g^\nu_t}\vp^\nu_t$. To apply Theorem \ref{nc2thm3} in \S\ref{nc6}, we will need $(\vp^\nu_t,g^\nu_t)$ to have {\it small torsion}, in an appropriate sense. If $\d\vp^\nu_t=0$ and $\d\psi^\nu_t=0$ then $(\vp^\nu_t,g^\nu_t)$ would be torsion-free, by Theorem \ref{nc2thm1}. In general $\d\vp^\nu_t\ne 0$ and $\d\psi^\nu_t\neq 0$, but as $\d\ti\vp^\nu_t=0$ and $\d\ti\psi^\nu_t=0$ with $\ti\vp^\nu_t=\vp^\nu_t+t^2\dddot\vp^2_{0,3}$ and $\ti\psi^\nu_t=\psi^\nu_t+t^4\dddot\psi^4_{2,2}$, we can regard $\md{t^2\dddot\vp^2_{0,3}}_{g^\nu_t}$ and $\md{t^4\dddot\psi^4_{2,2}}_{g^\nu_t}$ as measuring the torsion of $(\vp^\nu_t,g^\nu_t)$, as if $t^2\dddot\vp^2_{0,3}=t^4\dddot\psi^4_{2,2}=0$ then $(\vp^\nu_t,g^\nu_t)$ would be torsion-free. Thus, Corollary \ref{nc3cor} basically says that the torsion of $(\vp^\nu_t,g^\nu_t)$ is~$O(t^2r^2)$.

\begin{rem} It would be interesting to determine an invariant geometric interpretation for the choices of $\Up$ in Lemma~\ref{nc3lem2} and of $\ti\nabla^\nu$ in Lemma~\ref{nc3lem3}.
\end{rem}

\subsection{\texorpdfstring{Comparing $\Up^*(\vp),\Up^*(*\vp)$ and $\ti\vp^\nu,\ti\psi^\nu$}{Comparing Υ*(φ),Υ*(*φ) and ~φᵗᵛ,~ψᵗᵛ}}
\label{nc35}

We now construct forms $\eta,\ze$ on $U_R$ which we will use in \S\ref{nc6} to interpolate between $\Up^*(\vp),\Up^*(*\vp)$ and $\ti\vp^\nu,\ti\psi^\nu$, respectively. Define a 2-form $\eta$ and a 3-form $\ze$ on $U_R$ by
\e
\begin{split}
\eta&=\int_0^1t^{-1}\bigl(\de\cdot\Up_t^*(\vp)\bigr)\d t\ts-\frac{1}{2}\de\cdot\vp^2_{2,1},\\
\ze&=\int_0^1t^{-1}\bigl(\de\cdot\Up_t^*(*\vp)\bigr)\d t-\ts\frac{1}{2}\de\cdot\psi^2_{2,2} -\frac{1}{4}\de\cdot\psi^4_{4,0}.
\end{split}
\label{nc3eq31}
\e
Here the integrals are well defined because $\de\cdot\Up_0^*(\vp)=0$ and $\de\cdot\Up_0^*(*\vp)=0$, so
$\de\cdot\Up_t^*(\vp)$ and $\de\cdot\Up_t^*(*\vp)$ are both $O(t)$.

Let $F_s : \nu \to \nu$ be the flow of the dilation vector field $\de$. Then $(\Up \circ F_s) (x, \al) = \Up (x, e^s \al) = \Up_{e^s} (x, \al)$. Thus $\Up_t = \Up \circ F_{\log t}$, so we have
\begin{equation*}
\frac{\d}{\d t} \Up^*_t (\vp) = \frac{\d}{\d t} (\Up \circ F_{\log t})^* (\vp) = t^{-1} \cL_{\de} (\Up_t^* (\vp))
\end{equation*}
for $t > 0$. Using this we compute
\e 
\begin{aligned}
\d\biggl[\int_0^1t^{-1}\bigl(\de\cdot\Up_t^*(\vp)\bigr)\d t\biggr]&=\int_0^1t^{-1} \cL_\de(\Up_t^*(\vp))\d t=\int_0^1\frac{\d}{\d t}(\Up_t^*(\vp))\d t\\
&=\Up_1^*(\vp)-\Up_0^*(\vp).
\end{aligned}
\label{nc3eq32}
\e
Substituting \eq{nc3eq31} into \eq{nc3eq32}, using both \eq{nc3eq14} and the first line of \eq{nc3eq28}, and noting that $\Up_1=\Up$, $\Up_0^*(\vp)=\vp_{0,3}^0$, we obtain
\begin{align*}
\d \eta & = \Up_1^* (\vp) - \Up_0^* (\vp) - \dot \vp^2_{2,1} - \ddot \vp^2_{1,2} - \dddot \vp^2_{0,3} \\
& = \Up^* (\vp) - \vp^0_{0,3} - \dot \vp^2_{2,1} - \dddot \vp^2_{0,3}.
\end{align*}
In the same way, this time using the second line of \eq{nc3eq28} and $\Up_0^*(\psi) = 0$, we obtain
\begin{align*}
\d \zeta & = \Up_1^* (\psi) - \Up_0^* (\psi) - \dot \psi^2_{2,2} - \ddot \psi^2_{1,3} - \dot \psi^4_{4,0} - \ddot \psi^4_{3,1} - \dddot \psi^4_{2,2} \\
& = \Up^* (\psi) - \dot \psi^2_{2,2} - \dot \psi^4_{4,0} - \dddot \psi^4_{2,2}.
\end{align*}
Using \eq{nc3eq20} and $\ti\vp^\nu=\vp^\nu+\dddot\vp^2_{0,3}$ and\/ $\ti\psi^\nu=\psi^\nu+\dddot\psi^4_{2,2}$ from Corollary \ref{nc3cor} we conclude that
\e
\d\eta=\Up^*(\vp)-\ti\vp^\nu\vert_{U_R},\quad\text{and similarly}\quad
\d\ze=\Up^*(*\vp)-\ti\psi^\nu\vert_{U_R}.
\label{nc3eq33}
\e

Substituting \eq{nc3eq6} into \eq{nc3eq31} and integrating term by term yields
\e
\begin{split}
\eta&\sim \int_0^1\bigl(\ts\sum\limits_{n=1}^\iy t^{2n-1}\de\cdot\vp^{2n}\bigr)\d t-\frac{1}{2}\de\cdot\vp^2_{2,1}\\
&=\ts\sum\limits_{n=1}^\iy \ts\frac{1}{2n}\de\cdot\vp^{2n}-\frac{1}{2}\de\cdot\vp^2_{2,1}
=\ts\frac{1}{4}\de\cdot\vp^4_{2,1}+\frac{1}{4}\de\cdot\vp^4_{1,2}+
\sum\limits_{n=3}^\iy \ts\frac{1}{2n}\de\cdot\vp^{2n},
\end{split}
\label{nc3eq34}
\e
interpreted as a power series around the zero section $L$ in $U_R$, where we have used $\vp^2_{3,0}=\vp^2_{1,2}=\vp^4_{3,0}=0$ and $\de\cdot\vp^2_{0,3}=\de\cdot\vp^4_{0,3}=0$. From \eq{nc3eq18}, \eq{nc3eq34} and $\md{\de}=O(r)$ we see that
\e
\md{\eta}_{g^\nu}=O(r^3) \quad\text{and}\quad \md{\d\eta}_{g^\nu}=\bmd{\Up^*(\vp)-\ti\vp^\nu\vert_{U_R}}_{g^\nu}=O(r^2)
\label{nc3eq35}
\e
for small $r$. Similarly using $\psi^2_{4,0} = \psi^2_{3,1} = \psi^2_{1,3} = \psi^4_{3,1} = 0$ we find that
\begin{align*}
\zeta&\sim \int_0^1\bigl(\ts\sum\limits_{n=1}^\iy t^{2n-1}\de\cdot\psi^{2n}\bigr)\d t-\frac{1}{2}\de\cdot\psi^2_{2,2} -\frac{1}{4}\de\cdot\psi^4_{4,0}\\
&=\ts\sum\limits_{n=1}^\iy \ts\frac{1}{2n}\de\cdot\psi^{2n}-\frac{1}{2}\de\cdot\psi^2_{2,2} -\frac{1}{4}\de\cdot\psi^4_{4,0}\\
&=\ts\frac{1}{4}\de\cdot\psi^4_{2,2}+\frac{1}{4}\de\cdot\psi^4_{1,3}+
\sum\limits_{n=3}^\iy \ts\frac{1}{2n}\de\cdot\psi^{2n},
\end{align*}
and thus analogously that
\e
\md{\ze}_{g^\nu}=O(r^3) \quad\text{and}\quad \md{\d\ze}_{g^\nu}=\bmd{\Up^*(*\vp)-\ti\psi^\nu\vert_{U_R}}_{g^\nu}=O(r^2).
\label{nc3eq36}
\e

\section{\texorpdfstring{$\Gt$-structures on the resolution $P$ of $\nu/\{\pm 1\}$}{G₂-structures on the resolution P of ν/{±1}}}
\label{nc4}

In this section we define a resolution $\rho:P\ra \nu/\{\pm 1\}$ of the orbifold $\nu/\{\pm 1\}$, and a family of $\Gt$-structures $(\vp^P_t,g^P_t)$ on $P$ for $t>0$, which are asymptotic at infinity in $P$ to the $\Gt$-structures $(\vp^{\nu/\{\pm 1\}}_t,g^{\nu/\{\pm 1\}}_t)$ on $\nu/\{\pm 1\}$ induced by the $\Gt$-structure $(\vp^\nu_t,g^\nu_t)$ on $\nu$ defined in \eq{nc3eq21} of \S\ref{nc33}. We then modify the pair $\vp^P_t$, $\psi^P_t = \Theta(\vp^P_t)$ to a pair of closed forms $\ti \vp^P_t$, $\ti \psi^P_t$. However, the torsion $\ti \psi^P_t - \Theta(\ti \vp^P_t)$ will still be too large, and we will correct it further in \S\ref{nc5}.

We continue in the situation of \S\ref{nc3}, with $\Up,\ti\nabla^\nu$ chosen as in Proposition \ref{nc3prop1}. This section will use only the following data from \S\ref{nc3}:
\begin{itemize}
\setlength{\itemsep}{0pt}
\setlength{\parsep}{0pt}
\item The compact 3-manifold $L$, and rank 4 real vector bundle $\pi:\nu\ra L$.
\item The fibre metrics $h_\nu$ and $g_L$ on the bundles $\nu$ and $TL$ over $L$, respectively, and the radius function $r:\nu\ra[0,\iy)$ defined using~$h_\nu$.
\item The connection $\ti\nabla^\nu$ on $\nu$, which induces a splitting $T\nu=V\op H$ with $V\cong\pi^*(\nu)$ and $H\cong \pi^*(TL)$, and a notion of type $(i,j)$ forms on~$\nu$.
\item The data $\vp^\nu_t=\vp^0_{0,3}+t^2\dot\vp^2_{2,1}$, $\psi_t^\nu=t^2\dot\psi^2_{2,2}+t^4\dot\psi^4_{4,0}$, $g_t^\nu=g^\nu_{0,2}+t^2g^\nu_{2,0}$ on $\nu$ from \eq{nc3eq21} for $t>0$, giving a $\Gt$-structure $(\vp^\nu_t,g^\nu_t)$ on $\nu$ with 4-form $\psi^\nu_t = \Theta (\vp^\nu_t)$.
\item The forms $\dddot\vp^2_{0,3}$ and $\dddot\psi^4_{2,2}$ on $\nu$, where $\ti\vp^\nu_t=\vp^\nu_t+t^2\dddot\vp^2_{0,3}$ and $\ti\psi^\nu_t=\psi^\nu_t+t^4\dddot\psi^4_{2,2}$ are closed forms on $\nu$.
\item The closed, coclosed, nonvanishing 1-form $\la$ on $L$ from Assumption \ref{nc3ass}.
\end{itemize}

\subsection{\texorpdfstring{Defining the resolution $P$ of $\nu/\{\pm 1\}$}{Defining the resolution P of ν/{±1}}}
\label{nc41}

Define a smooth function $A:L\ra(0,\iy)$ by $A=\md{\la}_{g_L}$, noting that $\la$ is nowhere vanishing, and define a 1-form $e_1=A^{-1} \la$ on $L$, so that $e_1$ is of unit length with respect to the metric $g_L$ on $L$, and $\la=A e_1$. It will be convenient for us to extend $e_1$ to an oriented orthonormal basis $(e_1,e_2,e_3)$ of 1-forms on $L$, but this may not be possible globally on $L$, as the rank 2 subbundle $\an{e_1}^\perp\subset T^*L$ may not be trivial. So choose an open subset $L'\subseteq L$ on which $\an{e_1}^\perp$ is trivial, and choose 1-forms $e_2,e_3$ on $L'$ such that $(e_1,e_2,e_3)$ is orthonormal with respect to $g_L$, and oriented with respect to the orientation on $L$ induced by $\vp$, so that since $L$ is associative in $(M, \vp, g)$, we have $\vp\vert_{L'}=e_1\w e_2\w e_3$. Then $*\la=A e_2\w e_3$. Later we will show that the important structures we define are independent of the choice of $e_2,e_3$, and so are well defined over all of $L$, not just over~$L'\subseteq L$.

Since $L$ is associative in $(M,\vp,g)$, at each point $x\in L'$ there exists an isomorphism $T_xM\cong\R^7$ identifying $\vp\vert_x,g\vert_x,*\vp\vert_x$ with $\vp_0,g_0,*\vp_0$ in \eq{nc2eq1}--\eq{nc2eq3} and identifying $T_xL$ with $\bigl\{(x_1,x_2,x_3,0,0,0,0):x_j\in\R\bigr\}\subset\R^7$. We can also choose this to identify $e_1\vert_x,e_2\vert_x,e_3\vert_x$ with $\d x_1,\d x_2,\d x_3$. Therefore from equations \eq{nc2eq7} and \eq{nc2eq13} in \S\ref{nc24}, we see that there are unique smooth sections $\hat\om^I,\hat\om^J,\hat\om^K$ of $\La^2\nu^*\ra L'$ such that
\e
\begin{aligned}
\vp\vert_{L'}&=e_1\!\w\! e_2\!\w\! e_3\!-\!e_1\!\w\!\hat\om^I\!-\!e_2\!\w\!\hat\om^J\!-\!e_3\!\w\!\hat\om^K &&\text{in $\Ga^\iy(\La^3T^*M\vert_{L'})$,}\\
*\vp\vert_{L'}&=\vol_\nu\!-e_2\!\w\! e_3\!\w\!\hat\om^I\!-\!e_3\!\w\! e_1\!\w\!\hat\om^J\!-\!e_1\!\w\! e_2\!\w\!\hat\om^K &&\text{in $\Ga^\iy(\La^4T^*M\vert_{L'})$,}
\end{aligned}
\label{nc4eq1}
\e
where $\hat\om^I,\hat\om^J,\hat\om^K$ are the Hermitian forms with respect to $h_\nu$ of almost complex structures $I,J,K: \nu\vert_{L'}\ra \nu\vert_{L'}$ on the fibres of $\nu$ satisfying the quaternion relations \eq{nc2eq9}. That is, at each $x\in L'$ there exist linear coordinates $(y_1,y_2,y_3,y_4)$ on $\nu\vert_{L'}$ such that $\hat\om^I\vert_x,\hat\om^J\vert_x,\hat\om^K\vert_x,h_\nu\vert_x,I\vert_x,J\vert_x,K\vert_x$ are given by equations \eq{nc2eq7}--\eq{nc2eq11} for~$\om_0^I,\om_0^J,\om_0^K,h_0,I,J,K$.

If $e'_2,e'_3$ are alternative choices for $e_2,e_3$ yielding $\hat\om^{\prime I},\hat\om^{\prime J},\hat\om^{\prime K},I',J',K'$ then we may write $e_2'=\cos\phi\, e_2+\sin\phi \,e_3$ and $e_3'=-\sin\phi\, e_2+\cos\phi\, e_3$ for some smooth $\phi:L'\ra\R/2\pi\Z$, and comparing \eq{nc4eq1} for $e_1,e_2,e_3$ and $e_1,e_2',e_3'$ we see that
\e
\begin{aligned}
\hat\om^{\prime I}&=\hat\om^I, & \hat\om^{\prime J}&=\cos\phi\, \hat\om^J+\sin\phi\,\hat\om^K, & \hat\om^{\prime K}&=-\sin\phi\, \hat\om^J+\cos\phi\,\hat\om^K,\\
I'&=I, & J'&=\cos\phi\, J+\sin\phi\,K, & K'&=-\sin\phi\, J+\cos\phi\,K.
\end{aligned}
\label{nc4eq2}
\e
Hence $\hat\om^I,I$ are independent of the choice of $e_2,e_3$ in the orthonormal basis of sections $(e_1,e_2,e_3)$ on $L'$. So by covering $L$ by open $L'\subseteq L$ on which such a basis $(e_1,e_2,e_3)$ exists, we construct global $\hat\om^I\in\Ga^\iy(\La^2\nu^*)$ and $I:\nu\ra\nu$ satisfying \eq{nc4eq1} on each $L'$ in the cover.

Now $I:\nu\ra\nu$ with $I^2=-\bs 1$ is a complex structure on the fibres of $\nu\ra L$, making $\nu$ into a rank 2 complex vector bundle, with fibre $\C^2$.

\begin{rem} The complex structure $I$ on $\nu$, identifying the $\R^4$ fibres with $\C^2$, can also be understood as follows. The unit $1$-form $e_1$ on $L$ is metric dual to a unit vector field, still denoted $e_1$, on $L$. Because $L$ is associative in $M$, it follows that \emph{cross product} with $e_1$ takes normal vectors to normal vectors, and squares to minus the identity. Thus $I = e_1 \times (\cdot)$. See also Remark~\ref{assocrmk}.
\label{nc4rem1}
\end{rem}

Form the quotient $\nu/\{\pm 1\}$, as a 7-dimensional orbifold, and write $\varpi:\nu\ra\nu/\{\pm 1\}$ and $\pi:\nu/\{\pm 1\}\ra L$ for the obvious projections and $0:L\ra\nu/\{\pm 1\}$ for the zero section. Then $\pi:\nu/\{\pm 1\}\ra L$ is a fibre bundle with fibre $\C^2/\{\pm 1\}$. 

As in \S\ref{nc25}, the Eguchi--Hanson space $X$ is the blow-up $B:X\ra\C^2/\{\pm 1\}$ of $\C^2/\{\pm 1\}$ at 0, with exceptional divisor $Y=B^{-1}(0)$ where $Y\cong\CP^1$. We can do this blow-up construction fibrewise over $L$. So let $\rho:P\ra\nu/\{\pm 1\}$ be the bundle blow-up of $\nu/\{\pm 1\}$ along the zero section $0(L)\subset\nu/\{\pm 1\}$, using the complex structure $I$ on the fibres $\C^2/\{\pm 1\}$ to define the blow-up. Write $\si=\pi\ci\rho:P\ra L$. Then $P$ is a smooth 7-manifold, and $\si:P\ra L$ a smooth map which is a bundle with fibre the Eguchi--Hanson space $X$.

Write $Q=\rho^{-1}(0(L))\subset P$ for the preimage of the zero section in $P$, and write $\inc : Q\hookra P$ for the inclusion, and $\Pi=\si\vert_Q:Q\ra L$. Then $Q$ is a smooth 5-manifold, an embedded submanifold of $P$ with embedding $\inc:Q\hookra P$, and $\Pi:Q\ra L$ is smooth and a fibre bundle with fibre $\CP^1$ the exceptional divisor $Y\subset X$ from \S\ref{nc25}. For each $x\in L$, write $P_x=\si^{-1}(x)\subset P$ and $Q_x=\Pi^{-1}(x)\subset Q$ for the fibres of $\si:P\ra L$ and $\Pi:Q\ra L$ over $x$, respectively. Then $P_x$ is diffeomorphic to the Eguchi--Hanson space $X$, and $Q_x\subset P_x$ to the exceptional divisor~$Y\subset X$.

As in \S\ref{nc25}, we have $Y\cong\CP^1=\P(\C^2)$, and there is a natural projection $\pi:X\ra Y$ which realizes $X$ as the total space of the complex line bundle $T^*Y\ra Y$. Similarly, there is a natural projection $\pi:P\ra Q$ with $\Pi\ci\pi=\si:P\ra L$, such that $\pi_x:P_x\ra Q_x$ is identified with $\pi:X\ra Y$ on the fibres over each $x\in L$. We summarize our work so far in the diagram
\begin{equation*}
\xymatrix@C=70pt@R=15pt{
*+[r]{P} \ar[r]_(0.6){\rho} \ar[dr]^(0.65){\si} \ar@<-.5ex>[d]_{\pi} & \nu/\{\pm 1\} \ar@<-.5ex>[d]_{\pi} & *+[l]{\nu} \ar@<-.5ex>[d]_{\pi} \ar[l]^(0.45){\varpi} \\
*+[r]{Q=\P_\C(\nu)} \ar@<-.5ex>[u]_{\inc} \ar[r]^{\Pi} & L \ar@<-.5ex>[u]_{0} \ar@{=}[r] & *+[l]{L.\!} \ar@<-.5ex>[u]_{0} }
\end{equation*}
Note that $P,Q,\rho,\si,\pi,\Pi$ on the left hand side depend on the complex structure $I$ on the fibres of $\nu$ used to define the blow-up.

\begin{rem} There is an equivalent way to define $P,Q$ using principal bundles. Corresponding to the rank 2 complex vector bundle $(\nu,I)\ra L$ there is a natural $\GL(2,\C)$-principal bundle $F_{\GL(2,\C)}\ra L$, where points of $F_{\GL(2,\C)}$ over $x\in L$ are isomorphisms $(\nu\vert_x,I\vert_x)\cong(\C^2,i)$. Then we may write
\e
P\cong (F_{\GL(2,\C)} \t X)/\GL(2,\C),\qquad Q\cong (F_{\GL(2,\C)} \t Y)/\GL(2,\C),
\label{nc4eq3}
\e
where $\GL(2,\C)$ acts on $X$ and $Y$ in the natural ways, and on $F_{\GL(2,\C)}$ by the principal bundle action. Equation \eq{nc4eq3} makes it clear that $P,Q$ are bundles over $L$ with fibres $X,Y$, and that the bundle blow-up procedure is well defined.

If we also use the Hermitian metric $h_\nu$ on the fibres of $(\nu,I)$ we can instead define a natural $\U(2)$-principal bundle $F_{\U(2)}\ra L$, where points of $F_{\U(2)}$ over $x\in L$ are isomorphisms $(\nu\vert_x,I\vert_x,h_\nu\vert_x)\cong(\C^2,i,h_0)$. Then as in \eq{nc4eq3} we have natural diffeomorphisms
\e
P\cong (F_{\U(2)} \t X)/\U(2),\qquad Q\cong (F_{\U(2)} \t Y)/\U(2).
\label{nc4eq4}
\e
This means that $\U(2)$-invariant tensors on $X$, including the Eguchi--Hanson metrics $h_a$ for $a>0$ in \S\ref{nc25}, can be carried over to~$P$.
\label{nc4rem2}
\end{rem}

\begin{rem} Note that $\rho\vert_{P\sm Q}:P\sm Q\ra (\nu/\{\pm 1\})\sm 0(L)$ is a diffeomorphism, since $B\vert_{X\sm Y}:X\sm Y\ra(\C^2/\{\pm 1\})\sm\{0\}$ is a diffeomorphism. Thus, if $T$ is any tensor on the total space of $\nu$ that is invariant under $\{\pm 1\}$, it descends to $T$ on $\nu/\{\pm 1\}$, and then pulls back to a smooth tensor $\rho\vert_{P\sm Q}^*(T)$ on $P\sm Q$.

However, as explained in Remark \ref{nc2rem2} for $B:X\ra\C^2/\{\pm 1\}$, it is {\it not\/} necessary that $\rho\vert_{P\sm Q}^*(T)$ extends smoothly to a tensor on all $P$, although if it does extend, the extension is unique, as $P\sm Q$ is open and dense in $P$.

For example, the radius function $r:\nu\ra[0,\iy)$ defined using $h_\nu$ is $\{\pm 1\}$-invariant, and so descends to $r:\nu/\{\pm 1\}\ra[0,\iy)$. We write
\begin{equation*}
\check r=r\ci\rho:P\ra[0,\iy)
\end{equation*}
for the radius function on $P$, so that $Q=\check r^{-1}(0)$. In Remark \ref{nc2rem2} we explained that although $r^2:\C^2\ra\R$ and $r^2:\C^2/\{\pm 1\}\ra\R$ are smooth (in the orbifold sense), $r^2:X\ra\R$ is {\it not\/} smooth, but $r^4:X\ra\R$ is. As $r^4:X\ra\R$ is $\U(2)$-invariant, it follows from \eq{nc4eq4} that $\check r^2:P\ra\R$ is not smooth, but $\check r^4:P\ra\R$ is smooth.

In what follows, we often define metrics, exterior forms, and tensors on $P$ by pullback along $\rho\vert_{P\sm Q}:P\sm Q\ra (\nu/\{\pm 1\})\sm 0(L)$, as above, but we must be careful to justify that these extend smoothly from $P\sm Q$ to $P$.
\label{nc4rem3}	
\end{rem}

\subsection{\texorpdfstring{Constructing a splitting $TP\cong\breve V\op\breve H$}{Constructing a splitting TP≅V⊕Ĥ}}
\label{nc42}

In \S\ref{nc32} we chose a connection $\ti\nabla^\nu$ on $\nu\ra L$ and defined a splitting $T\nu=V\op H$ in \eq{nc2eq16}, where $V\cong \pi^*(\nu)$ and $H\cong \pi^*(TL)$. Similarly, we will need a splitting $TP\cong\breve V\op\breve H$ for the fibre bundle $\si:P\ra L$, where the vertical subbundle $\breve V=\Ker(\d\si:TP\ra\si^*(TL))$ is the subbundle of $TP$ whose fibre at a point $p$ is the tangent space at $p \in P_{\sigma(p)}$ of the Eguchi--Hanson fibre $P_{\sigma(p)}$ of $P$ over $\sigma(p) \in L$, and the horizontal subbundle $\breve H$ has~$\breve H\cong\si^*(TL)$.

The splitting $T\nu=V\op H$ is $\{\pm 1\}$-invariant, so it descends to $T(\nu/\{\pm 1\})=V\op H$. For compatibility with \S\ref{nc3} we would like $\rho:P\ra\nu/\{\pm 1\}$ to map $\breve V,\breve H$ on $P$ to $V,H$ on $\nu/\{\pm 1\}$ near infinity in $P$. This works for $\breve V,V$ because $\rho$ is fibre-preserving and $\breve V, V$ are canonical. However, we cannot simply define $\breve H$ by pulling back $H$ by $\rho:P\ra\nu/\{\pm 1\}$, since as in Remark \ref{nc4rem3}, the resulting $\breve H$ might not be smooth at~$Q\subset P$.

As $\rho:P\ra\nu/\{\pm 1\}$ was defined using the complex structure $I$ on the fibres of $\nu$ in \S\ref{nc41}, it is the case that the pullback of $H$ on $\nu/\{\pm 1\}$ to $P\sm Q$ extends smoothly over $Q$ if $\ti\nabla^\nu I=0$. This is a consequence of the description of $P$ in \eq{nc4eq4}. But $\ti\nabla^\nu$ was determined uniquely in Lemma \ref{nc3lem3}, so we are \emph{not} free to choose $\ti\nabla^\nu$ with~$\ti\nabla^\nu I=0$.

Choose another connection $\check\nabla^\nu$ on $\nu\ra L$ compatible with $h_\nu$ and $I$, so that $\check\nabla^\nu h_\nu=\check\nabla^\nu I=0$. This can always be done by general principal bundle theory, and is not unique. Write $\check\nabla^\nu=\ti\nabla^\nu+\Ga$ for $\Ga\in\Ga^\iy(\End(\nu)\ot T^*L)$. Then $\ti\nabla^\nu+s \Ga$ for $s\in[0,1]$ is a 1-parameter family of connections interpolating between $\ti\nabla^\nu$ at $s=0$, and $\check\nabla^\nu$ at $s=1$. Choose a smooth function $a:\R\ra\R$ with $a(r)=1$ for $r\le \ha$, and $a(r)\in(0,1)$ for $r\in(\ha,1)$, and $a(r)=0$ for $r\ge 1$. Define a splitting $TP\cong \breve V\op\breve H$ by $\breve V=\Ker(\d\si:TP\ra\si^*(TL))$, and for $p\in P\sm Q$, define $\breve H\vert_p$ to be identified by $\d_p\rho:T_pP\ra T_{\rho(p)}(\nu/\{\pm 1\})$ with the horizontal subspace of the connection $\ti\nabla^\nu+(a\ci \check r(p)) \Ga$ on $\nu$ at $\rho(p)$ in~$\nu/\{\pm 1\}$.

Then $\breve H\vert_p$ depends smoothly on $p\in P\sm Q$, since $a\ci\check r:P\ra\R$ is smooth. Where $\check r\ge 1$ in $P$ we have $a\ci\check r=0$, so $\breve H$ is defined using $\ti\nabla^\nu$, and coincides with the pullback of $H$ in \S\ref{nc32} as we want. Where $\check r\le\ha$ in $P$ we have $a\ci\check r=1$, so $\breve H$ is defined using $\check\nabla^\nu$, and as $\check\nabla^\nu I=0$ we see that $\breve H$ extends smoothly over $Q$. Thus, we have defined a smooth splitting of vector bundles on $P$:
\begin{equation*}
TP=\breve V\op\breve H,\;\>\text{where $\breve V\vert_p=T_pP_{\si(p)}\subset T_pP$, $p\in P$, and $\breve H\cong \si^*(TL)$.}
\end{equation*}
As in \eq{nc3eq12}, this induces a splitting
\e
\La^kT^*P= \bigop_{i+j=k, \;
0\le i\le 4,\; 0\le j\le 3}\La^i\breve V^*\ot \La^j\breve H^*,
\label{nc4eq5}
\e
and we call a $k$-form on $P$ {\it of type\/} $(i,j)$ if it lies in the factor $\La^i\breve V^*\ot \La^j\breve H^*$ in \eq{nc4eq5}. We write $[\al]_{i,j}$ for the type $(i,j)$ component of a $k$-form $\al$, so~$\al=\sum_{i+j=k}[\al]_{i,j}$.

For each $x\in L$, the fibre $P_x$ of $\si:P\ra L$ over $x$ is isomorphic to the Eguchi--Hanson space $X$, which is a complex surface, and $\breve V_p=T_pP_x$ for $p\in P_x$. The complex structures on each $P_x$ induce a unique vector bundle isomorphism $I:\breve V\ra\breve V$ on $P$ with~$I^2=-\bs 1$.

\begin{rem} We address here a subtle point. Recall that in $TP=\breve V\op\breve H$, the vertical subbundle $\breve V=\Ker(\d\si:TP\ra\si^*(TL))$ is canonical, but the horizontal subbundle $\breve H$ depends on an arbitrary choice. However, in $T^*P=\breve V^*\op\breve H^*$, the factor $\breve H^*$ is canonical, because it is the annihilator $(\breve V)^\ci$, and in addition $\breve H^*$ is the pullback $\si^*(T^*L)$ of the cotangent bundle of $L$. But $\breve V^*$ depends on a choice, as it is~$(\breve H)^\ci$.

Let $x\in L$, so that $P_x\subset P$ is the fibre of $\si:P\ra L$ over $x$, a 4-submanifold of $P$. Then $(0,1)$-forms on $P$ restrict to zero on $P_x$, so $(i,j)$-forms for $j>0$ restrict to zero on $P_x$. If $\al$ is a $k$-form on $P$, then $\al=\sum_{i+j=k}[\al]_{i,j}$ implies~that
\e
[\al]_{k,0}\vert_{P_x}=\al\vert_{P_x}\qquad\text{in $k$-forms on $P_x$.}
\label{nc4eq6}
\e

Here is the subtle point. Although $[\al]_{k,0}$ depends on the choice of $\breve H$, its restriction $[\al]_{k,0}\vert_{P_x}$ does not. In fact, $[\al]_{k,0}$ is independent of $\breve H$ {\it when considered as a section of the vector bundle $\La^k\breve V^*$ associated to\/} $\breve V=\Ker(\d\si:TP\ra\si^*(TL))$. But {\it when considered as a $k$-form on\/} $P$, it does depend on $\breve H$, as the embedding $\La^k\breve V^*\hookra\La^kT^*P$ depends on~$\breve H$.
\label{nc4rem4}	
\end{rem}

\subsection{\texorpdfstring{Defining some forms and tensors on $P$}{Defining some forms and tensors on P}}
\label{nc43}

In this section we define some forms and tensors on $P$ that will be used to construct $\Gt$-structures $(\vp_t^P,g_t^P)$ in \S\ref{nc44}. Let $f_a:(0,\iy)\ra\R$ be the smooth function in \eq{nc2eq15} such that $f_a(r):\C^2/\{\pm 1\}\ra\R$ is the K\"ahler potential of the Eguchi--Hanson space $(X,h_a)$, where $f_a$ depends smoothly on $a\in(0,\iy)$. Let $A:L\ra(0,\iy)$ be the smooth function $A(x) = \md{\la_x}_{g_L}$ from \S\ref{nc41}, so that $A\ci\si:P\ra(0,\iy)$ is smooth, and $\check r\vert_{P\sm Q}:P\sm Q\ra(0,\iy)$ is smooth. Thus $f_{A\ci\si}(\check r):P\sm Q\ra(0,\iy)$ is a smooth function mapping $p\mapsto f_{A\ci\si(p)}(\check r(p))$. Note that $f_{A\ci\si(p)}(\check r(p))\ra -\iy$ as $p\ra Q$ in $P$, so $f_{A\ci\si}(\check r)$ \emph{does not} extend smoothly from $P\sm Q$ to~$P$. 

Thus, $\d(f_{A\ci\si}(\check r))$ is a closed 1-form on $P\sm Q$, and we denote by $[\d(f_{A\ci\si}(\check r))]_{1,0}$ its type (1,0)-component in $\breve V^*\subset\breve V^*\op\breve H^*=T^*P$. Applying the complex structure $I:\breve V\ra\breve V$ in \S\ref{nc42} gives another type (1,0) 1-form $I([\d(f_{A\ci\si}(\check r))]_{1,0})$ on $P\sm Q$, so we can take $\d$ of this, and divide into components of type (2,0), (1,1), and (0,2), which we do to define 2-forms $\check\om^I,\ka_{1,1},\ka_{0,2}$ of types (2,0), (1,1), (0,2) on $P\sm Q\subset P$ as follows:
\e
\begin{split}
\check\om^I&=-\ts\frac{1}{4}\bigl[\d\bigl(I([\d(f_{A\ci\si}(\check r))]_{1,0})\bigr)\bigr]_{2,0}\\
&=-\ts\frac{1}{4}(A\ci\si) \bigl[\d\bigl((A\ci\si)^{-1} I([\d(f_{A\ci\si}(\check r))]_{1,0})\bigr)\bigr]_{2,0},\\
\ka_{1,1}&=-\ts\frac{1}{4}\bigl[\d\bigl((A\ci\si)^{-1} I([\d(f_{A\ci\si}(\check r))]_{1,0})\bigr)\bigr]_{1,1},\\
\ka_{0,2}&=-\ts\frac{1}{4}\bigl[\d\bigl((A\ci\si)^{-1} I([\d(f_{A\ci\si}(\check r))]_{1,0})\bigr)\bigr]_{0,2}.\end{split}
\label{nc4eq7}
\e
By $(A \ci \si)^{-1}$ we mean the smooth function $P \to (0, \iy)$ given by
\begin{equation*}
(A \ci \si)^{-1} (p) = [ (A \ci \si)(p) ]^{-1} = \frac{1}{ A(\sigma(p)) }.
\end{equation*}
We can also write this as $A^{-1} \ci \si$ where $A^{-1} (x) = [A(x)]^{-1}$. The first two lines of \eq{nc4eq7} are equal because $\d((A\ci\si)^{-1})=\si^*(\d(A^{-1}))$ is a 1-form of type (0,1), and so makes no contribution to the type (2,0) component.

\begin{rem} The motivation for defining $\check\om^I$ as we did in \eq{nc4eq7} comes from \eq{nc2eq16}. If one thinks of restricting to a fixed fibre $P_x$ of $P$, then since we take only vertical (fibre) derivatives in the definition of $\check\om^I$ it corresponds to the hyperK\"ahler form $\om^I_a$ on $X \cong P_x$ in the domain $X \sm Y$. We make this more explicit at the end of this section. On the other hand, the forms $\ka_{1,1}$ and $\ka_{2,2}$ were defined as they were in \eq{nc4eq7} precisely so that
\e
\d\bigl[(A\ci\si)^{-1} \check\om^I+\ka_{1,1}+\ka_{0,2}\bigr]=0,
\label{nc4eq8}
\e
which is easy to see because $(A\ci\si)^{-1} \check\om^I+\ka_{1,1}+\ka_{0,2}$ is exact.
\label{nc4rem5}
\end{rem}

We now use the crucial assumption that $\la$ is closed and coclosed on $L$ to derive two important relations, namely \eq{nc4eq9} and \eq{nc4eq10}, which will be used to construct \emph{closed} 3- and 4-forms on $P$ in \S\ref{nc44}.

Since $\la=A e_1$ is a closed 1-form on $L$, taking the exterior product of \eq{nc4eq8} with $\si^*(\la)=(A\ci\si)\si^*(e_1)$, which is of type (0,1), implies that
\e
\d\bigl[\si^*(e_1)\w \check\om^I+\si^*(\la)\w\ka_{1,1}+\si^*(\la)\w\ka_{0,2}\bigr]=0,
\label{nc4eq9}
\e
with $\si^*(e_1)\w \check\om^I$, $\si^*(\la)\w\ka_{1,1}$, $\si^*(\la)\w\ka_{0,2}$ of types (2,1), (1,2), (0,3), respectively.

Similarly, for $e_1,e_2,e_3$ over $L'\subset L$ as in \S\ref{nc41}, so that $*\la=A e_2\w e_3$ is a closed 2-form on $L$, taking the exterior product of \eq{nc4eq8} with $\si^*(*\la)=(A\ci\si)\si^*(e_2)\w\si^*(e_3)$, which is of type (0,2), implies that
\e
\d\bigl[\si^*(e_2)\w\si^*(e_3)\w \check\om^I+\si^*(*\la)\w\ka_{1,1}\bigr]=0,
\label{nc4eq10}
\e
on $\si^{-1}(L')\sm Q$, where $\si^*(*\la)\w\ka_{0,2}=0$ automatically as it is a (0,4)-form.

We want to show that $\check\om^I,\ka_{1,1},\ka_{0,2}$ extend uniquely to smooth forms on all of $P$. Combining equations \eq{nc2eq19} and \eq{nc4eq7} gives the alternative expressions
\e
\begin{aligned}
\check\om^I&\!=\!-\ts\frac{1}{2}(A\!\ci\!\si)\bigl[\d\bigl(I([\d(\log\check r)]_{1,0})\bigr)\bigr]_{2,0}
\!-\!\frac{1}{4}\bigl[\d\bigl(I([\d(H_{A\ci\si}(\check r^4))]_{1,0})\bigr)\bigr]_{2,0},\\
\ka_{1,1}&\!=\!-\ts\frac{1}{2}\bigl[\d\bigl(I([\d(\log\check r)]_{1,0})\bigr)\bigr]_{1,1}\!\!-\!\frac{1}{4}\bigl[\d\bigl((A\!\ci\!\si)^{-1} I([\d(H_{A\ci\si}(\check r^4))]_{1,0})\bigr)\bigr]_{1,1},\\
\ka_{0,2}&\!=\!-\ts\frac{1}{2}\bigl[\d\bigl(I([\d(\log\check r)]_{1,0})\bigr)\bigr]_{0,2}\!\!-\!\ts\frac{1}{4}\bigl[\d\bigl((A\!\ci\!\si)^{-1} I([\d(H_{A\ci\si}(\check r^4))]_{1,0})\bigr)\bigr]_{0,2}.
\end{aligned}\!\!
\label{nc4eq11}
\e
Recall that $\pi:P\ra Q$ may be considered as a complex line bundle over $Q$. The connection $\check\nabla^\nu$ on $\nu\ra L$, which is compatible with $h_\nu$ and $I$, induces a $\U(1)$-connection on this line bundle $\pi:P\ra Q$, whose curvature is a closed 2-form $F_Q$ on $Q$ restricting to $-2\om_{\CP^1}$ in \eq{nc2eq20} on each fibre $Q_x\cong\CP^1$ of $\Pi:Q\ra L$. Calculating in a local trivialization of $(\nu,I)\ra L$ we can show that
\begin{equation*}
\pi^*(F_Q)\vert_{P\sm Q}=\d\bigl(I([\d(\log\check r)]_{1,0})\bigr)\qquad\text{where $\check r\le \ha$ on $P\sm Q$.}
\end{equation*}
Here $\check r\le\ha$ is needed so that $\breve H$ used to define $[\cdots]_{1,0}$ comes from~$\check\nabla^\nu$.

Thus \eq{nc4eq11} implies that where $\check r\le\ha$ on $P\sm Q$ we have
\e
\begin{split}
\check\om^I&=-\ts\frac{1}{2}(A\ci\si)\cdot[\pi^*(F_Q)]_{2,0}
-\ts\frac{1}{4}\bigl[\d\bigl(I([\d(H_{A\ci\si}(\check r^4))]_{1,0})\bigr)\bigr]_{2,0},\\
\ka_{1,1}&=-\ts\frac{1}{2}[\pi^*(F_Q)]_{1,1}-\frac{1}{4}\bigl[\d\bigl((A\ci\si)^{-1}\cdot I([\d(H_{A\ci\si}(\check r^4))]_{1,0})\bigr)\bigr]_{1,1},\\
\ka_{0,2}&=-\ts\frac{1}{2}[\pi^*(F_Q)]_{0,2}-\frac{1}{4}\bigl[\d\bigl((A\ci\si)^{-1}\cdot I([\d(H_{A\ci\si}(\check r^4))]_{1,0})\bigr)\bigr]_{0,2}.
\end{split}
\label{nc4eq12}
\e
Now each term on the right hand sides of \eq{nc4eq12} extends smoothly to $P$, since $\pi^*(F_Q)$ is smooth on $P$ and $\check r^4:P\ra\R$ is smooth as in Remark \ref{nc4rem3}, so $H_{A\ci\si}(\check r^4):P\ra\R$ is smooth. Therefore $\check\om^I,\ka_{1,1},\ka_{0,2}$ are defined as smooth 2-forms on $P$. Also \eq{nc4eq8}--\eq{nc4eq9} hold on all of $P$, by continuity, and \eq{nc4eq10} holds on~$\si^{-1}(L')$.

We now make Remark \ref{nc4rem5} more explicit. Let $x\in L$. By restricting \eq{nc4eq7} to $P_x$, using \eq{nc4eq6} twice, and the fact that $\d$ commutes with pullbacks we find that
\e
\check\om^I\vert_{P_x}=-\ts\frac{1}{4}\d\bigl(I(\d(f_{A(x)}(\check r)\vert_{P_x}))\bigr),
\label{nc4eq13}
\e
where there are now no projections $[\cdots]_{1,0},[\cdots]_{2,0}$, and we use $A\ci\si\vert_{P_x}\equiv A(x)$. By construction, $(P_x,I\vert_{P_x})$ is isomorphic as a complex surface to the Eguchi--Hanson space $X$. By \eq{nc4eq4}, this isomorphism is only canonical up to the action of $\U(2)$ on $X$. Comparing \eq{nc2eq16} and \eq{nc4eq13} shows that this isomorphism $P_x\cong X$ identifies $\check\om^I\vert_{P_x}$ with the K\"ahler form $\om_{A(x)}^I$ of the hyperK\"ahler metric $h_{A(x)}$ on $X$. Therefore if we define a section $\check g_{2,0}$ of $S^2\breve V^*\ra P$ by
\e
\check g_{2,0}(v,w)=\check\om^I(v,Iw)\qquad \text{for all $v,w\in\Ga^\iy(\breve V)$}
\label{nc4eq14}
\e 
then $\check g_{2,0}\vert_{P_x}$ is identified with the Eguchi--Hanson metric $h_{A(x)}$ on $X$ by the isomorphism $P_x\cong X$. That is, $\check g_{2,0}$ is a family of Eguchi--Hanson metrics on the fibres $P_x$ of~$\si:P\ra L$ at $x\in L$ .

\subsection{\texorpdfstring{A family of $\Gt$-structures $(\vp_t^P,g_t^P)$ on $P$}{A family of G₂-structures (φᵗᴾ,gᵗᴾ) on P}}
\label{nc44}

Now return to our choice in \S\ref{nc41} of open $L'\subseteq L$, and oriented orthonormal basis $(e_1,e_2,e_3)$ of 1-forms on $L'$. Recall that in \S\ref{nc41} we had defined $\la=A e_1$ for smooth $A:L\ra(0,\iy)$, and the data $\hat\om^I,\hat\om^J,\hat\om^K,I,J,K$ on $L'$. We will show that the important structures we define are independent of the choices of $L',e_2,e_3$.

In \S\ref{nc25} we noted that because $B:X\ra\C^2/\{\pm 1\}$ is a holomorphic blow-up, and $\om^J_0 + i\om^K_0 = \d z_1\w\d z_2$ is a holomorphic $(2,0)$-form on $\C^2/\{\pm 1\}$, the pullback $\om^J + i\om^K := B^*(\om^J_0 + i\om^K_0)$ is a holomorphic, and hence smooth, complex 2-form on $X$, so that $\om^J=B^*(\om^J_0)$ and $\om^K=B^*(\om^K_0)$ are smooth. Similarly, there are unique smooth sections $\check\om^J,\check\om^K$ of $\La^2\breve V^*$ on $\si^{-1}(L')\subseteq P$, such that for each $x\in L'$, under the identifications $\La^2\breve V^*\vert_{P_x}\cong \La^2T^*P_x$ and $\La^2T^*(\nu_x/\{\pm 1\})\cong\pi^*(\La^2\nu_x^*)$ we have $\check\om^J\vert_{P_x}\cong \rho^*(\pi^*(\hat\om^J\vert_x))$ and $\check\om^K\vert_{P_x}\cong \rho^*(\pi^*(\hat\om^K\vert_x))$.

Thus, for each $x\in L'$, we have 2-forms $\check\om^I\vert_{P_x},\check\om^J\vert_{P_x},\check\om^K\vert_{P_x}$ on $P_x$ which are identified with the 2-forms $\om^I_{A(x)},\om^J,\om^K$ on the Eguchi--Hanson space $X$ in \S\ref{nc25} by the isomorphism $P_x\cong X$. We have 
\begin{equation*}
\om^I_{A(x)}\w\om^I_{A(x)}=\om^J\w\om^J=\om^K\w\om^K
\end{equation*}
on $X$ by properties of hyperK\"ahler manifolds. Thus we deduce that
\e
\check\om^I\w\check\om^I=\check\om^J\w\check\om^J=\check\om^K\w\check\om^K\quad\text{in $\La^4\breve V^*$ on $\si^{-1}(L')\subseteq P$.}
\label{nc4eq15}
\e

Let $\de$ denote the dilation vector field on $\C^2/\{\pm 1\}$. Since the closed form $\om^J_0 + i \om^K_0 = \d z_1 \w \d z_2$ is homogeneous of order 2 under dilations, we deduce that on $\C^2/\{\pm 1\}$ we have
\e
\om^J_0+i\om^K_0=\ha\cL_\de(\om^J_0+i\om^K_0)=\ha\d[\de\cdot \om^J_0+i\de\cdot\om^K_0].
\label{nc4eq16}
\e
In complex coordinates, $\de = z_k \partial_{z_k} + \bar z_k \partial_{\bar z_k}$. Hence $\de\cdot \om^J_0+i\de\cdot\om^K_0=z_1\d z_2-z_2\d z_1$ is a holomorphic (1,0)-form on $\C^2/\{\pm 1\}$, so that $B^*(\de\cdot \om^J_0+i\de\cdot\om^K_0)$ is a holomorphic (and hence smooth) 1-form on $X$. Thus there are unique smooth sections $\mu^J,\mu^K$ of $\breve V^*$ on  $\si^{-1}(L')\subseteq P$, such that for each $x\in L'$, under the identifications $\La^2\breve V^*\vert_{P_x}\cong \La^2T^*P_x$ and $\La^2T^*(\nu_x/\{\pm 1\})\cong\pi^*(\La^2\nu_x^*)$ we have $\mu^J\vert_{P_x}\cong \ha\rho^*(\de\cdot\pi^*(\hat\om^J\vert_x))$ and $\mu^K\vert_{P_x}\cong \ha\rho^*(\de\cdot\pi^*(\hat\om^K\vert_x))$. Pulling back \eq{nc4eq16} to $P_x$ gives
\e 
\label{nc4eq17}
\check\om^J\vert_{P_x}=\d(\mu^J\vert_{P_x})\quad\text{and}\quad \check\om^K\vert_{P_x}=\d(\mu^K\vert_{P_x}).
\e
Now using \eq{nc4eq6} and the fact that restriction of 2-forms in $\La^2\breve V^*\subset\La^2T^*P$ to $P_x$ is an isomorphism $\La^2\breve V^*\vert_{P_x}\ra \La^2T^*P_x$, we see that \eq{nc4eq17} implies
\e
\check\om^J=[\d\mu^J]_{2,0} \quad\text{and}\quad \check\om^K=[\d\mu^K]_{2,0}\quad\text{on $\si^{-1}(L')\subseteq P$.}
\label{nc4eq18}
\e

On $\C^2/\{\pm 1\}$ we have $\om_0^J\w\om_0^J=\om_0^K\w\om_0^K$, so contracting with $\de$ implies that $(\de\cdot\om_0^J)\w\om_0^J=(\de\cdot\om_0^K)\w\om_0^K$. Pulling this equation back to $\si^{-1}(L')$ as for $\check\om^J,\check\om^K,\mu^J,\mu^K$ implies that
\e
\mu^J\w\check\om^J=\mu^K\w\check\om^K.
\label{nc4eq19}
\e
Finally, on $\C^2/\{\pm 1\}$ we also have $\d[(\de\cdot\om_0^J)\w\om_0^J]=2\om_0^J\w\om_0^J$, so arguing as for \eq{nc4eq18}, using \eq{nc4eq15} and \eq{nc4eq19} we deduce that on $\si^{-1}(L')$ we have
\e
[\d(\mu^J\w\check\om^J)]_{4,0}=[\d(\mu^K\w\check\om^K)]_{4,0}=\check\om^I\w\check\om^I=\check\om^J\w\check\om^J=\check\om^K\w\check\om^K.
\label{nc4eq20}
\e

Recall we defined $\check g_{2,0}$ in \eq{nc4eq14}. Now we define a tensor $\check g_{0,2}$ in $S^2\breve V^*\subset S^2T^*P$ and exterior forms $\check\vp_{0,3},\ab\ldots,\ab\th_{2,2}$ on $\si^{-1}(L')\subset P$, where subscripts $(\cdots)_{i,j}$ indicate a form of type $(i,j)$, as follows. The motivation for the definitions will be apparent when we derive \eq{nc4eq22} after the definitions. We define:
\ea
\check g_{0,2}&=\si^*(e_1)^2+\si^*(e_2)^2+\si^*(e_3)^2,
\nonumber\\
\check\vp_{0,3}&=\si^*(e_1)\w\si^*(e_2)\w\si^*(e_3)=\si^*(\vol_L), 
\nonumber\\ 
\check\vp_{2,1}&=-\si^*(e_1)\w\check\om^I-\si^*(e_2)\w\check\om^J-\si^*(e_3)\w\check\om^K,
\nonumber\\
\check\psi_{2,2}&=-\si^*(e_2)\w\si^*(e_3)\w\check\om^I-\si^*(e_3)\w\si^*(e_1)\w\check\om^J-\si^*(e_1)\w\si^*(e_2)\w\check\om^K, 
\nonumber\\ 
\begin{split}
\check\psi_{4,0}&=\ha\check\om^I\w\check\om^I=\ha\check\om^J\w\check\om^J=\ha\check\om^K\w\check\om^K,\\
\xi_{1,2}&=-\si^*(\la)\w\ka_{1,1}+[\d(\si^*(e_2)\w\mu^J+\si^*(e_3)\w\mu^K)]_{1,2}, 
\end{split}
\label{nc4eq21}\\
\xi_{0,3}&=-\si^*(\la)\w\ka_{0,2}+[\d(\si^*(e_2)\w\mu^J+\si^*(e_3)\w\mu^K)]_{0,3} 
\nonumber\\
\chi_{1,3}&=-\si^*(*\la)\w\ka_{1,1}-[\d(\si^*(e_3)\w\si^*(e_1)\w\mu^J+\si^*(e_1)\w\si^*(e_2) \w\mu^K)]_{1,3},
\nonumber\\
\th_{3,1}&=\ts\frac{1}{2}[\d(\mu^J\w\check\om^J)]_{3,1}=\frac{1}{2}[\d(\mu^K\w\check\om^K)]_{3,1},
\nonumber\\
\th_{2,2}&=\ts\frac{1}{2}[\d(\mu^J\w\check\om^J)]_{2,2}=\frac{1}{2}[\d(\mu^K\w\check\om^K)]_{2,2}.  
\nonumber
\ea
Here the alternate expressions for $\check\psi_{4,0},\th_{3,1},\th_{2,2}$ come from \eq{nc4eq15} and \eq{nc4eq19}. Using equations \eq{nc4eq9}, \eq{nc4eq10}, \eq{nc4eq18}, \eq{nc4eq20}, \eq{nc4eq21} and $\d\vol_L=0$ we claim that
\e
\begin{aligned}
\d\check\vp_{0,3}&=0, &
\d\bigl[\check\vp_{2,1}+\xi_{1,2}+\xi_{0,3}\bigr]&=0,\\
\d\bigl[\check\psi_{2,2}+\chi_{1,3}\bigr]&=0, &
\d\bigl[\check\psi_{4,0}+\th_{3,1}+\th_{2,2}\bigr]&=0.
\end{aligned}
\label{nc4eq22}
\e
The first equation in \eq{nc4eq22} is obvious. We will derive the second equation. First, we write
\begin{align*}
\xi_{1,2} + \xi_{0,3} & = - \si^*(\la) \w (\ka_{1,1} + \ka_{0,2}) + \d ( \si^* (e_2) \w \mu^J + \si^* (e_3) \w \mu^K) \\
& \qquad {} - [\d ( \si^* (e_2) \w \mu^J + \si^* (e_3) \w \mu^K)]_{2,1} \\
& = - \si^*(\la) \w (\ka_{1,1} + \ka_{0,2}) + \d (\cdots) + \si^*(e_2) \w \check\om^J + \si^*(e_3) \w \check\om^K,
\end{align*}
using \eq{nc4eq18} in the last step. Thus we find
\begin{equation*}
\d (\check\vp_{2,1} + \xi_{1,2} + \xi_{0,3}) = - \d ( \si^*(\la) \w (\ka_{1,1} + \ka_{0,2}) + \si^*(e_1) \w \check\om^I) = 0
\end{equation*}
by \eq{nc4eq9}. The remaining two equations in \eq{nc4eq22} are proved similarly. 

As in \S\ref{nc41}, we are working with an arbitrary choice of $e_2,e_3$ on an open $L'\subseteq L$ with $(e_1,e_2,e_3)$ an oriented orthonormal basis of $T^*L'$. If $e'_2,e'_3$ are alternative choices for $e_2,e_3$ then $e_2'=\cos\phi\, e_2+\sin\phi \,e_3$, $e_3'=-\sin\phi\, e_2+\cos\phi\, e_3$ for some smooth $\phi:L'\ra\R/2\pi\Z$. Equation \eq{nc4eq2} expresses the corresponding $\hat\om^{\prime I},\hat\om^{\prime J},\hat\om^{\prime K},I',J',K'$ in terms of $\hat\om^{\prime I},\hat\om^{\prime J},\hat\om^{\prime K},I',J',K'$. From the definitions of $\check\om^J,\check\om^K,\mu^J,\mu^K$ above, we see that they transform by
\begin{align*}
\check\om^{\prime J}&=\cos\phi\, \check\om^J+\sin\phi\,\check\om^K, & \check\om^{\prime K}&=-\sin\phi\, \check\om^J+\cos\phi\,\check\om^K,\\
\mu^{\prime J}&=\cos\phi\, \mu^J+\sin\phi\,\mu^K, & \mu^{\prime K}&=-\sin\phi\, \mu^J+\cos\phi\,\mu^K.
\end{align*}
We can now check that all of $\check g_{0,2}, \ldots, \th_{2,2}$ in \eq{nc4eq21} are independent of the choice of $e_2,e_3$. Thus by covering $L$ by open $L'\subseteq L$ on which such a basis $(e_1,e_2,e_3)$ exists, we construct global $\check g_{0,2},\ldots,\th_{2,2}$ on all of $P$  satisfying \eq{nc4eq21} over each $L'$ in the cover. Then \eq{nc4eq22} holds over each $L'$ in the cover, so \eq{nc4eq22} holds on~$P$.

In a similar way to \eq{nc3eq21}, for all $t>0$ define a 3-form $\vp_t^P$, 4-form $\psi_t^P$ and metric $g_t^P$ on $P$ by
\e
\vp^P_t=\check\vp_{0,3}+t^2\check\vp_{2,1},\quad \psi_t^P=t^2\check\psi_{2,2}+t^4\check\psi_{4,0},\quad
g_t^P=\check g_{0,2}+t^2\check g_{2,0}.
\label{nc4eq23}
\e
Then over $\si^{-1}(L')$ as above, by \eq{nc4eq21} we have
\begin{align*}
\vp^P_t&=\si^*(e_1)\!\w\!\si^*(e_2)\!\w\!\si^*(e_3)\!-\!t^2\si^*(e_1)\!\w\!\check\om^I\!-\!t^2\si^*(e_2)\!\w\!\check\om^J\!-\!t^2\si^*(e_3)\!\w\!\check\om^K,\\
\psi_t^P&=\ha t^4\check\om^I\w\check\om^I - t^2\si^*(e_2)\w\si^*(e_3)\w\check\om^I \\
& \qquad {} - t^2\si^*(e_3)\w\si^*(e_1)\w\check\om^J - t^2\si^*(e_1)\w\si^*(e_2)\w\check\om^K,\\
g_t^P&=\si^*(e_1)^2+\si^*(e_2)^2+\si^*(e_3)^2+t^2\check g_{2,0}.
\end{align*}
Comparing these with \eq{nc2eq22} and noting that on the fibres $P_x$ of $\si:P\ra L$, the $\check\om^I,\check\om^J,\check\om^K,\check g_{2,0}$ correspond to $\hat\om^I_{A(x)},\hat\om^J,\hat\om^K,h_{A(x)}$ on the Eguchi--Hanson space $X$, we see that $(\vp_t^P,g_t^P)$ is a $\Gt$-structure on $P$, with 4-form $\psi_t^P= \Theta(\vp_t^P) = *_{g_t^P} \vp_t^P$.

Finally, we define on $P$ the following forms:
\e
\ti\vp_t^P=\vp_t^P+t^2\xi_{1,2}+t^2\xi_{0,3},\qquad
\ti\psi_t^P=\psi_t^P+t^2\chi_{1,3}+t^4\th_{3,1}+t^4\th_{2,2}.
\label{nc4eq24}
\e
Using equations \eq{nc4eq22} and \eq{nc4eq23} we find that $\ti\vp_t^P,\ti\psi_t^P$ are {\it closed} forms. However, we emphasize that $\Theta(\ti\vp_t^P) \neq \ti\psi_t^P$.

\subsection{\texorpdfstring{Comparing $\vp_t^P,\psi_t^P,\ti\vp_t^P,\ti\psi_t^P$ and $\vp_t^\nu,\psi_t^\nu,\ti\vp_t^\nu,\ti\psi_t^\nu$}{Comparing φᵗᴾ,ψᵗᴾ,~φᵗᴾ,~ψᵗᴾ and φᵗᵛ,ψᵗᵛ,~φᵗᵛ,~ψᵗᵛ}}
\label{nc45}

In \S\ref{nc33}--\S\ref{nc34} we defined a $\Gt$-structure $(\vp_t^\nu,g_t^\nu)$ with 4-form $\psi_t^\nu=*_{g_t^\nu}\vp_t^\nu$ on $\nu$ for $t>0$, and closed 3- and 4-forms $\ti\vp^\nu_t=\vp^\nu_t+t^2\dddot\vp^2_{0,3}$ and $\ti\psi^\nu_t=\psi^\nu_t+t^4\dddot\psi^4_{2,2}$ on $\nu$. These are all invariant under $\{\pm 1\}$, and so they descend to $\nu/\{\pm 1\}$. In \S\ref{nc44} we defined a $\Gt$-structure $(\vp_t^P,g_t^P)$ with 4-form $\psi_t^P=*_{g_t^P}\vp_t^P$ on $P$ for $t>0$, and closed 3- and 4-forms $\ti\vp_t^P=\vp_t^P+t^2\xi_{1,2}+t^2\xi_{0,3}$ and $\ti\psi_t^P=\psi_t^P+t^2\chi_{1,3}+t^4\th_{3,1}+t^4\th_{2,2}$ on $P$. We will now compare $\vp_t^P,\psi_t^P,\ti\vp_t^P,\ti\psi_t^P$ with the pullbacks of $\vp_t^\nu,\psi_t^\nu,\ti\vp_t^\nu,\ti\psi_t^\nu$ by $\rho:P\ra\nu/\{\pm 1\}$ on the region $\check r>1$ in~$P$, where the connections are compatible.

By the definition of $\rho$ in \S\ref{nc41}, the following map is a diffeomorphism:
\e
\rho\vert_{\check r>1}:\bigl\{p\in P:\check r(p)>1\bigr\}\longra\bigl\{z\in\nu/\{\pm 1\}: r(z)>1\bigr\}.
\label{nc4eq25}
\e
We have splittings $T(\nu/\{\pm 1\})=V\op H$ from \S\ref{nc32} and $TP=\breve V\op\breve H$ from \S\ref{nc42}. Here $H$ is defined using a connection $\ti\nabla^\nu$ on $\nu$, and $\breve H$ is also defined using $\ti\nabla^\nu$ in the region $\check r>1$. Therefore \eq{nc4eq25} identifies $\breve V,\breve H$ with~$V,H$.

As in \S\ref{nc41}, work with an oriented orthonormal basis $(e_1,e_2,e_3)$ of 1-forms on $L'\subseteq L$. Then in forms on $\nu/\{\pm 1\}$ we have
\ea
\vp^0_{0,3}&=\pi^*(e_1)\w\pi^*(e_2)\w\pi^*(e_3)=\pi^*(\vol_L),
\nonumber\\
\dot\vp^2_{2,1}&=-\pi^*(e_1)\w\hat\om^I-\pi^*(e_2)\w\hat\om^J-\pi^*(e_3)\w\hat\om^K,
\nonumber\\
\dot\psi^2_{2,2}&=-\pi^*(e_2)\w\pi^*(e_3)\w\hat\om^I-\pi^*(e_3)\w\pi^*(e_1)\w\hat\om^J-\pi^*(e_1)\w\pi^*(e_2)\w\hat\om^K,
\nonumber\\ 
\dot\psi^4_{4,0}&=\ha\hat\om^I\w\hat\om^I=\ha\hat\om^J\w\hat\om^J=\ha\hat\om^K\w\hat\om^K.
\label{nc4eq26}
\ea
In forms on $\bigl\{p\in P:\check r(p)>1\bigr\}$ we claim that
\e
\begin{split}
\check\om^I&=\rho\vert_{\check r>1}^*(\hat\om^I)-\ts\frac{1}{4}\bigl[\d\bigl(I([\d(G_{A\ci\si}(\check r))]_{1,0})\bigr)\bigr]_{2,0},\\
\check\om^J&=\rho\vert_{\check r>1}^*(\hat\om^J), \quad \check\om^K=\rho\vert_{\check r>1}^*(\hat\om^K), \quad \si^*(e_i)=\rho\vert_{\check r>1}^*(\pi^*(e_i)).
\end{split}
\label{nc4eq27}
\e
Here $G_a(r)=f_a(r)-r^2$ as in \eq{nc2eq17}. The first equation of \eq{nc4eq27} follows from \eq{nc2eq17}, \eq{nc4eq13}, and $\hat\om^I=-\ts\frac{1}{4}\bigl[\d\bigl(I([\d(r^2)]_{1,0})\bigr)\bigr]{}_{2,0}$. The second and third equations follow from the definitions of $\check\om^J,\check\om^K$ and $\rho^*(V^*)=\breve V^*\subset T^*P$ where $\check r>1$ in $P$. The fourth equation holds as~$\si=\pi\ci\rho$.

Now define a 2-form $\tau_{1,1}$ of type $(1,1)$ and a 3-form $\up_{1,2}$ of type $(1,2)$ on the open subset of $P$ where $\check r>1$ by
\e
\begin{split}
\tau_{1,1}&=-\ts\frac{1}{4}(A\ci\si)^{-1}\si^*(\la)\w I([\d(G_{A\ci\si}(\check r))]_{1,0}),\\
\up_{1,2}&=\ts\frac{1}{4}(A\ci\si)^{-1}\si^*(*\la)\w I([\d(G_{A\ci\si}(\check r))]_{1,0}).
\end{split}
\label{nc4eq28}
\e

Combining equations \eq{nc4eq21}, \eq{nc4eq26}, \eq{nc4eq27}, and \eq{nc4eq28} one finds that
\e
\begin{split}
\check\vp_{0,3}\vert_{\check r>1}&=\rho\vert_{\check r>1}^*(\vp^0_{0,3}),	\\
\check\vp_{2,1}\vert_{\check r>1}&=\rho\vert_{\check r>1}^*(\dot\vp^2_{2,1})+[\d\tau_{1,1}]_{2,1}, \\
\check\psi_{2,2}\vert_{\check r>1}&=\rho\vert_{\check r>1}^*(\dot\psi^2_{2,2})+[\d\up_{1,2}]_{2,2}, \\
\check\psi_{4,0}\vert_{\check r>1}&=\rho\vert_{\check r>1}^*(\dot\psi^4_{4,0}),\\
\end{split}
\label{nc4eq29}
\e
and consequently from \eq{nc3eq21} and \eq{nc4eq23} we obtain
\e
\vp_t^P\vert_{\check r>1}=\rho\vert_{\check r>1}^*(\vp_t^\nu)+t^2[\d\tau_{1,1}]_{2,1}, \quad
\psi_t^P\vert_{\check r>1}=\rho\vert_{\check r>1}^*(\psi_t^\nu)+t^2[\d\up_{1,2}]_{2,2}.
\label{nc4eq30}
\e

Similarly, working with an oriented orthonormal basis $(e_1,e_2,e_3)$ on $L'\subseteq L$ as in \S\ref{nc41}, from the definitions and equations above we can show that
\ea
\ti\vp_t^P\vert_{\check r>1}&=\!\check\vp_{0,3}\!-\!
t^2\d\bigl[\ts\si^*(e_1)\w I[\d (\ts\frac{1}{4}f_{A\ci\si}(\check r^2))]_{1,0}\!-\!
\si^*(e_2)\!\w\!\mu^J\!-\!\si^*(e_3)\!\w\!\mu^K
\bigr],
\nonumber\\
\rho\vert_{\check r>1}^*(\ti\vp_t^\nu)&=\check\vp_{0,3}-
t^2\d\bigl[\ts\si^*(e_1)\w I[\ts\frac{1}{4}\d(\check r^2)]_{1,0}-
\si^*(e_2)\w\mu^J-\si^*(e_3)\w\mu^K\bigr],
\nonumber\\
\begin{split}
\ti\psi_t^P\vert_{\check r>1}&=\d\bigl[t^4(\ts\frac{1}{2}\mu^J\w\check\om^J)+t^2\bigl(
\si^*(e_2\w e_3)\w I[\ts\frac{1}{4}\d (f_{A\ci\si}(\check r^2))]_{1,0}\\
&\qquad -\si^*(e_3\w e_1)\w\mu^J-\si^*(e_1\w e_2)\w\mu^K\bigr)\bigr],
\end{split}
\label{nc4eq31}\\
\rho\vert_{\check r>1}^*(\ti\psi_t^\nu)&=\d\bigl[t^4(\ts\frac{1}{2}\mu^J\w\check\om^J)+t^2\bigl(
\si^*(e_2\w e_3)\w I[\ts\frac{1}{4}\d(\check r^2)]_{1,0}
\nonumber\\
&\qquad -\si^*(e_3\w e_1)\w\mu^J-\si^*(e_1\w e_2)\w\mu^K\bigr)\bigr].
\nonumber
\ea
Combining $(A\ci\si)^{-1}\si^*(\la)=e_1$, $(A\ci\si)^{-1}\si^*(*\la)=e_2\w e_3$ and
equations \eq{nc2eq17}, \eq{nc4eq28}, and  \eq{nc4eq31}, we see that
\e
\ti\vp_t^P\vert_{\check r>1}=\rho\vert_{\check r>1}^*(\ti\vp_t^\nu)+t^2\d\tau_{1,1}, \quad
\ti\psi_t^P\vert_{\check r>1}=\rho\vert_{\check r>1}^*(\ti\psi_t^\nu)+t^2\d\up_{1,2}.
\label{nc4eq32}
\e
In the region $\check r>1$, from Corollary \ref{nc3cor}, equations \eq{nc4eq24}, \eq{nc4eq29}, \eq{nc4eq30}, \eq{nc4eq32}, and the fact that $\rho\vert_{\check r>1}$ takes $\breve V,\breve H$ to $V,H$, we see that
\e
\begin{gathered}
\xi_{1,2}\vert_{\check r>1}=[\d\tau_{1,1}]_{1,2},\quad
\xi_{0,3}\vert_{\check r>1}=\rho\vert_{\check r>1}^*(\dddot\vp^2_{0,3})+[\d\tau_{1,1}]_{0,3},\\
\chi_{1,3}\vert_{\check r>1}=[\d\up_{1,2}]_{1,3},\quad 
\th_{3,1}\vert_{\check r>1}=0,
\quad
\th_{2,2}\vert_{\check r>1}=\rho\vert_{\check r>1}^*(\dddot\psi^4_{2,2}).
\end{gathered}
\label{nc4eq33}
\e

Using \eq{nc2eq18} and \eq{nc4eq28} we can show that where $\check r>1$ on $P$ we have
\e
\begin{split}
&\bmd{\nabla^k(t^2\tau_{1,1})}_{g_t^P}=O(t^{1-k}\check r^{-3-k}),\quad \bmd{\nabla^k(t^2\up_{1,2})}_{g_t^P}=O(t^{1-k}\check r^{-3-k}),\\
&\bmd{\nabla^k(t^2[\d\tau_{1,1}]_{i,3-i})}_{g_t^P}=O(t^{2-i-k}\check r^{-2-i-k}), \\
&\bmd{\nabla^k(t^2[\d\up_{1,2}]_{i,4-i})}_{g_t^P}=O(t^{2-i-k}\check r^{-2-i-k}),
\end{split}
\label{nc4eq34}
\e
for all $i=0,1,2$ and $k=0,1,\ldots,$ where $\nabla$ is the Levi-Civita connection of $g_t^P$. Thus, \eq{nc4eq30} and \eq{nc4eq32} imply that $\vp_t^P,\psi_t^P,\ti\vp_t^P,\ti\psi_t^P$ are asymptotic to $\rho^*(\vp_t^\nu),\rho^*(\psi_t^\nu),\rho^*(\ti\vp_t^\nu),\rho^*(\ti\psi_t^\nu)$ as $\check r\ra\iy$ in~$P$.

\subsection{\texorpdfstring{Estimating the torsion of $(\vp^P_t,g^P_t)$}{Estimating the torsion of (φᵗᴾ,gᵗᴾ)}}
\label{nc46}

Just as for $(\vp^\nu_t,g^\nu_t)$ in the discussion after Corollary \ref{nc3cor}, we have constructed a $\Gt$-structure $(\vp^P_t,g^P_t)$ on $P$ with 4-form $\psi^P_t=*_{g^P_t}\vp^P_t$. To apply Theorem \ref{nc2thm3} in \S\ref{nc6}, we will need $(\vp^P_t,g^P_t)$ to have {\it small torsion}, in an appropriate sense. If $\d\vp^P_t=0$ and $\d\psi^P_t=0$ then $(\vp^P_t,g^P_t)$ would be torsion-free, by Theorem \ref{nc2thm1}. In general $\d\vp^P_t\ne 0$ and $\d\psi^P_t\neq 0$, but as $\d\ti\vp^P_t=0$ and $\d\ti\psi^P_t=0$ with $\ti\vp_t^P=\vp_t^P+t^2\xi_{1,2}+t^2\xi_{0,3}$ and $\ti\psi_t^P=\psi_t^P+t^2\chi_{1,3}+t^4\th_{3,1}+t^4\th_{2,2}$ by \eq{nc4eq24}, we can regard $\md{t^2\xi_{1,2}}_{g_t^P},\md{t^2\xi_{0,3}}_{g_t^P},\md{t^2\chi_{1,3}}_{g_t^P},\md{t^4\th_{3,1}}_{g_t^P},\md{t^4\th_{2,2}}_{g_t^P}$ as measuring the torsion of $(\vp^P_t,g^P_t)$. Thus, the next proposition estimates the torsion of~$(\vp^P_t,g^P_t)$.

\begin{prop} In the situation of\/ {\rm\S\ref{nc44},} for all\/ $k=0,1,\ldots$ we have
\ea
\bmd{\nabla^k(t^2\xi_{1,2})}_{g_t^P}&= \begin{cases} O(t^{1-k}), & \check r\le 1, \\ O(t^{1-k}\check r^{-3-k}), & \check r\ge 1, \end{cases} 
\label{nc4eq35} \allowdisplaybreaks\\
\bmd{\nabla^k(t^2\xi_{0,3})}_{g_t^P}&=\begin{cases} O(t^{2-k}), & \check r\le 1, \\ O(t^{2-k}\check r^{2-k}), & \check r\ge 1, \end{cases} 
\label{nc4eq36} \allowdisplaybreaks\\
\bmd{\nabla^k(t^2\chi_{1,3})}_{g_t^P}&=\begin{cases} O(t^{1-k}), & \check r\le 1, \\ O(t^{1-k}\check r^{-3-k}), & \check r\ge 1, \end{cases} 
\label{nc4eq37} \allowdisplaybreaks\\
\bmd{\nabla^k(t^4\th_{3,1})}_{g_t^P}&= \begin{cases} O(t^{1-k}), & \check r\le 1, \\ 0, & \check r\ge 1, \end{cases} 
\label{nc4eq38} \allowdisplaybreaks\\
\bmd{\nabla^k(t^4\th_{2,2})}_{g_t^P}&=\begin{cases} O(t^{2-k}), & \check r\le 1, \\ O(t^{2-k}\check r^{2-k}), & \check r\ge 1. \end{cases} 
\label{nc4eq39}
\ea
\label{nc4prop}	
\end{prop}

\begin{proof} The $\check r\le 1$ estimates hold because if $\al_{i,j}$ is a form or tensor of type $(i,j)$ on $P$ then $\md{\al_{i,j}}_{g_t^P}=O(t^{-i})$ on the compact region $\check r\le 1$ in $P$, by the definition \eq{nc4eq23} of $g_t^P$. The $\check r\ge 1$ estimates follow from \eq{nc4eq33}--\eq{nc4eq34} and the estimates $\md{\nabla^k(t^2\dddot\vp^2_{0,3})}_{g^P_t}=O(t^{2-k}\check r^{2-k})$ and $\md{\nabla^k(t^4\dddot\psi^4_{2,2})}_{g^P_t}=O(t^{2-k}\check r^{2-k})$, which both follow from Corollary \ref{nc3cor}.
\end{proof}

\section{\texorpdfstring{Correcting for the leading-order errors on $P$}{Correcting for the leading-order errors on P}}
\label{nc5}

\subsection{The correction theorem}
\label{nc51}

Our goal in \S\ref{nc6} is to glue together the $\Gt$-structures $(\vp, g)$ on $M/\an{\io}$, and $(\ti\vp^\nu_t, \ti g^\nu_t)$ on $\nu/\{\pm 1\}$, and $(\ti\vp^P_t, \ti g^P_t)$ on $P$, to get a closed $\Gt$-structure $(\vp_t^N, g_t^N)$ with small torsion on a compact 7-manifold $N$, and then apply Theorem \ref{nc2thm3} to show that $(\vp_t^N, g_t^N)$ has a small deformation to a torsion-free $\Gt$-structure on $N$ for small $t$. To do this we must ensure that the 4-form $\Th(\vp_t^N)-\psi_t^N$ which is a substitute for the torsion of $(\vp_t^N, g_t^N)$ satisfies several estimates, including $\nm{\Th(\vp_t^N)-\psi_t^N}_{L^2}\le Kt^{\frac{7}{2}+\al}$ for~$\al,K>0$.

Now in Proposition \ref{nc4prop}, the bounds for $\md{t^2\xi_{0,3}}_{g_t^P},$ $\md{t^4\th_{2,2}}_{g_t^P}$ will contribute $O(t^4)$ to $\nm{\Th(\vp_t^N)-\psi_t^N}_{L^2}$, which is good, as $4>\frac{7}{2}$. However, the bounds for $\md{t^2\xi_{1,2}}_{g_t^P}$, $\md{t^2\chi_{1,3}}_{g_t^P}$, $\md{t^4\th_{3,1}}_{g_t^P}$ would contribute $O(t^3)$ to $\nm{\Th(\vp_t^N)-\psi_t^N}_{L^2}$, which is bad, as the error is too large to apply Theorem \ref{nc2thm3}. In Theorem \ref{nc5thm} we construct forms on $P$ which will be used in \S\ref{nc6} to cancel out the leading-order error terms in $t^2\xi_{1,2},t^2\chi_{1,3},t^4\th_{3,1}$, and so make $\nm{\Th(\vp_t^N)-\psi_t^N}_{L^2}$ small enough to apply Theorem \ref{nc2thm3}.

\begin{thm} There exist\/ $2$-forms $\al_{0,2},\al_{2,0}$ and\/ $3$-forms\/ $\be_{0,3},\be_{2,1}$ on $P,$ where $\al_{i,j},\be_{i,j}$ are of type $(i,j),$ satisfying for all\/ $t>0$ the equation
\e
\begin{aligned}
& (D_{\vp^P_t} \Th) \bigl( t^2 [\d\al_{0,2}]_{1,2} + t^4 [\d\al_{2,0}]_{3,0} + t^2\xi_{1,2} \bigr) \\
& \qquad = t^2 \d \be_{0,3} + t^4 [ \d \be_{2,1} ]_{3,1} + t^2 \chi_{1,3} + t^4 \th_{3,1}.
\end{aligned}
\label{nc5eq1}
\e
Note that in \eq{nc5eq1} we have $\d \be_{0,3} = [ \d \be_{0,3} ]_{1,3}$ automatically, so all the derivatives in \eq{nc5eq1} are in fact vertical (fibre) derivatives. Moreover, for $\gamma > 0$ sufficiently small and for all $k \geq 0$, these forms satisfy the following estimates:
\ea
\bmd{\nabla^k(t^2\al_{0,2})}_{g_t^P} & = \begin{cases} O(t^{2-k}), & \check r \le 1, \\ O(t^{2-k} \check r^{-2-k+\ga}), & \check r \ge 1, \end{cases} 
\label{nc5eq2}
\allowdisplaybreaks\\
\bmd{\nabla^k(t^2[\d\al_{0,2}]_{i,3-i})}_{g_t^P} & = \begin{cases} O(t^{2-i-k}), & \check r \le 1,\; i=0,1, \\[4pt] O(t^{2-i-k} \check r^{-2-i-k+\ga}), & \check r\ge 1,\; i=0,1, \end{cases} 
\label{nc5eq3}
\allowdisplaybreaks\\
\bmd{\nabla^k(t^4\al_{2,0})}_{g_t^P} & = \begin{cases} O(t^{2-k}), & \check r \le 1, \\ O(t^{2-k} \check r^{-2-k+\ga}), & \check r \ge 1, \end{cases} 
\label{nc5eq4}
\allowdisplaybreaks\\
\bmd{\nabla^k(t^4[\d\al_{2,0}]_{i,3-i})}_{g_t^P} & = \begin{cases} O(t^{4-i-k}), & \check r \le 1,\; i=1,2,3, \\[4pt] O(t^{4-i-k} \check r^{-i-k+\ga}), & \check r \ge 1,\; i=1,2,3, \end{cases} 
\label{nc5eq5}
\allowdisplaybreaks\\
\bmd{\nabla^k(t^2\be_{0,3})}_{g_t^P} & = \begin{cases} O(t^{2-k}), & \check r \le 1, \\ O(t^{2-k}\check r^{-2-k+\ga}), & \check r \ge 1, \end{cases} 
\label{nc5eq6}
\allowdisplaybreaks\\
\bmd{\nabla^k(t^4\be_{2,1})}_{g_t^P} & = \begin{cases} O(t^{2-k}), & \check r \le 1, \\ O(t^{2-k}\check r^{-2-k+\ga}), & \check r \ge 1, \end{cases} 
\label{nc5eq7}
\allowdisplaybreaks\\
\bmd{\nabla^k(t^4[\d\be_{2,1}]_{i,j})}_{g_t^P} & = \begin{cases} O(t^{4-i-k}), & \check  r \le 1,\; i=1,2,3, \\[4pt] 
O(t^{4-i-k} \check r^{-i-k+\ga}), & \check r \ge 1,\; i=1,2,3, \end{cases} 
\label{nc5eq8}
\ea

If we take $\al_{2,0}$ and\/ $\be_{2,1}$ to be self-dual in the fibre directions, that is, $\al_{2,0}\in\Ga^\iy(\La^2_+V^*)$ and\/ $\be_{2,1}\in\Ga^\iy(\La^2_+V^*\ot H^*),$ then $\al_{0,2},\al_{2,0},\be_{0,3},\be_{2,1}$ are unique.
\label{nc5thm}	
\end{thm}

The remainder of \S\ref{nc5} is taken up with the proof of Theorem \ref{nc5thm}. In \S\ref{nc52} we establish several preliminary results needed for the proof. In \S\ref{nc53} we collect various facts about analysis on the Eguchi--Hanson fibres. Finally, in \S\ref{nc54} we complete the proof of Theorem \ref{nc5thm} in three steps.

\subsection{Preliminary results needed for the correction theorem}
\label{nc52}

In this section we collect several preliminary results that will be used to prove Theorem \ref{nc5thm}, including several results that use properties of $\Gt$-structures.

As in \S\ref{nc41}, we make an arbitrary choice of $e_2,e_3$ on an open $L'\subseteq L$ with $(e_1,e_2,e_3)$ an oriented orthonormal basis of $T^*L'$. All our computations in this section will be independent of this choice. We now introduce some notation that will be used only within \S\ref{nc5}, to simplify our expressions.  First, we will abuse notation and write $e_k$ for $\sigma^* (e_k)$, thinking of $e_k$ as a horizontal $1$-form on $P$ over $L'$. Moreover, over $L'$, a vertical $k$-form on $P$ can be thought of as a $k$-form on $X$ smoothly parametrized by~$x \in L'$. 

We will use $\dx,\dsx,\stx,\gx$, and $\volx$ to denote the exterior derivative, its formal adjoint, the Hodge star, the metric, and the volume form on the fibre $P_x \cong X$ over $x \in L'$. Note that $\dx = \d^{1,0}$ is the $\mathrm{(1,0)}$ part of $\d$. That is, $\dx$ is the exterior derivative in the vertical (fibre) direction. We have $\stx^2 = (-1)^k$ on $\Omega^k (P_x)$ and $\dsx = - \stx \! \dx \, \stx$ on $\Omega^{\bullet} (P_x)$. Finally, we will denote the hyperK\"ahler forms $\check \om^I,\check \om^J,\check \om^K$ by $\om_1,\om_2,\om_3$, respectively, and their associated complex structures $I,J,K$, by $J_1,J_2,J_3$. More precisely, the forms $\omega_1,\omega_2$, and $\omega_3$ are a hyperK\"ahler triple on $P_x \cong X$ for each $x \in L'$. Recall from \eq{nc2eq8} that $J_1 J_2 = -J_3 = -J_2 J_1$ on 1-forms, and cyclic permutations of this.

\begin{lem} Let\/ $\alpha \in \Omega^1 (P_x)$ and\/ $\gamma \in \Omega^3 (P_x)$. Then we have
\e
\omega_i \w (\stx \dx \alpha) = - \big( \dsx (J_i \alpha) \big) \volx, \qquad \stx (\dx \gamma) = \dsx (\stx \gamma). 
\label{nc5eq9}
\e
Moreover, if\/ $f$ is any function, then $\dsx (J_k \dx f) = 0$.
\label{nc5lem1}
\end{lem}

\begin{proof}
The second equation is immediate from $\stx^2 = -1$ on 3-forms and $\dsx = -\stx \! \dx \, \stx$. For the first equation, using $\stx \omega_i = \omega_i$ and $\dx \omega_i = 0$, and \eq{nc2eq12} we compute
\begin{align*}
\omega_i \w (\stx \dx \alpha) & = (\dx \alpha) \w (\stx \omega_i) = (\dx \alpha) \w \omega_i = \dx( \alpha \w \omega_i) \\
& = - \dx (\stx \stx (\alpha \w \omega_i)) = \dx \stx (J_i \alpha) = \stx (\stx \dx \stx) (J_i \alpha) \\
& = -\stx \big( \dsx (J_i \alpha) \big) = -\big( \dsx (J_i \alpha) \big) \volx
\end{align*}
as claimed. Finally, from \eq{nc5eq9}, for any function $f$ we have $- \dsx (J_k \dx f) \w \volx = \om_k \w \stx (\dx^2 f) = 0$, so $\dsx (J_k \dx f) = 0$.
\end{proof}

Consider now the $\Gt$-structure $(\vp_t^P, g_t^P)$ of \eq{nc4eq23}. In the remainder of \S\ref{nc5} we will drop the superscript $P$ for simplicity. Thus we write
\e
\begin{aligned}
\vp_t & = e_1 \w e_2 \w e_3 - t^2 (\om_1 \w e_1 + \om_2 \w e_2 + \om_3 \w e_3), \\
\psi_t & = t^4 \volx - t^2 (\om_1 \w e_2 \w e_3 + \om_2 \w e_3 \w e_1 + \om_3 \w e_1 \w e_2), \\
g_t & = (e_1)^2 + (e_2)^2 + (e_3)^2 + t^2 \gx.
\end{aligned}
\label{nc5eq10}
\e
At a point in $P$ lying over $x \in L'$, the metric $g_t^P$ is the Riemannian product of the flat metric on $\R^3 \cong \mathrm{span} \{ e_1, e_2, e_3 \}$ with the metric $\gx$. Let $\ast_t$ and $\vol_t$ denote the Hodge star and volume form of $(\vp_t, g_t)$. From $\vol_t = t^4 e_1 \w e_2 \w e_3 \w \volx$ it is easy to deduce that
\e
\ast_t ( \alpha \w \beta ) = (-1)^{kl} t^{4-2k} (\stx \alpha) \w (\ast_{\R^3} \beta) \quad \begin{matrix} \text{when $\alpha$ is a vertical $k$-form} \\
\text{and $\beta$ is an $l$-form on $\R^3$.} 
\end{matrix}
\label{nc5eq11}
\e

Recall that in \S\ref{nc44} we defined a closed $3$-form $\ti\vp_t^P$ and a closed $4$-form $\ti\psi_t^P$ by
\e
\ti\vp_t^P=\vp_t+t^2\xi_{1,2}+t^2\xi_{0,3},\qquad
\ti\psi_t^P=\psi_t+t^2\chi_{1,3}+t^4\th_{3,1}+t^4\th_{2,2}.
\label{nc5eq12}
\e

Consider the forms $\xi_{1,2}$, $\chi_{1,3}$, and $\th_{3,1}$, of equation \eq{nc5eq12}. With respect to the local frame $(e_1, e_2, e_3)$ we write
\e
\begin{aligned}
\xi_{1,2} & = \xi_1 \w e_2 \w e_3 + \xi_2 \w e_3 \w e_1 + \xi_3 \w e_1 \w e_2, \\
\chi_{1,3} & = \chi \w e_1 \w e_2 \w e_3, \\
\th_{3,1} & = \th_1 \w e_1 + \th_2 \w e_2 + \th_3 \w e_3,
\end{aligned}
\label{nc5eq13}
\e
where $\xi_1,\xi_2,\xi_3,\chi$ are vertical $1$-forms, and $\th_1,\th_2,\th_3$ are vertical $3$-forms. Over the open set $L'$, they can thus be viewed as forms on $X$ smoothly parametrized by $x \in L'$. In this section we will write $[(p,q)]$ for a form of type $(p,q)$ whose explicit expression will turn out to be irrelevant for us. Taking $\d$ of the expressions in \eq{nc5eq13}, we have
\e
\begin{aligned}
\d \xi_{1,2} & = \underbrace{(\dx \xi_1) \w e_2 \w e_3 + (\dx \xi_2) \w e_3 \w e_1 + (\dx \xi_3) \w e_1 \w e_2}_{\text{ type (2,2) }} + [(1,3)], \\
\d \chi_{1,3} & = \underbrace{(\dx \chi) \w e_1 \w e_2 \w e_3}_{\text{ type (2,3) }}, \\
\d \th_{3,1} & = \underbrace{(\dx \th_1) \w e_1 + (\dx \th_2) \w e_2 + (\dx \th_3) \w e_3}_{\text{type (4,1) }} + [(3,2)] + [(2,3)].
\end{aligned}
\label{nc5eq14}
\e

\begin{prop} The forms $\xi_1,\xi_2,\xi_3,\chi,\th_1,\th_2,\th_3$ satisfy the following three relations:
\e
\begin{aligned}
\dsx(\stx \theta_1) + \dsx(J_1 \chi) & = \dsx(J_2 \xi_3) - \dsx(J_3 \xi_2), \\
\dsx(\stx \theta_2) + \dsx(J_2 \chi) & = \dsx(J_3 \xi_1) - \dsx(J_1 \xi_3), \\
\dsx(\stx \theta_3) + \dsx(J_3 \chi) & = \dsx(J_1 \xi_2) - \dsx(J_2 \xi_1).
\end{aligned}
\label{nc5eq15}
\e
\label{nc5prop1}
\end{prop}

\begin{proof}
We have $\psi_t = \ast_t \vp_t = \Theta (\vp_t)$, and therefore the pair $(\vp_t, \psi_t)$ satisfies the fundamental relation
\e
\vp_t \w (\ast_t \d \vp_t) = \psi_t \w (\ast_t \d \psi_t).
\label{nc5eq16}
\e
Equation \eq{nc5eq16} is satisfied by any pair $(\vp, \psi) \in \Omega^3_+ \oplus \Omega^4_+$ with $\psi = \Theta(\vp)$, and is a characterization of the fact that the components $\pi_7 (\d \vp)$ and $\pi_7 (\d \psi)$ of the torsion that are $7$-dimensional can be identified by a $\Gt$-equivariant isomorphism. A proof of \eq{nc5eq16} can be found in \cite[Theorem 2.23]{Kari1}. (Note that even though \cite{Kari1} uses the opposite orientation convention, the identity \eq{nc5eq16} is independent of the choice of orientation convention.) We will show that \eq{nc5eq16} implies the three equations in \eq{nc5eq15}.

Since $\ti \vp_t^P$ and $\ti \psi_t^P$ are both closed, taking $\d$ of both expressions in \eq{nc5eq12} and substituting \eq{nc5eq14} gives
\begin{align*}
\d \vp_t & = - t^2 \d \xi_{1,2} - t^2 \d \xi_{0,3} \\
& = - t^2 \underbrace{\big( (\dx \xi_1) \w e_2 \w e_3 + (\dx \xi_2) \w e_3 \w e_1 + (\dx \xi_3) \w e_1 \w e_2 \big)}_{\text{ type (2,2) }} + t^2 [(1,3)]
\end{align*}
and
\begin{align*}
\d \psi_t & = - t^2 \d \chi_{1,3} - t^4 \th_{3,1} - t^4 \th_{2,2} \\
& = - t^2 \underbrace{\big( (\dx \chi) \w e_1 \w e_2 \w e_3 \big)}_{\text{ type (2,3) }} - t^4 \underbrace{\big( (\dx \th_1) \w e_1 + (\dx \th_2) \w e_2 + (\dx \th_3) \w e_3 \big)}_{\text{ type (4,1) }} \\
& \qquad {}+ t^4 [(3,2)] + t^4 [(2,3)].
\end{align*}
Applying $\ast_t$ to both expressions and using \eq{nc5eq11} yields
\begin{align*}
\ast_t (\d \vp_t) & = - t^2 \underbrace{\big( (\stx \dx \xi_1) \w e_1 + (\stx \dx \xi_2) \w e_2 + (\stx \dx \xi_3) \w e_3 \big)}_{\text{ type (2,1) }} + t^4 [(3,0)], \\
\ast_t (\d \psi_t) & = - \underbrace{\big( (\stx \dx \th_1) \w e_2 \w e_3 + (\stx \dx \th_2) \w e_3 \w e_1 + (\stx \dx \th_3) \w e_1 \w e_2 \big)}_{\text{ type (0,2) }} \\
& \qquad {}- t^2 \underbrace{(\stx \dx \chi)}_{\text{ type (2,0) }} + t^2 [(1,1)] + t^4 [(2,0)].
\end{align*}
Computing $\vp_t \w (\ast_t \d \vp_t)$ and $\psi_t \w (\ast_t \d \psi_t)$ using \eq{nc5eq10} and keeping track of types, we find
\begin{align*}
\vp_t \w (\ast_t \d \vp_t) & = t^4 \underbrace{\Big( \sum_{i,j=1}^3 \omega_i \w (\stx \dx \xi_j) \w e_i \w e_j \Big)}_{\text{ type (4,2) }} + t^4 [(3,3)], \\
\psi_t \w (\ast_t \d \psi_t) & = - t^4 \underbrace{\Big( \sum_{\substack{i,j,k \\ \text {cyclic}}} (\stx \dx \th_k) \w \volx \w e_i \w e_j \Big)}_{\text{ type (4,2) }} \\
& + t^4 \underbrace{\Big( \sum_{\substack{i,j,k \\ \text {cyclic}}} (\stx \dx \chi) \w \om_k \w e_i \w e_j \Big)}_{\text{ type (4,2) }}  + t^6 [(4,2)] + t^4 [(3,3)].
\end{align*}
Equating the terms of the form $t^4 [(4,2)]$ above, the fundamental relation \eq{nc5eq16} yields the three equations
\begin{align*}
\om_2 \w (\stx \dx \xi_3) - \om_3 \w (\stx \dx \xi_2) & = - (\stx \dx \th_1) \w \volx + (\stx \dx \chi) \w \om_1, \\
\om_3 \w (\stx \dx \xi_1) - \om_1 \w (\stx \dx \xi_3) & = - (\stx \dx \th_2) \w \volx + (\stx \dx \chi) \w \om_2, \\
\om_1 \w (\stx \dx \xi_2) - \om_2 \w (\stx \dx \xi_1) & = - (\stx \dx \th_3) \w \volx + (\stx \dx \chi) \w \om_3.
\end{align*}
Applying \eq{nc5eq9} now gives \eq{nc5eq15} as claimed.
\end{proof}

We will also need explicit formulas for $D_{\vp_t} \Th$ acting on forms of type $(3,0)$ and $(1,2)$.

\begin{lem} Let\/ $\pi_7$ be the projection $\pi_7 : \Omega^3 \to \Omega^3_7$ with respect to the $\Gt$-structure $\vp_t$ of\/ \eq{nc5eq10}. If\/ $\gamma_{3,0} = \eta$ is a vertical\/ $3$-form, then
\ea
& \pi_7 (\gamma_{3,0}) = 
\label{nc5eq17}\\
& \frac{1}{4} \eta - \frac{t^{-2}}{4} \big( (J_1 \stx \eta) \w e_2 \w e_3 + (J_2 \stx \eta) \w e_3 \w e_1 + (J_3 \stx \eta) \w e_1 \w e_2 \big).
\nonumber
\ea
If\/ $\gamma_{1,2} = \ze_1 \w e_2 \w e_3 + \ze_2 \w e_3 \w e_1 + \ze_3 \w e_1 \w e_2$ where the $\ze_k$'s are vertical $1$-forms, then
\ea
& \pi_7 (\gamma_{1,2}) = 
\label{nc5eq18}\\
& {} - \frac{t^2}{4} \big( (\stx J_1 \ze_1) + (\stx J_2 \ze_2) + (\stx J_3 \ze_3) \big) + \frac{1}{4} (\ze_1 + J_3 \ze_2 - J_2 \ze_3) \w e_2 \w e_3 
\nonumber\\
& \qquad {} + \frac{1}{4} (- J_3 \ze_1 + \ze_2 + J_1 \ze_3) \w e_3 \w e_1 + \frac{1}{4} ( J_2 \ze_1 - J_1 \ze_2 + \ze_3) \w e_1 \w e_2.
\nonumber
\ea

\label{nc5lem2}
\end{lem}

\begin{proof} The $\Gt$-structure $\vp_t$ satisfies
\e
\ast_t (\vp_t \w \ast_t (\vp_t \w \alpha) ) = - 4 \alpha \qquad \text{ for any $1$-form $\alpha$}.
\label{nc5eq19}
\e
This identity in fact holds for any $\Gt$-structure \cite[Proposition A.3]{Kari1}. (Note that even though \cite{Kari1} uses the opposite orientation convention, the identity \eq{nc5eq19} is independent of the choice of orientation convention.) If $\gamma \in \Omega^3$, then $\pi_7 \gamma = \ast_t (\vp_t \w \alpha)$ for some $1$-form $\alpha$. Since $\vp_t \w \gamma = \vp_t \w \pi_7 \gamma$, applying equation \eq{nc5eq19} yields
\begin{equation*}
\pi_7 \gamma = - \ts\frac{1}{4} \ast_t ( \vp_t \w \ast_t (\vp_t \w \gamma) ).
\end{equation*}
Let $\gamma_{3,0} = \eta$ be a vertical $3$-form. Then \eq{nc5eq10} and \eq{nc5eq11} give
\e
\ast_t (\vp_t \w \gamma_{3,0}) = t^{-2} \stx \eta.
\label{nc5eq20}
\e
Similarly if $\gamma_{1,2} = \ze_1 \w e_2 \w e_3 + \ze_2 \w e_3 \w e_1 + \ze_3 \w e_1 \w e_2$ where the $\ze_k$'s are vertical $1$-forms, we obtain
\e
\ast_t (\vp_t \w \gamma_{1,2}) = - \stx( \om_1 \w \ze_1 + \om_2 \w \ze_2 + \om_3 \w \ze_3) = J_1 \ze_1 + J_2 \ze_2 + J_3 \ze_3,
\label{nc5eq21}
\e
where we have also used \eq{nc2eq12}. Note that the right hand sides of both \eq{nc5eq20} and \eq{nc5eq21} are forms of type (1,0). If $\beta$ is of type (1,0), then again using \eq{nc5eq11} we compute
\begin{align*}
\ast_t (\vp_t \w \beta) & = \ast_t \big( e_1 \w e_2 \w e_3 \w \beta - t^2( \om_1 \w e_1 + \om_2 \w e_2 + \om_3 \w e_3) \w \beta \big) \\
& = t^2 (\stx \beta) - \stx(\om_1 \w \beta) \w e_2 \w e_3 \\
& \qquad{} - \stx(\om_2 \w \beta) \w e_3 \w e_1 - \stx(\om_3 \w \beta) \w e_1 \w e_2 \\
& = t^2 (\stx \beta) + (J_1 \beta) \w e_2 \w e_3 + (J_2 \beta) \w e_3 \w e_1 + (J_3 \beta) \w e_1 \w e_2.
\end{align*}
By \eq{nc5eq20} and \eq{nc5eq19}, to compute $\pi_7 (\gamma_{3,0})$ we substitute $\beta = - \frac{1}{4} t^{-2} \stx \eta$ into the above expression. This yields \eq{nc5eq17}, since $\stx^2 \eta = - \eta$. Similarly, by \eq{nc5eq21} and \eq{nc5eq19}, to compute $\pi_7 (\gamma_{1,2})$ we substitute $\beta = - \frac{1}{4} (J_1 \ze_1 + J_2 \ze_2 + J_3 \ze_3)$ into the above expression. This yields \eq{nc5eq18}, where we have also used \eq{nc2eq10}.
\end{proof}

\begin{cor} Let\/ $D_{\vp_t} \Th$ be the linearization of\/ $\Th$ at $\vp_t$ as in \eq{nc2eq6}. If\/ $\gamma_{3,0} = \eta$ is a vertical\/ $3$-form, then
\ea
& (D_{\vp_t} \Th) (\gamma_{3,0}) =
\label{nc5eq22} \\
& \qquad{} -\frac{t^{-2}}{2} (\stx \eta) \w e_1 \w e_2 \w e_3 + \frac{1}{2} \big( (J_1 \eta) \w e_1 + (J_2 \eta) \w e_2 + (J_3 \eta) \w e_3 \big).
\nonumber
\ea
If\/ $\gamma_{1,2} = \ze_1 \w e_2 \w e_3 + \ze_2 \w e_3 \w e_1 + \ze_3 \w e_1 \w e_2$ where the $\ze_k$'s are vertical\/ $1$-forms, then
\e
\begin{aligned}
& (D_{\vp_t} \Th) (\gamma_{1,2}) =\\
& \frac{1}{2} ( J_1 \ze_1 + J_2 \ze_2 + J_3 \ze_3 ) \w e_1 \w e_2 \w e_3 + \frac{t^2}{2} \stx (-\ze_1 + J_3 \ze_2 - J_2 \ze_3) \w e_1 \\
& \qquad {} + \frac{t^2}{2} \stx (- J_3 \ze_1 - \ze_2 + J_1 \ze_3) \w e_2 + \frac{t^2}{2} \stx ( J_2 \ze_1 - J_1 \ze_2 - \ze_3) \w e_3.
\end{aligned}
\label{nc5eq23}
\e
\label{nc5cor1}
\end{cor}

\begin{proof} From \eq{nc2eq6} we get $D_{\vp_t} \Th = \frac{7}{3} \ast_t \pi_1 + 2 \ast_t \pi_7 - \ast_t$. If $\gamma \in \Omega^3$, then $\pi_1 \gamma = f \vp_t$ for some function $f$. Then $\gamma \w \psi_t = (\pi_1 \gamma) \w \psi_t = 7 f \vol_t$, so $\pi_1 \gamma = \frac{1}{7} \ast_t (\gamma \w \psi_t)$. Since $\psi_t$ is of type (4,0) + (2,2), it follows that $\pi_1 = 0$ on forms of type $(3,0) + (1,2)$. Hence $D_{\vp_t} \Th = 2 \ast_t \pi_7 - \ast_t$ on forms of type $(3,0) + (1,2)$. Equations \eq{nc5eq22} and \eq{nc5eq23} now follow from \eq{nc5eq17} and \eq{nc5eq18} using \eq{nc5eq11} and the fact that $\stx$ commutes with $J_1$, $J_2$, and $J_3$.
\end{proof}

\begin{prop} Let\/ $\al_{0,2}$ and\/ $\al_{2,0}$ be $(0,2)$ and $(2,0)$ forms, respectively. Then the $5$-form
\e
\upsilon = \dx \big[ (D_{\vp_t} \Theta) (t^2 \dx \al_{0,2} + t^4 \dx \al_{2,0} + t^2 \xi_{1,2}) - t^2 \chi_{1,3} - t^4 \th_{3,1} \big]
\label{nc5eq24}
\e
lies in $\Omega^5_{14}$ with respect to $\vp_t$.
\label{nc5prop2}
\end{prop}

\begin{proof}
We need to show that $\psi_t \w (\ast_t \upsilon) = 0$. We write
\begin{equation*}
\al_{0,2} = a_1 e_2 \w e_3 + a_2 e_3 \w e_1 + a_3 e_1 \w e_2, \qquad \al_{2,0} = b,
\end{equation*}
where $a_1,a_2,a_3$ are functions and $b$ is a vertical $2$-form. Then we have
\ea
\dx \al_{0,2} & = [\d \al_{0,2}]_{1,2} = (\dx a_1) \w e_2 \w e_3 + (\dx a_2) \w e_3 \w e_1 + (\dx a_3) \w e_1 \w e_2, 
\nonumber\\
\dx \al_{2,0} & = [\d \al_{2,0}]_{3,0} = \dx b.
\label{nc5eq25}
\ea
For $k = 1, 2, 3,$ let
\e 
\ze_k = \dx a_k + \xi_k,
\label{nc5eq26}
\e
where $\xi_k$ was defined in \eq{nc5eq13}. Using the expressions in \eq{nc5eq25} and \eq{nc5eq13} we compute using Corollary \ref{nc5cor1} that
\begin{align*}
& (D_{\vp_t} \Theta)(t^2 \dx \al_{0,2} \!+\! t^4 \dx \al_{2,0} \!+\! t^2 \xi_{1,2}) = (D_{\vp_t} \Theta) \Big( t^2 \sum_{\substack{i,j,k \\ \text {cyclic}}} \ze_k \w e_i \w e_k + t^4 (\dx b) \Big) \\
& = \frac{t^2}{2} (J_1 \ze_1 + J_2 \ze_2 + J_3 \ze_3) \w e_1 \w e_2 \w e_3 + \frac{t^4}{2}  \sum_{\text{$i,j,k$ cyclic}\!\!\!\!\!\!\!\!\!\!\!\!\!\!\!\!} \stx (- \ze_i + J_k \ze_j - J_j \ze_k) \w e_i \\
& \quad {}- \frac{t^2}{2} (\stx \dx b) \w e_1 \w e_2 \w e_3 + \frac{t^4}{2} \big( (J_1 \dx b) \w e_1 + (J_2 \dx b) \w e_2 + (J_3 \dx b) \w e_3 \big).
\end{align*}
Collecting terms above and using \eq{nc5eq13} to express $\chi_{1,3}$ and $\th_{3,1}$ in terms of $\chi,\th_1,\th_2,\th_3$, we obtain
\e
\begin{aligned}
& (D_{\vp_t} \Theta)(t^2 \dx \al_{0,2} + t^4 \dx \al_{2,0} + t^2 \xi_{1,2}) - t^2 \chi_{1,3} - t^4 \th_{3,1} \\
& = \frac{t^2}{2} A \w e_1 \w e_2 \w e_3 + \frac{t^4}{2} \sum_{k=1}^3 B_k \w e_k,
\end{aligned}
\label{nc5eq27}
\e
where
\e
\begin{aligned}
A & = J_1 \ze_1 + J_2 \ze_2 + J_3 \ze_3 - \stx \dx b - 2 \chi, \\
B_1 & = \stx(- \ze_1 + J_3 \ze_2 - J_2 \ze_3) + J_1 \dx b - 2 \th_1, \\
B_2 & = \stx(- \ze_2 + J_1 \ze_3 - J_3 \ze_1) + J_2 \dx b - 2 \th_2, \\
B_3 & = \stx(- \ze_3 + J_2 \ze_1 - J_1 \ze_2) + J_3 \dx b - 2 \th_3.
\end{aligned}
\label{nc5eq28}
\e
Note that $A$ is a vertical 1-form and each $B_k$ is a vertical 3-form. By \eq{nc5eq24}, the $5$-form $\upsilon$ is $\dx$ of \eq{nc5eq27}. Thus we have
\begin{equation*}
\upsilon = \frac{t^2}{2} (\dx A) \w e_1 \w e_2 \w e_3 + \frac{t^4}{2} \sum_{k=1}^3 (\dx B_k) \w e_k,
\end{equation*}
and therefore using \eq{nc5eq11} and \eq{nc5eq10}, we find
\begin{align*}
\ast_t \upsilon & = \frac{t^2}{2} (\stx \dx A) + \frac{1}{2} (\stx \dx B_1) \w e_2 \w e_3 \\
& \qquad + \frac{1}{2} (\stx \dx B_2) \w e_3 \w e_1 + \frac{1}{2} (\stx \dx B_3) \w e_1 \w e_2, \\
\psi_t \w (\ast_t \upsilon) & = \frac{t^4}{2} \sum_{\substack{i,j,k \\ \text {cyclic}}} (\stx \dx B_k) \w \volx \w e_i \w e_j - \frac{t^4}{2} \sum_{\substack{i,j,k \\ \text {cyclic}}} (\stx \dx A) \w \om_k \w e_i \w e_j.
\end{align*}
Hence, using \eq{nc5eq9} we conclude that
\begin{equation*}
\psi_t \w (\ast_t \upsilon) = \frac{t^4}{2} \sum_{\substack{i,j,k \\ \text {cyclic}}} \big( \dsx (\stx B_k) + \dsx (J_k A) \big) \volx \w e_i \w e_j.
\end{equation*}
Thus, we find that $\upsilon$ is in $\Omega^5_{14}$ with respect to $\vp_t$ if and only if
\e
\dsx( \stx B_k + J_k A ) = 0 \qquad \text{ for all $k = 1, 2, 3$}.
\label{nc5eq29}
\e
Substituting the expressions \eq{nc5eq28} for $A,B_1,B_2,B_3$ we find that
\begin{align*}
\stx B_1 + J_1 A & = \ze_1 - J_3 \ze_2 + J_2 \ze_3 + J_1 \stx \! \dx b - 2 \stx \! \th_1 \\
& \qquad{} - \ze_1 - J_3 \ze_2 + J_2 \ze_3 - J_1 \stx \! \dx b - 2 J_1 \chi \\
& = 2 J_2 \ze_3 - 2 J_3 \ze_2 - 2 \stx \! \th_1 - 2 J_1 \chi,
\end{align*}
and cyclic permutations of this. Cancelling the factor of 2, the conditions \eq{nc5eq29} become
\begin{equation*}
\dsx ( J_2 \ze_3 - J_3 \ze_2 - \stx \! \th_1 - J_1 \chi) = 0,
\end{equation*}
and cyclic permutations of this. Recall that $\ze_k = \dx a_k + \xi_k$. From Lemma \ref{nc5lem1}, we have $\dsx (J_i \dx a_j) = 0$ for any $i,j$ and thus the conditions above become
\begin{align*}
\dsx( J_2 \xi_3 - J_3 \xi_2 - \stx \th_1 - J_1 \chi ) & = 0, \\
\dsx( J_3 \xi_1 - J_1 \xi_3 - \stx \th_2 - J_2 \chi ) & = 0, \\
\dsx( J_1 \xi_2 - J_2 \xi_1 - \stx \th_3 - J_3 \chi ) & = 0.
\end{align*}
But these are precisely the conditions \eq{nc5eq15} of Proposition \ref{nc5prop1}. Thus the proof is complete.
\end{proof}

\subsection{\texorpdfstring{Analytic results on the Eguchi--Hanson fibre $P_x \cong X$}{Analytic results on the Eguchi--Hanson fibre Pₓ≅X}}
\label{nc53}

Before we can give the proof of Theorem \ref{nc5thm}, in this section we collect some analytic results about the Eguchi--Hanson fibre $P_x \cong X$. First, the Riemannian manifold $(X, \gx)$ is an example of an \emph{asymptotically conical} (AC) Riemannian manifold, whose asymptotic cone is the cone over $\R \mathbb P^3$. More precisely, it is an \emph{asymptotically locally Euclidean} (ALE) Riemannian manifold, because the asymptotic cone $\R^4 / \{ \pm 1 \}$ is flat. We begin by stating some general results that are valid for any AC manifold, specialized to the case of dimension 4. Then we use these facts to establish results that require the stronger ALE property, as well as the specific topology of $X$. Finally the last result we collect in this section uses the fact that $(X, \gx)$ is also hyperK\"ahler.

On an AC manifold $(X, \gx)$, there exist reasonably nice \emph{Hodge-theoretic} results for forms with appropriate decay at infinity. The results we use are from Lockhart \cite{Lock} and Lockhart--McOwen \cite{LoMc}. We will not give detailed definitions of the weighted Sobolev spaces that we use, but rather only list the results we will need. A comprehensive summary of this theory using the same notation that we use here can be found in \cite[Section 4]{KaLo}. Although \cite{KaLo} is written for dimension 7, it is a simple matter to translate the results to dimension 4.

Specifically, the Lockhart--McOwen facts we will need are the following. Here all norms and covariant derivatives are with respect to the metric $\gx$ on $X$. Let $L^2_{k, \lambda} (E)$ denote the \emph{weighted Sobolev space} of sections of a tensor bundle $E$ over $X$ with rate $\lambda$. When $E = \Lambda^p T^*X$ is the bundle of $p$-forms on $X$, we will abbreviate our notation and write
\begin{equation*}
\Omega^p_{k, \lambda} = L^2_{k, \lambda} (\Lambda^p T^*X).
\end{equation*}
Similarly we write
\begin{equation*}
(\Omega^{2}_{\pm})_{k ,\lambda} = L^2_{k, \lambda} (\Lambda^2_{\pm} T^*X)
\end{equation*}
for the corresponding weighted Sobolev spaces of self-dual (anti-self-dual) $2$-forms on $X$.

\begin{enumerate}[(LM1)]
\item Let $r$ be the function on $X \setminus Y$ from \S\ref{nc25}. A smooth section $f$ of $E$ that lies in $L^2_{k, \lambda} (E)$ satisfies $\md{\nabla^k f} = O(r^{\lambda-k})$ as $r \to \infty$ for all $k \geq 0$. Conversely, a smooth section $f$ of $E$ such that $\md{\nabla^k f} = O(r^{\lambda - k})$ for all $k \geq 0$ as $r \to \infty$ lies in $L^2_{k, \lambda + \gamma} (E)$ for any $k \geq 0$ and any $\gamma > 0$. Moreover, we have $L^2_{k, \lambda} (E) \subseteq L^2_{k, \lambda'} (E)$ if $\lambda \leq \lambda'$.

\item For any rate $\lambda$, and any $k \geq 0$, the Laplacian $\lapx$ on $\Lambda^0 T^*X$ induces a continuous linear map
\begin{equation*}
(\lapx)_{k+2, \lambda} : \Omega^0_{k+2, \lambda} \to \Omega^0_{k, \lambda - 2},
\end{equation*}
and for generic (``noncritical'') rates, this map is Fredholm. At such rates, we have
\begin{equation*}
\Coker (\lapx)_{k+2, \lambda} \cong \Ker (\lapx)_{k, -2 - \lambda}.
\end{equation*}
Moreover, by elliptic regularity if $(\lapx)_{k+2, \lambda} f$ is smooth, then $f$ is.

\item In dimension 4 the rate $\lambda = -2$ is always critical for the Laplacian on functions, but there are no critical rates in the interval $(-2,0)$. Take $\gamma \in (0,2)$ so that $\lambda = -2 + \gamma$ is noncritical. Then by (LM2), the cokernel of $(\lapx)_{k+2, -2 + \gamma}$ is isomorphic to the kernel of $(\lapx)_{k, - \gamma}$. By elliptic regularity, elements in this kernel are smooth, and since by (LM1) they are $O(r^{- \gamma})$ as $r \to \infty$ with $\gamma > 0$, we conclude by the maximum principle that the kernel is trivial. Similarly $\Ker (\lapx)_{-2 + \gamma} = 0$ because $\gamma < 2$. Therefore, for $\gamma \in (0,2)$ the map
\begin{equation*}
(\lapx)_{k+2, -2 + \gamma} : \Omega^0_{k+2, -2 + \gamma} \to \Omega^0_{k, -4 + \gamma}
\end{equation*}
is an isomorphism.

\item The maps $\dx$ and $\dsx$ on $X$ induce continuous linear maps
\begin{equation*}
(\dx)_{k+1, \lambda} : \Omega^p_{k+1, \lambda} \to \Omega^{p+1}_{k, \lambda - 1}, \qquad
(\dsx)_{k+1, \lambda}  : \Omega^p_{k+1, \lambda} \to \Omega^{p-1}_{k, \lambda - 1}. 
\end{equation*}
for any rate $\lambda$ and any $k \geq 0$.

\item Elements of $\Omega^p_{k, \lambda}$ for $\lambda < -2$ are in the usual $L^2$ space. In particular, by (LM4) the spaces $\dx (\Omega^{p-1}_{k+1, \lambda + 1})$ and $\dsx (\Omega^{p+1}_{k+1, \lambda + 1})$ are $L^2$-orthogonal subspaces of $\Omega^p_{k, \lambda}$ if $\lambda < -2$.

\item If $w \in \Omega^1_{k+1, \lambda+1}$ is smooth, then $w \w \dx w$ is a $3$-form that is $O(r^{2 \lambda + 1})$ as $r \to \infty$. If $2 \lambda + 1 + 3 < 0$, equivalently $\lambda < -2$, then we can apply Stokes' Theorem to the exact $4$-form $\dx w \w \dx w$ to deduce that $\int_X (\dx w \w \dx w) = 0$. See, for example, the proof of~\cite[Lem.~4.68]{KaLo}.

\item Define the operator $\Dx : \Omega^0 (X) \oplus \Omega^2_+ (X) \to \Omega^1 (X)$ by $\Dx (f, \beta) = \dx f + \dsx \beta$. This operator is elliptic, and has formal adjoint $\Dsx: \Omega^1(X) \to \Omega^0(X) \oplus \Omega^2_+ (X)$ given by $\Dsx \alpha = (\dsx \alpha, \pi_+ \dx \alpha)$. For any rate $\lambda$, and any $k \geq 0$, it induces a continuous linear map
\begin{equation*}
(\Dx)_{k+1, \lambda} : \Omega^0_{k+1, \lambda} \oplus (\Omega^2_+)_{k+1, \lambda} \to \Omega^1_{k, \lambda - 1},
\end{equation*}
and for generic (noncritical) rates, this map is Fredholm. At such rates, we have
\begin{equation*}
\Coker (\Dx)_{k+1, \lambda} \cong \Ker (\Dsx)_{k, -3 - \lambda}.
\end{equation*}
Moreover, by elliptic regularity if $(\Dx)_{k+1, \lambda} (f, \beta)$ is smooth, then $(f, \beta)$ is smooth.

\item For the operators $(\lapx)_{k+2, \lambda}$, $(\Dx)_{k+1, \lambda}$, and $(\Dsx)_{k+1, \lambda}$ of (LM2) and (LM7), the kernel or cokernel can only change as we vary $\lambda$ if we cross a critical rate. The critical rates are those for which there exist nontrivial sections, which are homogeneous of order $\lambda$ with respect to dilations, in the kernel of the corresponding operators on the cone.

\item Let $\mathcal H^p_{\lambda} = \{ \alpha \in \Omega^p_{k, \lambda} : \dx \alpha = 0 \text{ and } \dsx \alpha = 0 \}$ which by elliptic regularity consists of smooth elements so is independent of $k$. Then we have
\begin{equation*}
\mathcal H^1_{-2} \cong H^1_{\mathrm{cs}} (X), \quad \mathcal H^2_{-2} \cong \Im ( H^2_{\mathrm{cs}} (X) \to H^2 (X) ), \quad \mathcal H^3_{-2} \cong H^3 (X),
\end{equation*}
where $H^p (X)$ and $H^p_{\text{cs}} (X)$ denote the $p^{\text{th}}$ de Rham cohomology and $p^{\text{th}}$ compactly supported de Rham cohomology of $X$, respectively. The map $H^2_{\mathrm{cs}} (X) \to H^2 (X)$ is the obvious one. See \cite[Th.~6.5.2]{Lota} for a more general statement.
\end{enumerate}

\begin{cor} Let\/ $(X, \gx)$ be an AC Riemannian manifold of dimension $4$. Let\/ $w \in \Omega^1_{k+1, \lambda + 1}$ such that\/ $\dx w \in (\Omega^2_-)_{k, \lambda}$. If\/ $\lambda < -2$ then~$\dx w = 0$.
\label{nc5cor2}
\end{cor}

\begin{proof} By hypothesis we have
\e 
\stx \! \dx w = - \dx w.
\label{nc5eq30}
\e
By (LM5), the $2$-form $\dx w$ is in $L^2$. Moreover, we can apply (LM6) and \eq{nc5eq30} to deduce that
\begin{equation*}
\nm{\dx w}_{L^2} = \int_X \dx w \w \stx \dx w = - \int_X \dx w \w \dx w = 0
\end{equation*}
as claimed.
\end{proof}

\begin{lem} Let\/ $(X, \gx)$ be the Eguchi--Hanson space. There are no critical rates for the operator $\Dsx$ in the interval\/~$[-2, 0]$.
\label{nc5lem3}
\end{lem}

\begin{proof} By (LM8), we must show that there do not exist any nonzero elements homogeneous of order $\lambda \in [-2, 0)$ in the kernel of the corresponding operator on the cone. Let $\alpha$ be such a $1$-form on the cone, homogeneous of order $\lambda$. Then $\d^* \alpha = 0$ and $\pi_+ \d \alpha = 0$. Thus $\d \d^* \alpha = 0$ and $\d^* \d \alpha = - \d^* \ast \d \alpha = 0$. So $\alpha$ is a harmonic $1$-form. The harmonic $1$-forms on the cone $\R^4 / \{\pm 1 \}$ correspond to the $\{\pm 1\}$-invariant harmonic $1$-forms on $\R^4$. But $\R^4$ admits a global trivialization by parallel $1$-forms. So any $1$-form on $\R^4$ is of the form $\alpha = \sum_{j=1}^4 f_j \d x_j$ and $\Delta \alpha = \sum_{j=1}^4 (\Delta f_j) \d x_j$. Hence the critical rates of $\Dsx$ are a \emph{subset} of the critical rates for $\Delta$ on functions on $\R^4$, which is also a cone over $\cS^3$. By (LM3), there are no critical rates in $(-2,0)$ for the Laplacian on functions, hence none for $\Dsx$ in that interval. 

It remains to consider the end points $\lambda = -2,0$. It is easy to check that the only harmonic functions on a metric cone in dimension $4$ that are homogeneous of order $\lambda = -2$ or $\lambda = 0$ are of the form $c r^{\lambda}$ for $c\in\R$. Therefore the only harmonic $1$-forms on $\R^4$ that are homogeneous of order $\lambda = -2$ or $\lambda = 0$ will be of the form $r^{\lambda} v$ for some parallel $1$-form $v$. But the nonzero parallel $1$-forms on $\R^4$ are not $\{ \pm 1 \}$-invariant, hence they do not descend to $\R^4 / \{ \pm 1 \}$. Thus $\lambda = -2$ and $\lambda = 0$ are also not critical rates for~$\Dsx$.
\end{proof}

\begin{prop} Let\/ $(X, \gx)$ be the Eguchi--Hanson space. For any $k \geq 0$ and for $\gamma > 0$ sufficiently small, we have
\begin{align*}
\Omega^1_{k, -3 + \gamma} & = \dx (\Omega^0_{k+1, -2 + \gamma}) \oplus \dsx ( (\Omega^2_+)_{k+1, -2 + \gamma}), \\
\Omega^3_{k, -3 + \gamma} & = \dsx (\Omega^4_{k+1, -2 + \gamma}) \oplus \dx ( (\Omega^2_+)_{k+1, -2 + \gamma}),
\end{align*}
where in both cases the summands are $L^2$-orthogonal. Moreover, any smooth closed form in $\Omega^1_{k, -3 + \gamma}$ or $\Omega^3_{k, -3 + \gamma}$ is necessarily $\dx$ of a smooth form in $\Omega^0_{k+1, -2 + \gamma}$ or $(\Omega^2_+)_{k+1, -2 + \gamma},$ respectively. 
\label{nc5prop3}
\end{prop}

\begin{proof} Consider (LM7) for rate $\lambda = -2 + \gamma$. For $\gamma > 0$ sufficiently small, this rate will be noncritical, and $\Coker (\Dx)_{k+1, -2 + \gamma} \cong \Ker (\Dsx)_{k, -1 - \gamma}$. By Lemma \ref{nc5lem3} and (LM8), we have that
\begin{equation*}
\Ker (\Dsx)_{k,-1 - \gamma} = \Ker (\Dsx)_{k, -2 - \epsilon}
\end{equation*}
for some $\epsilon > 0$. Let $\alpha \in \Ker (\Dsx)_{k, -2 - \epsilon}$. Then $\dsx \alpha = 0$ and $\pi_+ \dx \alpha = 0$. By Corollary \ref{nc5cor2} we have $\dx \alpha = 0$. Since $-2 - \epsilon < -2$, we deduce that $\alpha \in \mathcal H^1_{-2}$. By (LM9), in this case $\mathcal H^1_{-2} \cong H^1_{\mathrm{cs}} (X) = \{ 0 \}$, because $X \cong T^* \cS^2$. Hence, $\Coker (\Dx)_{k+1, -2 + \gamma} = \{ 0 \}$, and therefore
\begin{equation*}
\Omega^1_{k, -3 + \gamma} =  \dx (\Omega^0_{k+1, -2 + \gamma}) + \dsx ((\Omega^2_+)_{k+1, -2 + \gamma}),
\end{equation*}
and the sum is direct by (LM5). This establishes the first decomposition. The second decomposition is obtained by applying $\stx$ to both sides of the first one.

For either of the two decompositions, suppose that $\alpha = \dx \sigma + \dsx \tau$ is closed. Then $\dx \dsx \tau = 0$, and at this rate it lies in $L^2$. Thus $\dsx \tau = 0$, so $\alpha = \dx \sigma$ is exact. The final statement about smoothness follows from the elliptic regularity remark in (LM7).
\end{proof}

We close this section with a useful observation. The space $(X, \gx)$ is also hyperK\"ahler, so it is equipped with an orthonormal triple $\omega_1,\omega_2,\omega_3$ of \emph{parallel} self-dual $2$-forms. It is a standard fact that the curvature operator is anti-self-dual, since the Ricci curvature vanishes and the bundle $\Lambda^2_+ (T^* X)$ admits a parallel trivialization. This in turn implies that the curvature term in the Weitzenb\"ock formula on $2$-forms vanishes, so $\Delta_d = \nabla^* \nabla$ on $2$-forms. It then follows, because any self-dual $2$-form $\eta$ can be written as $\eta = f_1 \omega_1 + f_2 \omega_2 + f_3 \omega_3$ for some functions $f_1,f_2,f_3$ and the $\omega_i$'s are parallel, that
\begin{equation}
\begin{aligned} 
\lapx \eta & = \lapx (f_1 \omega_1 + f_2 \omega_2 + f_3 \omega_3) = (\lapx f_1) \omega_1 + (\lapx f_2) \omega_2 + (\lapx f_3) \omega_3 \\
& \qquad \text{ for any self-dual $2$-form $\eta$ on $X$.}
\end{aligned}
\label{nc5eq31}
\end{equation}
We will use this fact in the proof of Theorem \ref{nc5thm} in \S\ref{nc54}.

\subsection{Proof of the correction theorem}
\label{nc54}

We are now ready to give the proof of Theorem \ref{nc5thm}. We break the proof into three steps.
\smallskip

\noindent{\bf Step One.} Using the expression for $g_t$ in \eq{nc5eq10}, the expressions \eq{nc5eq13} for $\xi_{1,2},\chi_{1,3},\th_{3,1}$, and the estimates \eq{nc4eq35}, \eq{nc4eq37}, \eq{nc4eq38} we find that with respect to the metric $\gx$ on $X$,
\e 
\md{\nabla^k \xi_i} = O(r^{-3-k}), \qquad \md{\nabla^k \chi} = O(r^{-3-k}), \qquad \md{\nabla^k \th_i} = 0, 
\label{nc5eq32}
\e
for $r \geq 1$, for all $k \geq 0$ and $i = 1, 2, 3$. In particular, by (LM1) we deduce that
\e 
\xi_1, \xi_2, \xi_3, \chi \in \Omega^1_{k, -3 + \gamma}, \qquad \th_1, \th_2, \th_3 \in \Omega^3_{k, -3 + \gamma},
\label{nc5eq33}
\e
for any $\gamma > 0$ and any $k \geq 0$.

In order to simplify notation we define
\begin{align*}
\tilde A & = \ha (- J_1 \xi_1 - J_2 \xi_2 - J_3 \xi_3) + \chi, &
\tilde B_1 & = \ha \stx \! (\xi_1 - J_3 \xi_2 + J_2 \xi_3) + \th_1, \\
\tilde B_2 & = \ha \stx \! (\xi_2 - J_1 \xi_3 + J_3 \xi_1) + \th_2, &
\tilde B_3 & = \ha \stx \! (\xi_3 - J_2 \xi_1 + J_1 \xi_2) + \th_3,
\end{align*}
where we have used tildes to avoid confusion with the similar but different \eq{nc5eq28}. Note that $\tilde A$ is a vertical 1-form, and the $\tilde B_i$'s are vertical 3-forms on $P$. Since $\stx$ and the $J_i$'s are isometries, it follows from \eq{nc5eq33} that
\e 
\tilde A, \, \stx \tilde B_1, \, \stx \tilde B_2, \, \stx \tilde B_3 \in \Omega^1_{k, -3 + \gamma},
\label{nc5eq34}
\e
for any $\gamma > 0$ and all $k \geq 0$.

Define on $P$ the following 2-form:
\begin{equation} 
\vartheta = -\ast_t \dx \bigl[ (D_{\vp_t} \Theta) (t^2 \xi_{1,2}) - t^2 \chi_{1,3} - t^4 \th_{3,1} \bigr] + t^2 (\dx \tilde A).
\label{nc5eq35}
\end{equation}
We will rewrite expression \eq{nc5eq35} as follows. Using the expressions \eq{nc5eq13} for $\xi_{1,2},\chi_{1,3}$, and $\th_{3,1}$, and equation \eq{nc5eq23} to compute $(D_{\vp_t} \Theta)(t^2 \xi_{1,2})$, we find that
\begin{align*}
(D_{\vp_t} \Theta) (t^2 \xi_{1,2}) - t^2 \chi_{1,3} - t^4 \th_{3,1} & = \frac{t^2}{2} ( J_1 \xi_1 + J_2 \xi_2 + J_3 \xi_3 - 2 \chi) \w e_1 \w e_2 \w e_3 \\
& \qquad {}+ \frac{t^4}{2} \big( \stx \! (-\xi_1 + J_3 \xi_2 - J_2 \xi_3) - 2 \th_1\big) \w e_1 \\
& \qquad {} + \frac{t^4}{2} \big( \stx \! (- J_3 \xi_1 - \xi_2 + J_1 \xi_3) - 2 \th_2\big) \w e_2 \\
& \qquad {}+ \frac{t^4}{2} \big( \stx \! ( J_2 \xi_1 - J_1 \xi_2 - \xi_3) - 2 \th_3\big) \w e_3.
\end{align*}
Thus we have
\begin{equation*}
(D_{\vp_t} \Theta) (t^2 \xi_{1,2}) - t^2 \chi_{1,3} - t^4 \th_{3,1} = - t^2 \tilde A \w e_1 \w e_2 \w e_3 - t^4 ( \tilde B_1 \w e_1 + \tilde B_2 \w e_2 + \tilde B_3 \w e_3 )
\end{equation*}
and hence
\begin{equation*}
\begin{aligned}
& \dx \bigl[ (D_{\vp_t} \Theta) (t^2 \xi_{1,2}) - t^2 \chi_{1,3} - t^4 \th_{3,1} \bigr] \\
& = -t^2 (\dx \tilde A) \w e_1 \w e_2 \w e_3 - t^4 \big( (\dx \tilde B_1) \w e_1 + (\dx \tilde B_2) \w e_2 + (\dx \tilde B_3) \w e_3 \big).
\end{aligned}
\end{equation*}
Applying \eq{nc5eq11} again, we deduce that the expression \eq{nc5eq35} can be written as
\e 
\begin{aligned}
\vartheta & = -\ast_t \dx \bigl[ (D_{\vp_t} \Theta) (t^2 \xi_{1,2}) - t^2 \chi_{1,3} - t^4 \th_{3,1} \bigr] + t^2 (\dx \tilde A) \\
& = t^2 (\dx \tilde A + \stx \dx \tilde A) + (\stx \dx \tilde B_1) \w e_2 \w e_3 \\
& \qquad{} + (\stx \dx \tilde B_2) \w e_3 \w e_1 + (\stx \dx \tilde B_3) \w e_1 \w e_2.
\end{aligned}
\label{nc5eq36}
\e

Now on $P$ consider the following equation:
\begin{equation} 
(\dx \! \ast_t \dx \! \ast_t - \ast_t \dx \! \ast_t \dx)(t^2 \al_{0,2} + t^4 \al_{2,0}) = \vartheta. 
\label{nc5eq37}
\end{equation}
If, in equation \eq{nc5eq37}, the $\dx$'s were replaced by the full exterior derivative $\d$ on $P$, then the left hand side would be $\Delta_t = \d \ast_t \d \ast_t - \ast_t \d \ast_t \d$, the Laplacian with respect to $g_t$ on $2$-forms. Thus, we expect \eq{nc5eq37} to give Laplace equations on each fibre $P_x$. Explicitly, as before we write
\e \label{nc5eq38}
\al_{0,2} = a_1 e_2 \w e_3 + a_2 e_3 \w e_1 + a_3 e_1 \w e_2, \qquad \al_{2,0} = b,
\e
where $a_1,a_2,a_3$ are functions and $b$ is a vertical $2$-form on $P$. Using \eq{nc5eq11} repeatedly we compute
\begin{align*}
\dx \al_{0,2} & = (\dx a_1) \w e_2 \w e_3 + (\dx a_2) \w e_3 \w e_1 + (\dx a_3) \w e_1 \w e_2, \\
\dx \al_{2,0} & = \dx b, \\
\ast_t \dx \al_{0,2} & = t^2 \big( (\stx \dx a_1) \w e_1 + (\stx \dx a_2) \w e_2 + (\stx \dx a_3) \w e_3 \big), \\
\ast_t \dx \al_{2,0} & = t^{-2} (\stx \dx b) \w e_1 \w e_2 \w e_3, \\
\dx \! \ast_t \dx \al_{0,2} & = t^2 \big( (\dx \! \stx \! \dx a_1) \w e_1 + (\dx \! \stx \! \dx a_2) \w e_2 + (\dx \! \stx \! \dx a_3) \w e_3 \big), \\
\dx \! \ast_t \dx \al_{2,0} & = t^{-2} (\dx \! \stx \! \dx b) \w e_1 \w e_2 \w e_3,
\end{align*}
and hence we find
\ea
\ast_t \dx \! \ast_t \dx \al_{0,2} & = t^{-2} \big( (\stx \dx \! \stx \! \dx a_1) \w e_2 \w e_3 + (\stx \dx \! \stx \! \dx a_2) \w e_3 \w e_1 
\nonumber\\
& \qquad{} + (\stx \dx \! \stx \! \dx a_3) \w e_1 \w e_2 \big)
\nonumber\\
& = -t^{-2} (\dsx \dx a_1) \w e_2 \w e_3 -t^{-2} (\dsx \dx a_2) \w e_3 \w e_1 
\label{nc5eq39}\\
& \qquad{} -t^{-2} (\dsx \dx a_3) \w e_1 \w e_2, 
\nonumber\\
\ast_t \dx \! \ast_t \dx \al_{2,0} & = t^{-2} (\stx \dx \! \stx \! \dx b) = - t^{-2} \dsx \dx b,
\nonumber
\ea
where we have used that $\dsx = -\stx \! \dx \stx$ on forms of any degree since $X$ is even-dimensional. Similarly we find that
\begin{align*}
\ast_t \al_{0,2} & = t^4 \big( (\stx a_1) \w e_1 + (\stx a_2) \w e_2 + (\stx a_3) \w e_3 \big), \\
\ast_t \al_{2,0} & = (\stx b) \w e_1 \w e_2 \w e_3, \\
\dx \! \ast_t \al_{0,2} & = 0, \\
\dx \! \ast_t \al_{2,0} & = (\dx \! \stx \! b) \w e_1 \w e_2 \w e_3, \\
\ast_t \dx \! \ast_t \al_{0,2} & = 0, \\
\ast_t \dx \! \ast_t \al_{2,0} & = - t^{-2} (\stx \dx \! \stx \! b),
\end{align*}
and hence that
\e
\label{nc5eq40}
\begin{aligned}
\dx \! \ast_t \dx \! \ast_t \al_{0,2} & = 0, \\
\dx \! \ast_t \dx \! \ast_t \al_{2,0} & = - t^{-2} (\dx \! \stx \! \dx \! \stx \! b) = t^{-2} \dx \dsx b.
\end{aligned}
\e
The expressions in \eq{nc5eq39} and \eq{nc5eq40} yield that
\e \label{nc5eq41}
\begin{aligned}
& (\dx \! \ast_t \dx \! \ast_t - \ast_t \dx \! \ast_t \dx)(t^2 \al_{0,2} + t^4 \al_{2,0}) \\
& = (\lapx a_1) \w e_2 \w e_3 + (\lapx a_2) \w e_3 \w e_1 + (\lapx a_3) \w e_1 \w e_2 + t^2 (\lapx b).
\end{aligned}
\e
Equation \eq{nc5eq41} expresses the left hand side of \eq{nc5eq37} in terms of the fibre Laplacians, as expected.

Equating \eq{nc5eq41} and \eq{nc5eq36}, and using that $\stx^2 = -1$ on odd forms and $\dsx = - \stx \! \dx \stx$ on all forms, we deduce that equation \eq{nc5eq37} is in fact equivalent on each fibre $P_x$ to the following system:
\begin{align} 
& \lapx a_1 = \dsx (\stx \tilde B_1), \qquad \lapx a_2 = \dsx (\stx \tilde B_2), \qquad \lapx a_3 = \dsx (\stx \tilde B_3), 
\label{nc5eq42}\\ 
& \hspace{1.3in} \lapx b = (\dx \tilde A + \stx \dx \tilde A).
\label{nc5eq43}
\end{align}
By equation \eq{nc5eq34} and (LM4), the right hand sides of the four equations above are all in $\Omega^0_{k, -4 + \gamma}$ and $(\Omega^2_+)_{k, -4 + \gamma}$ for all $k \geq 0$ and all $\gamma > 0$. Moreover, the right hand side of \eq{nc5eq43} is also self-dual. 

Therefore by equation \eq{nc5eq31} and fact (LM3), for $\gamma \in (0,2)$, on each fibre $P_x$ of $\si:P\ra L$ for $x\in L$ there exist unique functions $a_1,a_2,a_3$, and a unique self-dual $2$-form $b$ solving \eq{nc5eq42} and \eq{nc5eq43} in $L^2_{k+2,-2+\gamma}$ on $P_x$ for all $k \geq 0$. Here we take $b$ self-dual to ensure that $\al_{2,0}\in\Ga^\iy(\La^2_+V^*)$, as in the last part of Theorem \ref{nc5thm}. In particular by the elliptic regularity remark in (LM2), these solutions are smooth. Moreover, by (LM1), we in fact have
\e 
\md{\nabla^k a_i} = O(r^{-2 - k + \gamma}), \qquad \md{\nabla^k b} = O(r^{-2 - k + \gamma}),
\label{nc5eq44}
\e
on $P_x$ as $r \to \infty$, for all $k \geq 0$ and~$i=1,2,3$. 

We can show that these solutions on $P_x$ for $x\in L$ depend smoothly on the base point $x$ using the Banach space implicit function theorem, by trivializing locally and observing that \eq{nc5eq42}--\eq{nc5eq43} depend smoothly on the base point $x$. Hence smooth $a_1,a_2,a_3$ and $b$ satisfying \eq{nc5eq42}--\eq{nc5eq44} exist on all of $P$, and are unique provided $b$ is self-dual on the fibres $P_x$. By \eq{nc5eq10} for $t = 1$ and \eq{nc5eq38}, we deduce that we have found unique smooth $2$-forms $\al_{2,0}$ and $\al_{0,2}$ solving \eq{nc5eq37} such that with respect to the metric $g_1$ on $P$, 
\begin{equation*}
\md{\nabla^k \al_{2,0}}_{g_1^P} = O(\check r^{-2 - k + \gamma}), \qquad \md{\nabla^k \al_{0,2}}_{g_1^P} = O(\check r^{-2 - k + \gamma})
\end{equation*}
for $\check r \geq 1$, and these are unique provided $\al_{2,0}\in\Ga^\iy(\La^2_+V^*)$. Measuring norms using the metric $g_t^P$ instead, the estimates \eq{nc5eq2}--\eq{nc5eq5} follow easily.

If we choose $\gamma > 0$ sufficiently small, then by (LM5) since $-3 + \gamma < -2$, we deduce by $L^2$-orthogonality of the images of $\dx$ and $\dsx$ that the solutions to \eq{nc5eq43} and in fact solve the uncoupled equations
\begin{equation*}
\dx \dsx b = \dx \tilde A, \qquad \dsx \dx b = \stx \dx \tilde A.
\end{equation*}
Retracing our steps backwards using \eq{nc5eq39}, \eq{nc5eq40}, and \eq{nc5eq35}, we find that the solution to \eq{nc5eq37} we have produced is actually a solution to the two uncoupled equations
\begin{align*}
\dx \! \ast_t \dx \! \ast_t (t^2 \al_{0,2} + t^4 \al_{2,0}) & = t^2 (\dx \tilde A), \\
- \ast_t \dx \! \ast_t \dx (t^2 \al_{0,2} + t^4 \al_{2,0}) & = -\ast_t \dx \bigl[ (D_{\vp_t} \Theta) (t^2 \xi_{1,2}) - t^2 \chi_{1,3} - t^4 \th_{3,1} \bigr].
\end{align*}
We throw away the first equation, and write the second equation as
\e 
\dx \! \ast_t \dx (t^2 \al_{0,2} + t^4 \al_{2,0}) = \dx \bigl[ (D_{\vp_t} \Theta) (t^2 \xi_{1,2}) - t^2 \chi_{1,3} - t^4 \th_{3,1} \bigr].
\label{nc5eq45}
\e

\noindent{\bf Step Two.} We have constructed $\al_{0,2}$ and $\al_{2,0}$ solving \eq{nc5eq45} and satisfying \eq{nc5eq2}--\eq{nc5eq5}. To simplify notation, define the 3-form
\begin{equation*}
C = t^2 \dx \al_{0,2} + t^4 \dx \al_{2,0}.
\end{equation*}
Note that from \eq{nc5eq3} and \eq{nc5eq5} for all $k \geq 0$ and $\gamma > 0$ sufficiently small we have
\e 
\bmd{\nabla^k C}_{g_t^P} = \begin{cases} O(t^{1-k}), & \check r \le 1, \\[4pt] O(t^{1-k} \check r^{-3-k+\ga}), & \check r \ge 1. \end{cases}
\label{nc5eq46}
\e
Since $C$ is a form of type $(1,2) + (3,0)$, by the proof of Corollary \ref{nc5cor1} we have 
\begin{equation*}
(D_{\vp_t} \Theta)(C) = 2 \ast_t \pi_7 C - \ast_t C.
\end{equation*}
Thus we can rewrite \eq{nc5eq45} as
\begin{equation*}
\dx \! \ast_t C = \dx \bigl[ 2 \ast_t \pi_7 C - (D_{\vp_t} \Theta) (C) \bigr] = \dx \bigl[ (D_{\vp_t} \Theta) (t^2 \xi_{1,2}) - t^2 \chi_{1,3} - t^4 \th_{3,1} \bigr],
\end{equation*}
and rearrange it as
\e \label{nc5eq47}
\dx \bigl[ (D_{\vp_t} \Theta)(t^2 \dx \al_{0,2} + t^4 \dx \al_{2,0} + t^2 \xi_{1,2}) - t^2 \chi_{1,3} - t^4 \th_{3,1} \bigr] = 2 \, \dx \! \ast_t \pi_7 C.
\e
The form $\ast_t \pi_7 C$ is in $\Omega^4_7$ with respect to $\vp_t$, and hence can be written as
\e \label{nc5eq48}
\ast_t \pi_7 C = w \w \vp_t
\e
for some unique $1$-form $w$ on $P$. Since $\ast_t$ is an isometry and wedge product with $\vp_t$ on $1$-forms is an isometry up to a constant factor, we deduce from \eq{nc5eq48} and \eq{nc5eq46} that for all $k \geq 0$ and $\gamma > 0$ sufficiently small we have
\begin{equation*}
\bmd{\nabla^k w}_{g_t^P} = \begin{cases} O(t^{1-k}), & \check r \le 1, \\[4pt] O(t^{1-k} \check r^{-3-k+\ga}), & \check r \ge 1, \end{cases}
\end{equation*}
from which it follows from \eq{nc5eq10} and (LM1) that
\e \label{nc5eq49}
w \in \Omega^1_{k, -3 + \gamma} \, \text{ for all $k \geq 0$ and $\gamma > 0$ sufficiently small}.
\e
By Lemma \ref{nc5lem2} we know $\pi_7 C$ is type $(1,2) + (3,0)$, so $\ast_t \pi_7 C$ is type $(3,1) + (1,3)$, and we deduce from the fact that $\vp_t$ is type $(0,3) + (2,1)$ that the $1$-form $w$ is in fact purely vertical. Thus, the right hand side of \eq{nc5eq47} is
\e \label{nc5eq50}
2 \, \dx \! \ast_t \pi_7 C = 2 \, \dx (w \w \vp_t) = 2 (\dx w) \w \vp_t, 
\e
where we have used the fact that $\dx \vp_t = 0$, which follows from \eq{nc5eq10} as the $\omega_i$'s are closed in the fibre direction and the $e_j$'s are pulled back from the base.

Proposition \ref{nc5prop2} precisely says that the left hand side of \eq{nc5eq47} lies in $\Omega^5_{14}$ with respect to $\vp_t$. Hence, the $\Omega^5_7$ component of the right hand side must vanish, so by \eq{nc5eq50} we conclude that $\pi_7 \bigl( (\dx w) \w \vp_t \bigr) = 0$ and thus $\pi_7 (\dx w) = 0$. But this means that $(\dx w) \w \psi_t = 0$, and from \eq{nc5eq10} and the fact that $\dx w$ is of type $(2,0)$ we deduce that $(\dx w) \w \omega_k = 0$ for $k = 1, 2, 3$. 

Thus $\dx w$ is anti-self-dual on the fibres. Since $-3 + \gamma < -2$, by \eq{nc5eq49} and Corollary \ref{nc5cor2} we find that $\dx w = 0$ and we conclude that both sides of \eq{nc5eq47} are actually zero. Thus we see from equation \eq{nc5eq24} that~$\upsilon = 0$.
\smallskip

\noindent{\bf Step Three.} In the previous step, we showed the $2$-forms $\al_{0,2}$ and $\al_{2,0}$ constructed in Step One satisfying \eq{nc5eq2}--\eq{nc5eq5} are in fact solutions to the equation
\e
\dx \bigl[ (D_{\vp_t} \Theta)(t^2 \dx \al_{0,2} + t^4 \dx \al_{2,0} + t^2 \xi_{1,2}) - t^2 \chi_{1,3} - t^4 \th_{3,1} \bigr] = 0.
\label{nc5eq51}
\e
Note that equation \eq{nc5eq51} is precisely $\dx$ of equation \eq{nc5eq1}. Rewrite \eq{nc5eq1} as
\e 
t^2 \dx \be_{0,3} \!+\! t^4 \dx \be_{2,1} \!= \!(D_{\vp_t} \Theta)(t^2 \dx \al_{0,2} \!+\! t^4 \dx \al_{2,0} \!+\! t^2 \xi_{1,2}) \!-\! t^2 \chi_{1,3} \!-\! t^4 \th_{3,1}.
\label{nc5eq52}
\e
By \eq{nc5eq51} the right hand side above is a $4$-form on $P$ in the kernel of $\dx$. In fact, in \eq{nc5eq27} we determined that
\e 
\begin{aligned}
& (D_{\vp_t} \Theta)(t^2 \dx \al_{0,2} + t^4 \dx \al_{2,0} + t^2 \xi_{1,2}) - t^2 \chi_{1,3} - t^4 \th_{3,1} \\
& = \frac{t^2}{2} A \w e_1 \w e_2 \w e_3 + \frac{t^4}{2} \sum_{k=1}^3 B_k \w e_k,
\end{aligned}
\label{nc5eq53}
\e
with $A,B_1,B_2,B_3$ given by \eq{nc5eq28}. Equation \eq{nc5eq51} therefore says that
\begin{equation*}
\dx A = 0, \qquad \dx B_1 = 0, \qquad \dx B_2 = 0 \qquad \text{and}\qquad \dx B_3 = 0.
\end{equation*}
By \eq{nc5eq44} we have
\e 
\md{\nabla^k (\dx a_i)} = O(r^{-3 - k + \gamma}), \qquad \md{\nabla^k (\dx b)} = O(r^{-3 - k + \gamma}),
\label{nc5eq54}
\e
for $i = 1, 2, 3$ and all $k \geq 0$. Combining \eq{nc5eq54} with \eq{nc5eq32}, \eq{nc5eq26}, and \eq{nc5eq28} we deduce just as in \eq{nc5eq34} that
\e 
A \in \Omega^1_{k, -3 + \gamma}, \qquad B_1, \, B_2, \, B_3 \in \Omega^3_{k, -3 + \gamma},
\label{nc5eq55}
\e
for any $\gamma > 0$ sufficiently small and all $k \geq 0$. Now write
\e 
\be_{0,3} = h e_1 \w e_2 \w e_3, \qquad \be_{2,1} = f_1 \w e_1 + f_2 \w e_2 + f_3 \w e_3,
\label{nc5eq56}
\e
where $h$ is a function and $f_1,f_2,f_3$ are vertical $2$-forms on $P$. We will take $f_1,f_2,f_3$ to be self-dual on the fibres $P_x$, so that $\be_{2,1}\in\Ga^\iy(\La^2_+V^*\ot H^*)$, as in the last part of Theorem \ref{nc5thm}. We have
\begin{align*}
\dx \be_{0,3} & = (\dx h) \w e_1 \w e_2 \w e_3, \\
\dx \be_{2,1} & = (\dx f_1) \w e_1 + (\dx f_2) \w e_2 + (\dx f_3) \w e_3.
\end{align*}
Using the above two expressions and equations \eq{nc5eq52}--\eq{nc5eq53}, we find that \eq{nc5eq1} is in fact equivalent on each fibre $P_x$ to the following system of equations:
\e 
\dx h = \ha A, \qquad \dx f_1 = \ha B_1, \qquad \dx f_2 = \ha B_2, \qquad \dx f_3 = \ha B_3.
\label{nc5eq57}
\e

By \eq{nc5eq55} and Proposition \ref{nc5prop3}, on each fibre $P_x$ there exists a smooth $h$ in $\Omega^0_{k, - 2 + \gamma}$ and smooth $f_1, f_2, f_3$ in $(\Omega^2_+)_{k, -2 + \gamma}$ satisfying \eq{nc5eq57}. Suppose $h'$ was another solution in $\Omega^0_{k, - 2 + \gamma}$ to $\dx h' = \frac{1}{2} A$. Then $h - h'$ is a constant function on $X$, and decays to zero at infinity, so $h' = h$. If $f_j'$ was another solution in $(\Omega^2_+)_{k, -2 + \gamma}$ to $\dx f' = \frac{1}{2} B_j$, then $f_j - f_j'$ is a closed self-dual $2$-form on $X$ (and thus also coclosed) which decays at rate $-2 + \gamma$ at infinity. Thus by~\eqref{nc5eq31}, we have $f_j - f_j' = c_1 \omega_1 + c_2 \omega_2 + c_3 \omega_3$ for some decaying harmonic functions $c_1, c_2, c_3$ on $X$, which must therefore vanish by the maximum principle. Thus $f_j' = f_j$, so $h,f_1,f_2,f_3$ are unique.

As for $a_1,a_2,a_3,b$ in Step One, we can show that these $h,f_1,f_2,f_3$ on $P_x$ for $x\in L$ depend smoothly on the base point $x$ using the Banach space implicit function theorem, as equations \eq{nc5eq57} depend smoothly on the base point $x$. Thus by \eq{nc5eq10} for $t=1$ and \eq{nc5eq56}, we deduce that we have found smooth 3-forms $\be_{0,3}$ and $\be_{2,1}$ solving equation \eq{nc5eq52}, and hence \eq{nc5eq1}, such that with respect to the metric $g_1$ on $P$, 
\begin{equation*}
\md{\nabla^k \be_{0,3}}_{g_1^P} = O(\check r^{-2 - k + \gamma}), \qquad \md{\nabla^k \be_{2,1}}_{g_1^P} = O(\check r^{-2 - k + \gamma})
\end{equation*}
for $\check r \geq 1$, and these $\be_{0,3},\be_{2,1}$ are unique provided $\be_{2,1}\in\Ga^\iy(\La^2_+V^*\ot H^*)$. Measuring norms using the metric $g_t^P$ instead, the estimates \eq{nc5eq6}--\eq{nc5eq8} follow easily. This completes the proof of Theorem~\ref{nc5thm}.

\section{\texorpdfstring{Torsion-free $\Gt$-structures on the resolution $N$ of $M/\an{\io}$}{Torsion-free G₂-structures on the resolution N of M/<ι>}}
\label{nc6}

Sections \ref{nc61}--\ref{nc64} prove our main theorem, Theorem \ref{nc6thm}, and \S\ref{nc65}--\S\ref{nc66} explain two generalizations of it.

\subsection{\texorpdfstring{The resolution $N$ of $M/\an{\io}$}{The resolution N of M/<ι>}}
\label{nc61}

In \S\ref{nc3}, given $(M,\vp,g),\io,L$ as in Assumption \ref{nc3ass}, we wrote $\nu\ra L$ for the normal bundle of $L$ in $M$ and $r:\nu\ra[0,\iy)$ for the radius function on $\nu$, so that $r^{-1}(0)$ is the zero section $0(L)$. We chose a tubular neighbourhood $U_R\subset\nu$ of the zero section $0(L)$ in $\nu$ and a tubular neighbourhood map $\Up:U_R\ra M$, which is a diffeomorphism onto an open subset of $M$, where $U_R=\bigl\{(x,\al)\in\nu:r(x,\al)<R\bigr\}$ for small $R>0$. For $t>0$ we also defined $\Up_t:U_{t^{-1}R}\ra M$ by $\Up_t:(x,\al)\mapsto \Up(x,t \al)$, as in \eq{nc3eq5}. These are equivariant under the actions of $\{\pm 1\}$ on $\nu$ and $\an{\io}$ on $M$, and so they descend to~$\Up_t:U_{t^{-1}R}/\{\pm 1\}\ra M/\an{\io}$.

In \S\ref{nc4} we defined a 7-manifold $P$, a 5-submanifold $Q\subset P$, and a proper, continuous map $\rho:P\ra\nu/\{\pm 1\}$, where $\rho^{-1}(0(L))=Q$, and $\rho\vert_{P\sm Q}:P\sm Q\ra(\nu/\{\pm 1\})\sm 0(L)$ is smooth and a diffeomorphism. We defined $\check r=r\ci\rho:P\ra[0,\iy)$ to be the pullback of the radius function, so that~$Q=\check r^{-1}(0)$.

For $t>0$, define a compact, smooth 7-manifold $N$ by
\e
N=\bigl[\rho^{-1}(U_{t^{-1}R}/\{\pm 1\})\amalg (M\sm L)/\an{\io}\bigr]\big/\approx.
\label{nc6eq1}
\e
Here $\rho^{-1}(U_{t^{-1}R}/\{\pm 1\})\subset P$ is open, a noncompact 7-manifold, and $(M\sm L)/\an{\io}$ is the nonsingular part of the orbifold $M/\an{\io}$, another noncompact 7-manifold. We define $\approx$ to be the equivalence relation on the disjoint union of these two 7-manifolds which identifies $x\in\rho^{-1}(U_{t^{-1}R}/\{\pm 1\})\sm Q$ with $\Up_t\ci\rho(x)$ in $(M\sm L)/\an{\io}$. Since $\Up_t\ci\rho$ is a diffeomorphism between the relevant open subsets of $\rho^{-1}(U_{t^{-1}R}/\{\pm 1\})$ and $(M\sm L)/\an{\io}$, and the quotient topological space is compact and Hausdorff, this $N$ is a smooth compact 7-manifold.

Technically $N$ depends on $t>0$, so it might be better to write it as $N_t$ rather than $N$. But there are canonical diffeomorphisms $N_t\cong N_1$ for all $t>0$ acting as the identity on the subsets $(M\sm L)/\an{\io}\subset N_t$, $(M\sm L)/\an{\io}\subset N_1$, and identifying $\rho^{-1}(U_{t^{-1}R}/\{\pm 1\})\subset N_t$ with $\rho^{-1}(U_R/\{\pm 1\})\subset N_1$ by rescaling by $t$. So by an abuse of notation we will identify the $N_t$ for $t>0$, and write them all as~$N$.

By another abuse of notation, we will regard $\rho^{-1}(U_{t^{-1}R}/\{\pm 1\})$ and $(M\sm L)/\an{\io}$ as open subsets of $N$, and functions, metrics, and exterior forms defined on $\rho^{-1}(U_{t^{-1}R}/\{\pm 1\})\subset P$ or on $(M\sm L)/\an{\io}\subset M/\an{\io}$ as being defined on the corresponding subsets of $N$, without changing notation.

There is an obvious continuous map $\pi:N\ra M/\an{\io}$ acting by $\Up_t\ci\rho$ on $\rho^{-1}(U_{t^{-1}R}/\{\pm 1\})$, and by the inclusion $(M\sm L)/\an{\io}\hookra M/\an{\io}$ on $(M\sm L)/\an{\io}$. Then $\pi\vert_{N\sm Q}:N\sm Q\ra (M/\an{\io})\sm L$ is smooth, and a diffeomorphism. However, as in Remark \ref{nc2rem2} and \S\ref{nc41}, $\pi$ is not a smooth map of orbifolds near $Q$, although $\pi\vert_Q:Q\ra L$ is smooth and an~$\cS^2$-bundle.

Readers are warned about a possible source of confusion concerning the radius functions $r:\nu/\{\pm 1\}\ra[0,\iy)$ and $\check r:P\ra[0,\iy)$. The map $\Up:U_R/\{\pm 1\}\ra M/\an{\io}$ identifies $U_R/\{\pm 1\}$ with an open neighbourhood of $L$ in $M/\an{\io}$. We think of the radius function $r:U_R/\{\pm 1\}\ra[0,R)$ as being the distance to $L$ in~$M/\an{\io}$.

However, the map $\Up_t\ci\rho:\rho^{-1}(U_{t^{-1}R}/\{\pm 1\})\ra M/\an{\io}$ includes a rescaling by $t>0$ in the fibres of $\nu/\{\pm 1\}$. So $\check r$ maps $\rho^{-1}(U_{t^{-1}R}/\{\pm 1\})\ra[0,t^{-1}R)$, and $t\check r$ on $\rho^{-1}(U_{t^{-1}R}/\{\pm 1\})$ is identified with the function $r$ on $\Up(U_R/\{\pm 1\})\subset M/\an{\io}$. That is, {\it we should identify the radius function $r$ in most of\/} \S\ref{nc3} {\it with\/ $t\check r$ in\/} \S\ref{nc4}--\S\ref{nc5}, {\it not with\/} $\check r$, and we should think of $t\check r$ (rather than $\check r$) as being the distance to $Q$ in~$\rho^{-1}(U_{t^{-1}R}/\{\pm 1\})\subset P$.

Next we compute the Betti numbers $b^k(N)=\dim H^k(N;\R)$. It is enough to compute $b^0,\ldots,b^3$, as $b^{7-k}(N)=b^k(N)$ by Poincar\'e duality.

\begin{prop} Let\/ $M/\an{\io},L$ and\/ $N$ be as above. Then the Betti numbers of\/ $N$ are given by
\e
b^k(N)=b^k(M/\an{\io})+b^{k-2}(L).
\label{nc6eq2}
\e
Moreover, the fundamental groups of\/ $N$ and\/ $M/\an{\io}$ are isomorphic.
\label{nc6prop1}	
\end{prop}

\begin{proof} We have the following commutative diagram of topological spaces, where the rows are embeddings:
\e
\begin{gathered}
\xymatrix@C=100pt@R=15pt{ *+[r]{Q\,\,} \ar@{^{(}->}[r]_\inc \ar[d]^{\pi\vert_Q} & *+[l]{N} \ar[d]_\pi \\
*+[r]{L\,\,} \ar@{^{(}->}[r]^\inc & *+[l]{M/\an{\io}.\!} }	
\end{gathered}
\label{nc6eq3}
\e
Consider the following diagram of (relative) cohomology groups:
\ea
\text{\begin{footnotesize}$\displaystyle
\xymatrix@!0@C=40pt@R=25pt{ 
& 0 \ar[d] && 0 \ar[d] && 0 \ar[d] && 0 \ar[d] 
\\
*+[l]{\cdots} \ar[r] & H^{k-1}(L;\R) \ar[dd]^(0.4){\pi\vert_Q^*=\begin{pmatrix} \id \\ 0 \end{pmatrix}} \ar[rr] && H^k(M/\an{\io},L;\R) \ar[dd]_\cong^{\pi^*} \ar[rr] && H^k(M/\an{\io};\R) \ar[dd]^{\pi^*} \ar[rr]^{\inc^*} && H^k(L;\R) \ar[dd]^(0.4){\pi\vert_Q^*=\begin{pmatrix} \id \\ 0 \end{pmatrix}} \ar[r] & *+[r]{\cdots}
\\
\\
*+[l]{\cdots} \ar[r] & {\begin{subarray}{l} \ts H^{k-1}(Q;\R)\cong \\
\ts H^{k-1}(L;\R)\op \\ \ts H^{k-3}(L;\R) \end{subarray}} \ar[dd]^(0.6){\begin{pmatrix} 0 & \id  \end{pmatrix}} \ar[rr] && H^k(N,Q;\R) \ar[dd] \ar[rr] && H^k(N;\R) \ar@{.>}[dd] \ar[rr]^{\inc^*} && {\begin{subarray}{l} \ts H^k(Q;\R)\cong \\
\ts H^k(L;\R)\op \\ \ts H^{k-2}(L;\R) \end{subarray}} \ar[dd]^(0.6){\begin{pmatrix} 0 & \id  \end{pmatrix}} \ar[r] & *+[r]{\cdots}
\\
\\
*+[l]{\cdots} \ar[r] & H^{k-3}(L;\R) 
 \ar[d] \ar[rr] && 0 \ar[d] \ar[rr] && H^{k-2}(L;\R) \ar[d] \ar[rr]^\id && H^{k-2}(L;\R)  \ar[d] \ar[r] & *+[r]{\cdots}
\\
& 0 && 0  && 0  && 0.\! }$	\end{footnotesize}}
\nonumber\\[-16pt]
\label{nc6eq4}
\\[-16pt]
\nonumber
\ea
Here the first and second rows are the relative cohomology exact sequences for the pairs $(M/\an{\io},L)$ and $(N,Q)$. The squares between the first and second rows commute as \eq{nc6eq3} commutes. In the second column, $\pi^*:H^k(M/\an{\io},L;\R)\ra H^k(N,Q;\R)$ is an isomorphism since collapsing $L$ to a point in $M/\an{\io}$ yields the same pointed topological space as collapsing $Q$ to a point in $N$. Thus the second column is exact.

In the first and fourth columns, $\pi\vert_Q:Q\ra L$ is an $\cS^2$-bundle, which is trivial as $\nu\ra L$ is a trivial $\C^2$-bundle (see Remark~\ref{assocrmk}). Hence $Q\cong L\t\cS^2$, so
\e
H^k(Q;\R)\cong \bigop\nolimits_{i+j=k\!\!\!\!}H^i(L;\R)\ot H^j(\cS^2;\R)\cong H^k(L;\R)\op H^{k-2}(L;\R)
\label{nc6eq5}
\e
by the K\"unneth Theorem. Thus the first and fourth columns are exact. The bottom left square commutes trivially, and the third row is exact.

We now know that the rows of \eq{nc6eq4} are exact, the first, second and fourth columns are exact, the squares between the first and second row commute, and the left hand square between the second and third row commutes. By homological algebra, it follows that there is a unique morphism `$\dashra$' as shown in \eq{nc6eq4} which makes the third column exact. Hence $H^k(N;\R)\cong H^k(M/\an{\io};\R)\op H^{k-2}(L;\R)$. Taking dimensions proves~\eq{nc6eq2}.

For the final part, we can show that $\pi_*:\pi_1(N)\ra\pi_1(M/\an{\io})$ is an isomorphism using the Seifert--van Kampen Theorem and the fact that $\rho_*:\pi_1(P)\ra\pi_1(\nu/\{\pm 1\})$ is an isomorphism, since $\pi_1(P),\pi_1(\nu/\{\pm 1\})$ are both isomorphic to $\pi_1(L)$, as the fibres of $\si:P\ra L$ and $\pi:\nu/\{\pm 1\}\ra L$ are simply-connected.
\end{proof}

\subsection{\texorpdfstring{The $\Gt$-structures $(\vp_t^N,g_t^N)$ on $N$}{The G₂-structures (φᵗᴺ,gᵗᴺ) on N}}
\label{nc62}

Let $a:[0,\iy)\ra\R$ be a smooth function with $a(x)=0$ for $x\in[0,1]$, and $a(x)\in(0,1)$ for $x\in(1,2)$, and $a(x)=1$ for $x\in[2,\iy)$. Let $\eta$ and $\ze$ be as defined in \S\ref{nc35}. For some sufficiently small $\ep>0$ and all $t\in(0,\ep]$, we define a 3-form $\vp_t^N$ on $N$ by
\e
\vp_t^N\!=\!\begin{cases} \ti\vp_t^P+\d\bigl[t^2\al_{0,2}+t^4\al_{2,0}\bigr], & {\begin{subarray}{l}\ts \text{if $\check r\le t^{-1/9}$} \\ \ts\text{in $\rho^{-1}(U_{t^{-1}R}/\{\pm 1\}),$}\end{subarray}} \\[8pt]
{\begin{subarray}{l}\ts \ti\vp_t^P\!+\!\d\bigl[t^2\al_{0,2}\!+\!t^4\al_{2,0}\!+\!a(t^{1/9}\check r)\!\cdot\!\Up_*(\eta)\bigr],
\end{subarray}}
& {\begin{subarray}{l}\ts\text{if $t^{-1/9}\le\check r\le 2t^{-1/9}$} \\ \ts\text{in $\rho^{-1}(U_{t^{-1}R}/\{\pm 1\}),$}\end{subarray}} \\[8pt]
{\begin{subarray}{l}\ts\ti\vp_t^P+\d\bigl[t^2\al_{0,2}+t^4\al_{2,0} +\Up_*(\eta)\bigr]=\\
\ts \rho^*(\ti\vp_t^\nu)\!+\!\d\bigl[t^2\tau_{1,1}\!+\!t^2\al_{0,2}\!+\!t^4\al_{2,0}\!+\!\Up_*(\eta)\bigr],\!\!\! \end{subarray}}
& {\begin{subarray}{l}\ts\text{if $2t^{-1/9}\le\check r\le t^{-4/5}$} \\ \ts \text{in $\rho^{-1}(U_{t^{-1}R}/\{\pm 1\}),$}\end{subarray}} \\[8pt]
{\begin{subarray}{l}\ts \rho^*(\ti\vp_t^\nu)+\d\bigl[(1-a(t^{4/5}\check r))\cdot{} \\
\ts (t^2\tau_{1,1}\!+\!t^2\al_{0,2}\!+\!t^4\al_{2,0})\!+\!\Up_*(\eta)\bigr], \end{subarray}} 
& {\begin{subarray}{l}\ts\text{if $t^{-4/5}\le\check r\le 2t^{-4/5}$} \\ \ts\text{in $\rho^{-1}(U_{t^{-1}R}/\{\pm 1\}),$}\end{subarray}} \\[8pt]
\rho^*(\ti\vp_t^\nu)\!+\!\d\Up_*(\eta)\!=\!\Up_*(\ti\vp^\nu\!+\!\d\eta)\!=\!\vp, & {\begin{subarray}{l}\ts\text{if $2t^{-4/5}\le\check r<t^{-1}R$} \\ \ts\text{in $\rho^{-1}(U_{t^{-1}R}/\{\pm 1\})$,}\end{subarray}} \\[8pt]
\vp, &  \text{in $(M\sm \Up(U_R))/\an{\io}$.}
\end{cases}
\label{nc6eq6}
\e
Here we require $\ep>0$ to be small enough that if $0<t\le\ep$ then
\begin{equation*}
1<t^{-1/9}<2t^{-1/9}<t^{-4/5}<2t^{-4/5}<t^{-1}R,	
\end{equation*}
so that each region in \eq{nc6eq6} is nonempty. These particular exponents of $-\tfrac{1}{9}$ and $-\tfrac{4}{5}$ have been chosen so that the torsion estimates computed in Proposition \ref{nc6prop2} below will be sufficiently small to produce estimates in Proposition \ref{nc6prop3} that will satisfy the hypotheses of Theorem \ref{nc2thm3}.

The equality in the third case in \eq{nc6eq6} comes from \eq{nc4eq32}, and the equalities in the fifth case come from \eq{nc3eq33} and the equivalence of $\rho^*(\ti\vp_t^\nu)$ and $\Up_*(\ti\vp^\nu)$ under the identification of their domains by $\approx$ in \eq{nc6eq1}. Considering the transitions between regions we see that $\vp_t^N$ is smooth on $N$, and since $\ti\vp^P_t,\ti\vp^\nu_t,\vp$ are closed, $\vp_t^N$ is also closed.

Similarly, we define a smooth closed 4-form $\psi_t^N$ on $N$ for all $t\in(0,\ep]$ by
\e
\psi_t^N\!=\!\begin{cases} \ti\psi_t^P+\d\bigl[t^2\be_{0,3}+t^4\be_{2,1}\bigr], & {\begin{subarray}{l}\ts \text{if $\check r\le t^{-1/9}$} \\ \ts\text{in $\rho^{-1}(U_{t^{-1}R}/\{\pm 1\}),$}\end{subarray}} \\[8pt]
{\begin{subarray}{l}\ts \ti\psi_t^P\!+\!\d\bigl[t^2\be_{0,3}\!+\!t^4\be_{2,1}\!+\!a(t^{1/9}\check r)\!\cdot\!\Up_*(\ze)\bigr],
\end{subarray}}
& {\begin{subarray}{l}\ts\text{if $t^{-1/9}\le\check r\le 2t^{-1/9}$} \\ \ts\text{in $\rho^{-1}(U_{t^{-1}R}/\{\pm 1\}),$}\end{subarray}} \\[8pt]
{\begin{subarray}{l}\ts\ti\psi_t^P+\d\bigl[t^2\be_{0,3}+t^4\be_{2,1} +\Up_*(\ze)\bigr]=\\
\ts \rho^*(\ti\psi_t^\nu)\!+\!\d\bigl[t^2\up_{1,2}\!+\!t^2\be_{0,3}\!+\!t^4\be_{2,1}\!+\!\Up_*(\ze)\bigr],\!\!\! \end{subarray}}
& {\begin{subarray}{l}\ts\text{if $2t^{-1/9}\le\check r\le t^{-4/5}$} \\ \ts \text{in $\rho^{-1}(U_{t^{-1}R}/\{\pm 1\}),$}\end{subarray}} \\[8pt]
{\begin{subarray}{l}\ts \rho^*(\ti\psi_t^\nu)+\d\bigl[(1-a(t^{4/5}\check r))\cdot{} \\
\ts (t^2\up_{1,2}\!+\!t^2\be_{0,3}\!+\!t^4\be_{2,1})\!+\!\Up_*(\ze)\bigr], \end{subarray}} 
& {\begin{subarray}{l}\ts\text{if $t^{-4/5}\le\check r\le 2t^{-4/5}$} \\ \ts\text{in $\rho^{-1}(U_{t^{-1}R}/\{\pm 1\}),$}\end{subarray}} \\[8pt]
\rho^*(\ti\psi_t^\nu)\!+\!\d\Up_*(\ze)\!=\!\Up_*(\ti\psi^\nu\!+\!\d\ze)\!=\!*\vp, & {\begin{subarray}{l}\ts\text{if $2t^{-4/5}\le\check r<t^{-1}R$} \\ \ts\text{in $\rho^{-1}(U_{t^{-1}R}/\{\pm 1\})$,}\end{subarray}} \\[8pt]
*\vp, &  \text{in $(M\sm \Up(U_R))/\an{\io}$.}
\end{cases}
\label{nc6eq7}
\e

We claim that for $\ep>0$ small enough, $\vp_t^N$ will be a positive 3-form on $N$ for all $t\in(0,\ep]$ in the sense of \S\ref{nc22}. To establish this, we show that if $t$ is small enough, then at all points in $N$ the $3$-form $\vp_t^N$ is close to a positive $3$-form, because then Proposition \ref{nc2prop2} implies that $\vp_t^N$ will be positive when $t$ is sufficiently small. Recall that $r = t \check r$. If we consider $\vp_t^N - \vp_t^P$, then this difference includes terms $t^2 \xi_{0,3}$ and $\d\Upsilon_*(\eta)$, which by \eq{nc4eq24}, \eq{nc3eq35}, and Proposition \ref{nc4prop}, are of size $O(r^2) = O(t^2 \check r^2)$. These are too big when $r = O(1)$, which happens when $\check r = O(t^{-1})$. So we can approximate $\vp_t^N$ by $\vp_t^P$ only when $\check r$ is not too large, and the region $\check r \leq t^{-1/2}R$ will do. On the other hand, if we consider $\vp_t^N - \vp$, then this difference includes a term $t^2 \d \tau_{1,1}$, which by \eq{nc4eq34} is of size $O(\check r^{-4})$. This is too big when $\check r = O(1)$, which happens when $r = O(t)$. So we can approximate $\vp_t^N$ by $\vp$ only when $\check r$ is not too small, and the region $\check r \geq t^{-1/2}R$ will do. Together these regions cover all points of $N$, so $\vp_t^N$ is indeed positive for all $t\in(0,\ep]$ if $\ep$ is sufficiently small.

\subsection{\texorpdfstring{Estimating the torsion of $(\vp_t^N,g_t^N)$}{Estimating the torsion of (φᵗᴺ,gᵗᴺ)}}
\label{nc63}

Since $\d\vp_t^N=0$ and $\d\psi_t^N=0$, if $\Th(\vp_t^N)-\psi_t^N=0$ then the $\Gt$-structure $(\vp_t^N,g_t^N)$ is torsion-free. Thus, we can regard the next two propositions, which bound $\Th(\vp_t^N)-\psi_t^N$, as measuring the torsion of~$(\vp_t^N,g_t^N)$.

\begin{prop} In the situation above, for all\/ $t\in(0,\ep]$ we have
\ea
\begin{split}
&\bmd{\Th(\vp_t^N)-\psi_t^N}{}_{g_t^N}=\\
&\begin{cases} O(t^2), & \text{if\/ $\check r\le 1$ in $\rho^{-1}(U_{t^{-1}R}/\{\pm 1\}),$} \\
O(t^2\check r^2), & \text{if\/ $1\le\check r\le t^{-1/9}$ in $\rho^{-1}(U_{t^{-1}R}/\{\pm 1\}),$} \\
O(t^{16/9}),
& \text{if\/ $t^{-1/9}\le\check r\le 2t^{-1/9}$ in $\rho^{-1}(U_{t^{-1}R}/\{\pm 1\}),$} \\
O(t^2\check r^{-2+\ga}),
& \text{if\/ $2t^{-1/9}\le\check r\le t^{-4/5}$ in $\rho^{-1}(U_{t^{-1}R}/\{\pm 1\}),$ $\ga>0,$} \\
O(t^{16/5}), 
& \text{if\/ $t^{-4/5}\le\check r\le 2t^{-4/5}$ in $\rho^{-1}(U_{t^{-1}R}/\{\pm 1\}),$} \\
0, & \text{if\/ $2t^{-4/5}\le\check r<t^{-1}R$ in $\rho^{-1}(U_{t^{-1}R}/\{\pm 1\}),$} \\
0, & \text{in $(M\sm \Up(U_R))/\an{\io},$}
\end{cases}
\end{split}
\label{nc6eq8}
\allowdisplaybreaks\\
\begin{split}
&\bmd{\d\bigl(\Th(\vp_t^N)-\psi_t^N\bigr)}{}_{g_t^N}=\\
&\begin{cases} O(t), & \text{if\/ $\check r\le 1$ in $\rho^{-1}(U_{t^{-1}R}/\{\pm 1\}),$} \\
O(t\check r), & \text{if\/ $1\le\check r\le t^{-1/9}$ in $\rho^{-1}(U_{t^{-1}R}/\{\pm 1\}),$} \\
O(t^{8/9}),
& \text{if\/ $t^{-1/9}\le\check r\le 2t^{-1/9}$ in $\rho^{-1}(U_{t^{-1}R}/\{\pm 1\}),$} \\
O(t\check r^{-3+\ga}),
& \text{if\/ $2t^{-1/9}\le\check r\le t^{-4/5}$ in $\rho^{-1}(U_{t^{-1}R}/\{\pm 1\}),$ $\ga>0,$} \\ 
O(t^3), 
& \text{if\/ $t^{-4/5}\le\check r\le 2t^{-4/5}$ in $\rho^{-1}(U_{t^{-1}R}/\{\pm 1\}),$} \\
0, & \text{if\/ $2t^{-4/5}\le\check r<t^{-1}R$ in $\rho^{-1}(U_{t^{-1}R}/\{\pm 1\}),$} \\
0, &  \text{in $(M\sm \Up(U_R))/\an{\io}$.}
\end{cases}
\end{split}
\label{nc6eq9}
\ea
\label{nc6prop2}	
\end{prop}

\begin{proof} For the first two cases in \eq{nc6eq8}, if $\check r\le t^{-1/9}$ in $\rho^{-1}(U_{t^{-1}R}/\{\pm 1\})$, we have
\ea
\Th(\vp_t^N)-\psi_t^N&=\Th\bigl(\ti\vp_t^P+\d[t^2\al_{0,2}+t^4\al_{2,0}]\bigr)-\ti\psi_t^P-\d[t^2\be_{0,3}+t^4\be_{2,1}]
\nonumber\\
&=\Th\bigl(\vp_t^P+t^2\xi_{1,2}+t^2\xi_{0,3}+t^2\d\al_{0,2}+t^4\d\al_{2,0}\bigr)
\nonumber\\
&\qquad-\psi_t^P-t^2\chi_{1,3}-t^4\th_{3,1}-t^4\th_{2,2}-t^2\d\be_{0,3}-t^4\d\be_{2,1}
\nonumber\\
&=(D_{\vp_t^P}\Th)\bigl(t^2\xi_{1,2}+t^2\xi_{0,3}+t^2\d\al_{0,2}+t^4\d\al_{2,0}\bigr)
\nonumber\\
&\qquad -t^2\chi_{1,3}-t^4\th_{3,1}-t^4\th_{2,2}-t^2\d\be_{0,3}-t^4\d\be_{2,1}
\label{nc6eq10}\\
&\qquad +F_{\vp_t^P}\bigl(t^2\xi_{1,2}+t^2\xi_{0,3}+t^2\d\al_{0,2}+t^4\d\al_{2,0}\bigr)
\nonumber\\
&=(D_{\vp_t^P}\Th)\bigl(t^2\xi_{0,3}+t^2[\d\al_{0,2}]_{0,3}+t^4[\d\al_{2,0}]_{2,1}+t^4[\d\al_{2,0}]_{1,2}\bigr)
\nonumber\\
&\qquad -t^4\th_{2,2}-t^4[\d\be_{2,1}]_{2,2}-t^4[\d\be_{2,1}]_{1,3}
\nonumber\\
&\qquad +F_{\vp_t^P}\bigl(t^2\xi_{1,2}+t^2\xi_{0,3}+t^2\d\al_{0,2}+t^4\d\al_{2,0}\bigr),
\nonumber
\ea
using \eq{nc6eq6}--\eq{nc6eq7} in the first step; \eq{nc4eq24} in the second; Proposition \ref{nc2prop2} and $*_{\vp_t^P}\vp_t^P=\psi_t^P$ in the third; and $\d\al_{0,2}=[\d\al_{0,2}]_{1,2}+[\d\al_{0,2}]_{0,3}$, etc., and equation \eq{nc5eq1} of Theorem \ref{nc5thm} to cancel seven terms in the fourth.

If $\ep>0$ is small enough then $\bmd{\vp_t^N-\vp_t^P}{}_{g_t^P}$ is small on $\check r\le t^{-1/9}$, so \eq{nc2eq5} in Proposition \ref{nc2prop2} applies for some $C>0$. Also \eq{nc2eq6} implies that $\bmd{(D_{\vp_t^P}\Th)(\al)}{}_{g_t^P}\le\frac{4}{3}\bmd{\al}{}_{g_t^P}$ for any 3-form $\al$. Thus from \eq{nc6eq10} we deduce that if $\check r\le t^{-1/9}$ we have
\ea
&\bmd{\Th(\vp_t^N)-\psi_t^N}{}_{g_t^P}\le \ts\frac{4}{3}\bigl(\bmd{t^2\xi_{0,3}}{}_{g_t^P}+\bmd{t^2[\d\al_{0,2}]_{0,3}}{}_{g_t^P}+\bmd{t^4[\d\al_{2,0}]_{2,1}}{}_{g_t^P}
\nonumber\\
&\quad +\bmd{t^4[\d\al_{2,0}]_{1,2}}{}_{g_t^P}\bigr)+\bmd{t^4\th_{2,2}}{}_{g_t^P}+\bmd{t^4[\d\be_{2,1}]_{2,2}}{}_{g_t^P}+\bmd{t^4[\d\be_{2,1}]_{1,3}}{}_{g_t^P}
\nonumber\\
&\quad +C\bigl(\bmd{t^2\xi_{1,2}}{}_{g_t^P}+\bmd{t^2\xi_{0,3}}{}_{g_t^P}+\bmd{t^2\d\al_{0,2}}{}_{g_t^P}+\bmd{t^4\d\al_{2,0}}{}_{g_t^P}\bigr)^2.
\label{nc6eq11}
\ea

All the terms on the right hand side of \eq{nc6eq11} are estimated in Proposition \ref{nc4prop} and Theorem \ref{nc5thm}. The dominant terms are $\bmd{t^2\xi_{0,3}}{}_{g_t^P}$, $\bmd{t^4\th_{2,2}}{}_{g_t^P}$, which are $O(t^2)$ if $\check r\le 1$ and $O(t^2\check r^2)$ if $1\le\check r\le t^{-1/9}$. Thus we see that
\e
\bmd{\Th(\vp_t^N)-\psi_t^N}{}_{g_t^P}=\begin{cases} O(t^2), & \text{if $\check r\le 1$ in $\rho^{-1}(U_{t^{-1}R}/\{\pm 1\}),$} \\
O(t^2\check r^2), & \text{if $1\le\check r\le t^{-1/9}$ in $\rho^{-1}(U_{t^{-1}R}/\{\pm 1\})$.} \end{cases}
\label{nc6eq12}
\e
This is not quite what is wanted for \eq{nc6eq8}, as we are taking norms using $g_t^P$ rather than $g_t^N$. However, $g_t^P$ and $g_t^N$ are $C^0$-close in the region $\check r\le t^{-1/9}$, so $\md{\cdots}_{g_t^P}$ and $\md{\cdots}_{g_t^N}$ differ by a bounded factor, and thus \eq{nc6eq12} implies the first two cases in \eq{nc6eq8}.

For the third case in \eq{nc6eq8}, when $t^{-1/9}\le\check r\le 2t^{-1/9}$, we proceed as in \eq{nc6eq10}--\eq{nc6eq12} but including the extra terms from \eq{nc6eq6}--\eq{nc6eq7}:
\begin{equation*}
(D_{\vp_t^P}\Th)\bigl(\d[a(t^{1/9}\check r)\cdot\Up_*(\eta)]\bigr) - \d[a(t^{1/9}\check r)\cdot\Up_*(\ze)].
\end{equation*}
Using \eq{nc3eq35} and \eq{nc3eq36}, and noting as in \S\ref{nc61} that $r$ in \S\ref{nc35} corresponds to $t\check r$ above, we find that these terms are $O(t^{16/9})$, agreeing with $O(t^2\check r^2)$ in the previous case when $\check r=O(t^{-1/9})$. The third case in \eq{nc6eq8} follows. 

For the fourth case, if $2t^{-1/9}\le\check r\le t^{-4/5}$ in $\rho^{-1}(U_{t^{-1}R}/\{\pm 1\})$ we have
\ea
&\Th(\vp_t^N)-\psi_t^N 
\nonumber\\
& =\Th\bigl(\ti\vp_t^P+\d[t^2\al_{0,2}+t^4\al_{2,0}+\Up_*(\eta)]\bigr) 
\nonumber\\
&\quad {}-\bigl(\ti\psi_t^P+\d[t^2\be_{0,3}+t^4\be_{2,1}+\Up_*(\ze)]\bigr)-\bigl\{\Th(\vp)-\psi\bigr\} 
\nonumber\\
&=\Th\bigl(\vp_t^P+t^2\xi_{1,2}+t^2\xi_{0,3}+t^2\d\al_{0,2}+t^4\d\al_{2,0}+\Up_*(\d\eta)\bigr) 
\nonumber\\
&\quad {}-\psi_t^P-t^2\chi_{1,3}-t^4\th_{3,1}-t^4\th_{2,2}-t^2\d\be_{0,3}-t^4\d\be_{2,1}-\Up_*(\d\ze) 
\nonumber\\
&\quad {}-\bigl\{\Th\bigl(\rho^*(\vp^\nu_t+t^2\dddot\vp^2_{0,3})+\Up_*(\d\eta)\bigr)-\bigl(\rho^*(\psi^\nu_t+t^4\dddot\psi^4_{2,2})+\Up_*(\d\ze)\bigr)\bigr\} 
\nonumber\\
&=(D_{\vp_t^P}\Th)\bigl(t^2\xi_{1,2}+t^2\xi_{0,3}+t^2\d\al_{0,2}+t^4\d\al_{2,0}+\Up_*(\d\eta)\bigr) -t^2\chi_{1,3}-t^4\th_{3,1} 
\nonumber\\
&\quad {}-t^2\d\be_{0,3}-t^4\d\be_{2,1}+F_{\vp_t^P}\bigl(t^2\xi_{1,2}+t^2\xi_{0,3}+t^2\d\al_{0,2}+t^4\d\al_{2,0}\!+\!\Up_*(\d\eta)\bigr) 
\nonumber\\
&\quad {}-(D_{\rho^*(\vp^\nu_t)}\Th)\bigl(t^2\xi_{0,3}-t^2[\d\tau_{1,1}]_{0,3}+\Up_*(\d\eta)\bigr) 
\nonumber\\
&\quad {}-F_{\rho^*(\vp^\nu_t)}\bigl(t^2\xi_{0,3}-t^2[\d\tau_{1,1}]_{0,3}+\Up_*(\d\eta)\bigr) 
\nonumber\\
&=(D_{\vp_t^P}\Th)\bigl(t^2[\d\al_{0,2}]_{0,3}+t^4[\d\al_{2,0}]_{2,1}+t^4[\d\al_{2,0}]_{1,2}\bigr)+(D_{\rho^*(\vp^\nu_t)}\Th)\bigl(t^2[\d\tau_{1,1}]_{0,3}\bigr) 
\nonumber\\
&\quad{} +\bigl(D_{\vp_t^P}\Th-D_{\rho^*(\vp^\nu_t)}\Th\bigr)\bigl(t^2\xi_{0,3}+\Up_*(\d\eta)\bigr)-t^4[\d\be_{2,1}]_{2,2}-t^4[\d\be_{2,1}]_{1,3} 
\nonumber\\
&\quad{} +F_{\vp_t^P}\bigl(t^2\xi_{1,2}+t^2\xi_{0,3}+t^2\d\al_{0,2}+t^4\d\al_{2,0}\!+\!\Up_*(\d\eta)\bigr) 
\nonumber\\
&\quad{} -F_{\rho^*(\vp^\nu_t)}\bigl(t^2\xi_{0,3}-t^2[\d\tau_{1,1}]_{0,3}+\Up_*(\d\eta)\bigr).
\label{nc6eq13}
\ea
Here we use \eq{nc6eq6}--\eq{nc6eq7} and $\Th(\vp)=\psi$ in the first step; equations \eq{nc3eq33}, \eq{nc4eq24} and $\Up_*(\ti\vp^\nu)=\rho^*(\ti\vp_t^\nu)$, $\Up_*(\ti\psi^\nu)=\rho^*(\ti\psi_t^\nu)$, and Corollary \ref{nc3cor} in the second; equation \eq{nc4eq33}, Proposition \ref{nc2prop2} and $*_{\vp_t^P}\vp_t^P=\psi_t^P$, $*_{\rho^*(\vp_t^\nu)}\rho^*(\vp_t^\nu)=\rho^*(\psi_t^\nu)$ in the third, in particular to cancel $t^4 \th_{2,2}$ with $t^4 \dddot \psi_{2,2}$ and to cancel two $\Up_* (\d \ze)$ terms; and $\d\al_{0,2}=[\d\al_{0,2}]_{1,2}+[\d\al_{0,2}]_{0,3}$, etc., and equation \eq{nc5eq1} of Theorem \ref{nc5thm} to cancel seven terms in the fourth.

Just as we did for \eq{nc6eq11}, we can now estimate all the terms in the final step of \eq{nc6eq13} when  $2t^{-1/9}\le\check r\le t^{-4/5}$, using Propositions \ref{nc2prop2} and \ref{nc4prop}, Theorem \ref{nc5thm}, equation \eq{nc4eq34}, and the estimate
\begin{equation} \label{nc6eq14}
\bmd{D_{\vp_t^P}\Th-D_{\rho^*(\vp^\nu_t)}\Th}_{g_t^P}=O(\check r^{-4}).
\end{equation}
The estimate \eq{nc6eq14} is immediate from $\bmd{\vp_t^P-\rho^*(\vp^\nu_t)}=O(\check r^{-4})$ which itself follows from \eq{nc4eq30} and \eq{nc4eq34}. In this case we find that the dominant terms come from $t^2[\d\al_{0,2}]_{0,3}$, $t^4[\d\al_{2,0}]_{2,1}$, $t^4[\d\be_{2,1}]_{2,2}$ which contribute $O(t^2\check r^{-2+\ga})$ for $\ga>0$ by Theorem \ref{nc5thm}, and all other terms are smaller. This proves the fourth case in \eq{nc6eq8}.

The fifth case of \eq{nc6eq8} is as for the fourth case with $\check r=O(t^{-4/5})$, but with additional terms involving
\begin{equation*}
\d\bigl[a(t^{4/5}\check r)\cdot(t^2\tau_{1,1}+t^2\al_{0,2}+t^4\al_{2,0})\bigr],\;\>
\d\bigl[a(t^{4/5}\check r)\cdot(t^2\up_{1,2}+t^2\be_{0,3}+t^4\be_{2,1})\bigr].
\end{equation*}
Using \eq{nc4eq34} and Theorem \ref{nc5thm} we find that these contribute errors of size $O(t^{16/5})$ and $O(t^{17/5-4\ga/5})$ for $\ga>0$. The $O(t^{16/5})$ error dominates, so the fifth case follows.

The sixth and seventh cases are immediate as $\vp_t^N=\vp$, $\psi_t^N=*\vp$ with $\Th(\vp)=*\vp$. This proves \eq{nc6eq8}.

Finally, as in \eq{nc4eq34} and Proposition \ref{nc4prop} and Theorem \ref{nc5thm}, the principle is that adding a derivative $\nabla$ multiplies the estimates by $t^{-1}$ if $\check r\le 1$ and by $t^{-1}\check r^{-1}$ if $\check r\ge 1$, so \eq{nc6eq9} follows easily in the same way as \eq{nc6eq8}.
\end{proof}

We can now estimate the norms of $\Th(\vp_t^N)-\psi_t^N$ needed in Theorem \ref{nc2thm3}.

\begin{prop} For all\/ $t\in(0,\ep]$ we have
\e
\begin{gathered}
\bnm{\Th(\vp_t^N)-\psi_t^N}_{C^0} =O(t^{16/9}),
\qquad
\bnm{\Th(\vp_t^N)-\psi_t^N}_{L^2} =O(t^{32/9}),
\\
\text{and}\qquad \bnm{\d(\Th(\vp_t^N)-\psi_t^N)}_{L^{14}} =O(t^{8/7}),
\end{gathered}
\label{nc6eq15}
\e
where the norms are computed using the metric $g_t^N$ on $N$.
\label{nc6prop3}	
\end{prop}

\begin{proof} Table \ref{nc6tab1} gives the orders of the contributions to the norms in \eq{nc6eq15} from each region in \eq{nc6eq8}--\eq{nc6eq9}. We can compute each of these contributions using \eq{nc6eq8}--\eq{nc6eq9} and the facts, which are obvious from the construction, that the region $\check r\le 1$ in $(N,g_t^P)$ has volume $O(t^4)$, and for $1 \le \check r \le 2t^{-4/5}$, the region in $(N,g_t^P)$ with $s \le \check r \le s + \de s$ has volume $O(t^4 s^3 \de s)$ for small~$\de s>0$.

\begin{table}[htb]
\begin{center}
\begin{tabular}{ |c|c|c|c| } 
\hline
$\vphantom{{}^(_(}$ & $\!\!\bnm{\Th(\vp_t^N)\!-\!\psi_t^N}_{C^0}\!\!$ & $\!\!\bnm{\Th(\vp_t^N)\!-\!\psi_t^N}_{L^2}\!\!$ & $\!\!\bnm{\d(\Th(\vp_t^N)\!-\!\psi_t^N)}_{L^{14}}\!\!$ \\
\hline
$\check r\le 1$ & $O(t^2)$ & $O(t^4)$ & $O(t^{9/7})$ \\
\hline
$1\le\check r\le t^{-1/9}$ & $O(t^{16/9})$ & $O(t^{32/9})$ & $O(t^{8/7})$ \\
\hline
$\!\!t^{-1/9}\!\le\!\check r\!\le\! 2t^{-1/9}\!\!$ & $O(t^{16/9})$ & $O(t^{32/9})$& $O(t^{8/7})$ \\
\hline
$\!\!2t^{-1/9}\!\le\!\check r\!\le\! t^{-4/5}\!\!$ & $O(t^{20/9-\ga/9})$ & $O(t^{4-4\ga/5})$ & $O(t^{100/63-\ga/9})$ \\
\hline
$\!\!t^{-4/5}\!\le\!\check r\!\le\! 2t^{-4/5}\!\!$ & $O(t^{16/5})$ & $O(t^{18/5})$ & $O(t^{107/35})$ \\
\hline
$r\ge 2t^{-4/5}$ & 0 & 0 & 0 \\
\hline
\end{tabular}
\end{center}
\caption{Contributions of regions to norms of $\Th(\vp_t^N)-\psi_t^N$, any $\ga>0$}
\label{nc6tab1}
\end{table}

The $O(t^4)$ volume estimate in the region $\check r \leq 1$ gives the first row of Table \ref{nc6tab1}. If a quantity is $O(t^{a} \check r^b)$ for $b \neq -2$ in the region $A \leq \check r \leq B$, then its $L^2$ norm on this region is given by
\begin{equation*}
\bnm{ O(t^a \check r^b) }_{L^2} = O \left( \int_A^B (t^a s^b)^2 t^4 s^3 ds \right)^{\frac{1}{2}} = O \bigl( t^{a + 2}( A^{b + 2} + B^{b + 2}) \bigr).
\end{equation*}
Similarly if a quantity is $O(t^{a} \check r^b)$ for $b \neq -\tfrac{2}{7}$ in the region $A \leq \check r \leq B$, then its $L^{14}$ norm on this region is given by
\begin{equation*}
\bnm{ O(t^a \check r^b) }_{L^{14}} = O \left( \int_A^B (t^a s^b)^{14} t^4 s^3 ds \right)^{\frac{1}{14}} = O \bigl( t^{a + \tfrac{2}{7}}( A^{b + \tfrac{2}{7}} + B^{b + \tfrac{2}{7}}) \bigr).
\end{equation*}
Using these observations, one obtains the entries of Table \ref{nc6tab1} easily. Taking the largest term in each column in Table \ref{nc6tab1} for small $t$ proves~\eq{nc6eq15}.
\end{proof}

\subsection{Proof of the main theorem}
\label{nc64}

We can now at last prove the main theorem of this paper:

\begin{thm} Let\/ $(M,\vp,g)$ be a compact, torsion-free $\Gt$-manifold, and let\/ $\io:M\ra M$ be a nontrivial involution preserving $(\vp,g),$ so that the fixed locus $L$ of\/ $\io$ is a compact associative $3$-fold in $M$ by Proposition\/~{\rm\ref{nc2prop3}}. 

Suppose $L$ is nonempty, and that there exists a closed, coclosed, nonvanishing $1$-form $\la$ on $L$. That is, $\la\in\Om^1(L)$ with\/ $\d\la=\d^*\la=0,$ where $\d^*$ is defined using $g\vert_L,$ and\/ $\la\vert_x\ne 0$ in $T_x^*L$ for all\/~$x\in L$.

Then there exists a compact\/ $7$-manifold\/ $N$ defined as a resolution of singularities $\pi:N\ra M/\an{\io}$ of the $7$-orbifold\/ $M/\an{\io}$ along its singular locus $L\subset M/\an{\io},$ by gluing in a bundle $\si:P\ra L$ along $L$, with fibre the Eguchi--Hanson space $X$ from {\rm\S\ref{nc25},} where $P$ is constructed using $\la$. The preimage $\pi^{-1}(L)$ is a $5$-submanifold\/ $Q$ of\/ $N,$ and\/ $\pi\vert_Q:Q\ra L$ is a smooth bundle with fibre\/ $\cS^2$. The fundamental group of\/ $N$ satisfies $\pi_1(N)\!\cong\!\pi_1(M/\an{\io}),$ and the Betti numbers~are 
\e
b^k(N)=b^k(M/\an{\io})+b^{k-2}(L).
\label{nc6eq16}
\e

There exists a smooth family $(\ti\vp_t^N,\ti g_t^N)$ of torsion-free $\Gt$-structures on $N$ for $t\in(0,\ep],$ with\/ $\ep>0$ small, such that\/ $(\ti\vp_t^N,\ti g_t^N)\ra\pi^*(\vp,g)$ in $C^0$ away from $Q$ as $t\ra 0,$ and for each\/ $x\in L$ the fibre $\pi^{-1}(x)\cong\cS^2$ with metric $\ti g_t^N\vert_{\pi^{-1}(x)}$ approximates a small round\/ $2$-sphere with area $\pi t^2\bmd{\la\vert_x}$ for small\/ $t$. The metrics $\ti g_t^N$ on $N$ have holonomy $\Gt$ if and only if\/ $M/\an{\io}$ has finite fundamental group.

\label{nc6thm}	
\end{thm}

\begin{proof} The first part of the theorem, the construction of $\pi:N\ra M/\an{\io}$ and $Q$, was done in \S\ref{nc61}, and the claims on fundamental group and Betti numbers are proved in Proposition \ref{nc6prop1}. For the second part, we will apply Theorem \ref{nc2thm3} to the family of $\Gt$-structures $(\vp_t^N,g_t^N)$ and 4-forms $\psi_t^N$ on $N$ for $t\in(0,\ep]$ defined in \S\ref{nc62}. By Proposition \ref{nc6prop3}, to ensure that Theorem \ref{nc2thm3}(i) holds for $(\vp_t^N,g_t^N)$, $\psi_t^N$ for all $t\in(0,\ep]$, we need to find $\al > 0$ such that $\al < \tfrac{16}{9}$, $\tfrac{7}{2} + \al < \tfrac{32}{9}$, and $-\tfrac{1}{2} + \al < \tfrac{8}{7}$. In particular, these all hold for $\al < \tfrac{1}{18}$.

We can see from the construction that $g_t^N$ satisfies Theorem \ref{nc2thm3}(ii),(iii) for some $K_2,K_3>0$ and all $t\in(0,\ep]$, since where $\check r\le 1$ the local injectivity radius of $g_t^N$ is $O(t)$ and the curvature is $O(t^{-2})$, and where $1\le \check r\le t^{-1}R$ the local injectivity radius of $g_t^N$ is $O(t\check r)$ and the curvature is $O(t^{-2}\check r^{-2})$, and outside $\rho^{-1}(U_{t^{-1}R}/\{\pm 1\})$ the local injectivity radius and curvature of $g_t^N$ are both~$O(1)$.

Thus Theorem \ref{nc2thm3} shows that after making $\ep>0$ smaller if necessary, for all $t\in(0,\ep]$ there exists a smooth, torsion-free $\Gt$-structure $(\ti\vp_t^N,\ti g_t^N)$ on $N$ such that $\nm{\ti\vp_t^N-\vp_t^N}_{C^0}\le K_4 \,t^{1/18}$ for some $K_4>0$ independent of $t$, computing the norm using $g_t^N$, and $[\ti\vp_t^N]=[\vp_t^N]$ in $H^3(M,\R)$. The proof of Theorem \ref{nc2thm3} in \cite{Joyc3} implies that $(\ti\vp_t^N,\ti g_t^N)$ depends smoothly on $t\in(0,\ep]$, as $(\vp_t^N,g_t^N)$ does.

The claims on convergence $(\ti\vp_t^N,\ti g_t^N)\ra\pi^*(\vp,g)$ away from $Q$ as $t\ra 0$, and $(\pi^{-1}(x),\ti g_t^N\vert_{\pi^{-1}(x)})$ approximating  a small round 2-sphere with area $\pi\bmd{\la\vert_x}t^2$ for $x\in L$ and small $t$, hold as $\nm{\ti\vp_t^N-\vp_t^N}_{C^0}\le K_4 \,t^{1/18}$ and the analogous claims hold for $(\vp_t^N,g_t^N)$ by construction, using the fact that the area of the bolt is $\pi\bmd{\la\vert_x}t^2$ as in Remark \ref{nc2rem3}. Since $\pi_1(N)\cong\pi_1(M/\an{\io})$, the last part follows from Theorem \ref{nc2thm2}. This completes the proof.
\end{proof}

\begin{rem} We explain why it is necessary that the 1-form $\la$ on $L$ used in the construction be closed and coclosed. Without assuming $\d\la=\d(*\la)=0$, suppose that using $\la$ we construct a 7-manifold $N$ and a family of $G_2$-structures $(\vp_t^N,g_t^N)$ for $t\in(0,\ep]$ as in \S\ref{nc62}, and suppose that as in the theorem there exist a family of torsion-free $G_2$-structures $(\ti\vp_t^N,\ti g_t^N)$ on $N$ with $\nm{\ti\vp_t^N-\vp_t^N}_{g_t^N,C^0}=O(t^{1/18})$ for $t\in(0,\ep]$. As in \S\ref{nc61} we have a 5-submanifold $Q\subset N$ with a submersion $\Pi:Q\ra L$, whose fibres $\Pi^{-1}(x)$ for $x\in L$ are round 2-spheres $\cS^2$ with area $\pi t^2\bmd{\la\vert_x}$ in the metric~$g_t^N$. 

Recall that if $f:X\ra Y$ is a submersion of compact, oriented manifolds with $\dim X=m$, $\dim Y=n$, and $\al$ is a $k$-form on $X$ for $k\ge m-n$, there is a unique pushforward $f_*(\al)$, a $(k-m+n)$-form on $Y$, with the property that if $\be$ is an $(m-k)$-form on $Y$ then $\int_Yf_*(\al)\w\be=\pm\int_X\al\w f^*(\be)$. Such pushforwards have the property that $\d(f_*(\al))=f_*(\d\al)$. Consider the equations
\e
\begin{split}
\Pi_*(\ti\vp_t^N\vert_Q)&=\Pi_*(\vp_t^N)+O(t^{37/18})=-\pi t^2\la+O(t^{37/18}),\\
\Pi_*(\Th(\ti\vp_t^N)\vert_Q)&=\Pi_*(\Th(\vp_t^N))+O(t^{37/18})=-\pi t^2(*\la)+O(t^{37/18}).
\end{split}
\label{nc6eq17}
\e
Here the first steps hold as $\nm{\ti\vp_t^N-\vp_t^N}_{g_t^N,C^0}=O(t^{1/18})$ and the fibres of $\Pi$ have volume $O(t^2)$, and the second steps hold by definition of $\vp_t^N$. As $\ti\vp_t^N,\Th(\ti\vp_t^N)$ are closed, the left hand sides of \eq{nc6eq17} are closed. Hence, multiplying \eq{nc6eq17} by $t^{-2}$ and taking the limit as $t\ra 0$, we see that $\la$ and $*\la$ are closed.
\label{nc6rem}	
\end{rem}

\subsection{\texorpdfstring{Replacing $M/\an{\io}$ by a $\Gt$-orbifold}{Replacing M/<ι> by a G₂-orbifold}}
\label{nc65}

The construction of \S\ref{nc3}--\S\ref{nc64} begins with a compact torsion-free $\Gt$-manifold $(M,\vp,g)$ with an involution $\io:M\ra M$ with $\io^*(\vp)=\vp$, where the fixed locus $L$ of $\io$ is a 3-submanifold in $M$. Then $M/\an{\io}$ is an orbifold, with orbifold stratum $L$, and $(\vp,g)$ descends to a torsion-free $\Gt$-structure on $M/\an{\io}$ in the orbifold sense, also written $(\vp,g)$. We define a resolution $\pi:N\ra M/\an{\io}$ and construct torsion-free $\Gt$-structures $(\ti\vp_t^N,\ti g_t^N)$ on~$N$.

In fact the whole programme also works if we start not with a quotient $M/\an{\io}$, but with a compact torsion-free $\Gt$-orbifold $(M',\vp,g)$, such that the only orbifold stratum of $M'$ is a compact 3-submanifold $L$ such that $M'$ near $L$ is locally modelled on $\R^3\t(\R^4/\{\pm 1\})$, and we are given a closed, coclosed, nonvanishing 1-form $\la$ on $(L,g\vert_L)$. It is not necessary that $M'$ should be a global quotient $M/\an{\io}$.

We chose to start from $M/\an{\io}$ in the main text for simplicity, to avoid dealing with orbifolds. But in fact the general case of $\Gt$-orbifolds $(M',\vp,g)$ is essentially identical, requiring no extra work, so we will not go through the details. In \S\ref{nc72}--\S\ref{nc74} we discuss examples involving $\Gt$-orbifolds $(M',\vp,g)$ in which $M'$ cannot be written as $M/\an{\io}$ for a 7-manifold~$M$.

\subsection{\texorpdfstring{Twisting $\la$ by a principal $\Z_2$-bundle}{Twisting λ by a principal ℤ₂-bundle}}
\label{nc66}

The construction of \S\ref{nc3}--\S\ref{nc64} involves a closed, coclosed, nonvanishing 1-form $\la$ on the fixed locus $L$ of $\io$ in $M$. This $\la$ was used in \S\ref{nc4} to define a resolution $P$ of $\nu/\{\pm 1\}$, where $\la$ is used to define a complex structure $I$ on the fibres of $\nu\ra L$, and then $\rho:P\ra\nu/\{\pm 1\}$ is the bundle blow-up of $\nu/\{\pm 1\}$ along the zero section $0(L)$ using~$I$.

Consider the effect of replacing $\la$ by $-\la$ in \S\ref{nc4}. This replaces $I$ by $-I$, but the blow-up $P$ is unchanged. It turns out that the data $\vp_t^P,\psi_t^P,g_t^P,\ti\vp_t^P,\ab\ti\psi_t^P,\ab\xi_{1,2},\ab\xi_{0,3},\ab\chi_{1,3},\ab\th_{3,1},\ab\th_{2,2},\ab\ldots$ defined in \S\ref{nc44} is unchanged as well. This holds because $\check\om^I,e_1,e_2\w e_3$ change sign, but combinations such as $\si^*(e_1)\w\check\om^I$ and $\si^*(e_2)\w\si^*(e_3)\w\check\om^I$ in \eq{nc4eq21} are unchanged. The constructions of \S\ref{nc5}--\S\ref{nc64}, and the final result $N,(\ti\vp_t^N,\ti g_t^N)$ in Theorem \ref{nc6thm}, only involve data which are unchanged by replacing $\la$ by~$-\la$. 

In fact the construction still works if $\la$ is only defined up to sign $\pm\la$ {\it locally in\/} $L$. Here is what we mean by this. Suppose $\pi:Z\ra L$ is a principal $\Z_2$-bundle over $L$. Then we can form the vector bundles $\La^kT^*L\ot_{\Z_2}Z\ra L$ of $k$-forms on $L$ {\it twisted by\/} $Z$. The $\d$ and $\d^*$ operators on $L$ extend naturally to $\d:\Ga^\iy(\La^kT^*L\ot_{\Z_2}Z)\ra \Ga^\iy(\La^{k+1}T^*L\ot_{\Z_2}Z)$ and $\d^*:\Ga^\iy(\La^kT^*L\ot_{\Z_2}Z)\ra \Ga^\iy(\La^{k-1}T^*L\ot_{\Z_2}Z)$, so it makes sense for a $Z$-twisted $k$-form $\la$ to be closed and coclosed. If locally in $L$ we choose a trivialization $Z\cong L\t\Z_2$, then we have local canonical isomorphisms $\La^kT^*L\ot_{\Z_2}Z\cong\La^kT^*L$, which identify $Z$-twisted $k$-forms with ordinary $k$-forms. These depend on the trivialization up to sign.

We can generalize the construction of \S\ref{nc3}--\S\ref{nc64} as follows: we choose a principal $\Z_2$-bundle $\pi:Z\ra L$, and take $\la$ to be a closed, coclosed, nonvanishing, $Z$-{\it twisted\/} 1-form on $L$, instead of an ordinary 1-form. Locally in $L$ we can identify $\la$ with an ordinary 1-form on $L$ up to sign, and this is enough to define $P,\vp_t^P,\psi_t^P,g_t^P,\ldots,N,\ti\vp_t^N,\ti g_t^N$ in \S\ref{nc4}--\S\ref{nc64}, since as above these are all unchanged by replacing $\la$ by $-\la$, and the analogue of Theorem \ref{nc6thm} holds.

Proposition \ref{nc6prop1} and equation \eq{nc6eq16} must be modified. In the $Z$-twisted case we no longer have $Q\cong L\t\cS^2$, so \eq{nc6eq5} does not hold. Instead it turns out that
\begin{equation*}
H^k(Q;\R)\cong H^k(L;\R)\op H^{k-2}(L,Z;\R),
\end{equation*}
where $H^*(L,Z;\R)$ is the $Z$-{\it twisted de Rham cohomology\/} of $L$, the cohomology of the complex $\bigl(\Ga^\iy(\La^kT^*L\ot_{\Z_2}Z)_{k\ge 0},\d\bigr)$. Thus in the third row of \eq{nc6eq4} we replace $H^*(L;\R)$ by $H^*(L,Z;\R)$, and \eq{nc6eq2} and \eq{nc6eq16} are replaced by
\e 
b^k(N)=b^k(M/\an{\io})+b^{k-2}(L,Z),
\label{nc6eq18}
\e
where $b^k(L;Z)=\dim H^k(L,Z;\R)$ are the $Z$-twisted Betti numbers of $L$. In Example \ref{nc7ex3} we construct compact $\Gt$-manifolds using $Z$-twisted 1-forms.

\section{Examples} 
\label{nc7}

\subsection{General discussion on how to produce examples}
\label{nc71}

To apply Theorem \ref{nc6thm}, we have to find:
\begin{itemize}
\setlength{\itemsep}{0pt}
\setlength{\parsep}{0pt}
\item[(a)] A compact torsion-free $\Gt$-manifold $(M,\vp,g)$ and an involution $\io:M\ra M$ with $\io^*(\vp)=\vp$ and fixed points a 3-submanifold $L\subset M$.

Or more generally, as in \S\ref{nc65}, a compact torsion-free $\Gt$-orbifold $(M',\vp,g)$
with one orbifold stratum $L$ locally modelled on $\R^3\t(\R^4/\{\pm 1\})$. 
\item[(b)] A closed, coclosed, nonvanishing 1-form $\la$ on~$L$.

Or more generally, as in \S\ref{nc66}, a principal $\Z_2$-bundle $\pi:Z\ra L$ and a closed, coclosed, nonvanishing, $Z$-twisted 1-form on~$L$.
\end{itemize}

There are several obvious ways to produce the data in~(a):
\begin{itemize}
\setlength{\itemsep}{0pt}
\setlength{\parsep}{0pt}
\item[(i)] We could make $(M,\vp,g),\io$ or $(M',\vp,g)$ using one of the known constructions of compact 7-manifolds with holonomy $\Gt$ due to the first author, Kovalev--Lee, or Corti--Haskins--Nordstr\"om--Pacini \cite{CHNP2,Joyc1,Joyc2,Joyc3,Joyc4,Kova,KoLe}, and either leave some orbifold singularities unresolved, or do the construction $\Z_2$-equivariantly, and divide by $\Z_2$ at the end of the construction.
\item[(ii)] We could take $M'$ of the form $(T^3\t K3)/\Ga$ for $\Ga$ a finite group (e.g. $\Ga=\Z_2^2$), where $T^3$ carries a flat metric and $K3$ a hyperK\"ahler metric. 
\item[(iii)] We could take $M=\cS^1\t Y$ for $Y$ a Calabi--Yau 3-fold, and $\io:\cS^1\t Y\ra\cS^1\t Y$ to act by $\io(x,y)=(-x,\tau(y))$ for $\tau:Y\ra Y$ an antiholomorphic involution, with fixed locus a special Lagrangian 3-fold $L'$ in $Y$. Then $L=\{0,\ha\}\t L'$. This situation was described in Examples \ref{nc2ex1} and \ref{nc2ex2}.
\end{itemize}

However, there is a difficulty in finding a suitable 1-form $\la$ in (b). As $(L,g\vert_L)$ is a compact Riemannian 3-manifold, we know by Hodge theory that the vector space $\cH^1$ of closed, coclosed 1-forms $\la$ on $L$ is isomorphic to the de Rham cohomology group $H^1(L;\R)$, which we can often compute. But unless we know something about the metric $g\vert_L$ on $L$, we do not know that any $\la\in\cH^1$ is {\it nonvanishing\/}; conceivably every $\la\in\cH^1$ might have at least one zero.

Thus, to apply Theorem \ref{nc6thm} we restrict to cases in which we can describe $(L,g\vert_L)$ as a Riemannian 3-manifold. Note in particular that in case (iii), if the Calabi--Yau metric $h$ on $Y$ comes from Yau's proof of the Calabi Conjecture, then without further information on $h$, we know nothing about $h\vert_{L'}$, so we cannot deduce the existence of closed, coclosed, nonvanishing 1-forms $\la$ on~$L'$.

Here are some useful facts for constructing such 1-forms~$\la$:
\begin{itemize}
\setlength{\itemsep}{0pt}
\setlength{\parsep}{0pt}
\item[(A)] Suppose $(L,g)$ is a 3-torus $T^3$ with a flat metric. Then there is a space $\R^3\sm\{0\}$ of closed, coclosed, nonvanishing 1-forms $\la$ on $L$.
\item[(B)] Suppose $(L,g)$ is a product $(\cS^1\t\Si,\d x^2+h)$, where $(\Si,h)$ is a compact Riemannian 2-manifold (of any genus $g\ge 0$), and $x\in\R/\Z$ is the coordinate on $\cS^1$. Then $\la=\d x$ is closed, coclosed, and nonvanishing on~$L$.
\item[(C)] Suppose $(L,g)$ is a product $(\cS^1\t\Si,\d x^2+h)$, where $\Si$ is an oriented torus $T^2$, and $h$ is a Riemannian metric on $\Si$, which need not be flat.

The notion of harmonic 1-form on $(\Si,h)$ depends only on the conformal class of $h$, and $h$ is conformal to a flat metric. Thus there is an $\R^2\sm\{0\}$ space of closed, coclosed, nonvanishing 1-forms on $(\Si,h)$, which pull back to closed, coclosed, nonvanishing 1-forms $\la$ on~$L$. 
\item[(D)] In each of (A)--(C) we can replace $L$ by $L/\Ga$, where $\Ga$ is a finite group acting freely on $L$ by isometries preserving~$\la$. 

More generally, we can allow $\Ga$ to preserve $\la$ only up to sign, and then $\la$ descends to a closed, coclosed, nonvanishing, $Z$-twisted 1-form on~$L/\Ga$.
\item[(E)] Let $(L,g)$ be a compact Riemannian 3-manifold, and $\la$ a closed, coclosed, nonvanishing 1-form on $(L,g)$. Suppose $g'$ is another Riemannian metric on $L$. Then there exists a unique closed, coclosed 1-form $\la'$ on $(L,g')$ with $[\la']=[\la]\in H^1_{\rm dR}(L;\R)$. One can show that if $\nm{g'-g}_{C^0}$ is small then $\nm{\la'-\la}_{C^0}\le C\nm{g'-g}_{C^0}$ for $C>0$ depending only on $L,g,\la$. But if $\nm{\la'-\la}_{C^0}$ is sufficiently small then $\la'$ is nonvanishing, as $\la$ is. This proves that if $(L,g)$ admits a closed, coclosed, nonvanishing 1-form $\la$, and $g'$ is any other metric on $L$ with $\nm{g'-g}_{C^0}$ sufficiently small, then $(L,g')$ also admits a closed, coclosed, nonvanishing 1-form~$\la'$.

We can apply this in (A)--(D), so that the metric $g$ need only be $C^0$-close to being flat, or a product~$\d x^2+h$.
\end{itemize}

\subsection{\texorpdfstring{Resolving orbifolds $T^7/\Ga$ in several stages}{Resolving orbifolds T⁷/Γ in several stages}}
\label{nc72}

In \cite{Joyc1,Joyc2}, \cite[\S 11--\S 12]{Joyc3} the first author constructed examples of compact 7-manifolds with holonomy $\Gt$ by resolving orbifolds $T^7/\Ga$, where $\Ga$ is a finite group acting on $T^7$ preserving a flat $\Gt$-structure $(\vp_0,g_0)$. We explain how to modify this construction. Suppose $(T^7/\Ga,\vp_0,g_0)$ is such a flat $\Gt$-orbifold, and the orbifold singularities of $T^7/\Ga$ split as $S\amalg T$, where $S$ is a union of orbifold strata that can be resolved using the methods of \cite{Joyc1,Joyc2,Joyc3}, and $T$ is a disjoint union of orbifold strata $T^3/\De$, where $\De$ is a finite group acting freely on $T^3\subset T^7$, and the singularities are locally modelled on~$(T^3\t(\R^4/\{\pm 1\}))/\De$.

Apply the methods of \cite{Joyc1,Joyc2,Joyc3} (which work for partial resolutions in orbifolds, as well as for complete resolutions in manifolds) to resolve the singularities $S$ in $T^7/\Ga$, but leave $T$ unresolved. This gives a 7-orbifold $M'$ with resolution map $\pi:M'\ra T^7/\Ga$, and a family of torsion-free $\Gt$-structures $(\vp_s,g_s)$ on $M'$ for $s\in(0,\ep]$ with $\ep>0$ small. Here $s$ is the length-scale of the resolution of the singular set $S$, so that near $\pi^{-1}(S)$ the curvature $R(g_s)$ is $O(s^{-2})$ and the injectivity radius $\de(g_s)$ is $O(s)$, and away from $\pi^{-1}(S)$ we have $\nm{g_s-\pi^*(g_0)}_{C^0}=O(s^\ka)$ for some $\ka>0$ as in Theorem \ref{nc2thm3} (we can take $\ka=\ha$ as in \cite[Th.~11.6.1]{Joyc3}).

Now $M'$ has singular set $T$. As $T$ is disjoint from $S$, we have $\nm{g_s\vert_T-\pi^*(g_0\vert_T)}_{C^0}=O(s^{1/2})$, so $g_s\vert_T$ is $C^0$-close to the flat metric $g_0\vert_T$ on $T$ for small $s>0$. Suppose each component $T^3/\De$ of $T$ has $b^1(T^3/\De)>0$ (or more generally, $T^3/\De$ has a principal $\Z_2$-bundle $Z$ with $H^1(T^3/\De,Z;\R)\ne 0$, as in \S\ref{nc66}). Then as in \S\ref{nc71}(A),(D) there exists a closed, coclosed, nonvanishing 1-form $\la_0$ on $(T,g_0\vert_T)$ (possibly $Z$-twisted as in \S\ref{nc66}), so by \S\ref{nc71}(E) there also exists a closed, coclosed, nonvanishing 1-form $\la_s$ on $(T,g_s\vert_T)$ for small enough $s>0$ (possibly $Z$-twisted as in \S\ref{nc66}). Thus we can apply Theorem \ref{nc6thm}, modified as in \S\ref{nc65}--\S\ref{nc66}, to construct a compact 7-manifold $N$ resolving $M'$, and a family $(\vp_{s,t}^N,g_{s,t}^N)$ of torsion-free $\Gt$-structures on $N$, where $t>0$ must be small in terms of $s$, that is,~$0<t\ll s\ll 1$.

Arguably these examples are not that interesting, as we could just have applied  \cite{Joyc1,Joyc2,Joyc3} to get torsion-free $\Gt$-structures $(\vp_s^N,g_s^N)$ on the same 7-man\-i\-fold $N$ in a single step, by resolving the singularities $S,T$ simultaneously, rather than first resolving $S$, then resolving $T$. However, the $\Gt$-structures $(\vp_{s,t}^N,g_{s,t}^N)$ we have constructed are different from those of \cite{Joyc1,Joyc2,Joyc3}, as they involve two length-scales $s,t$ with~$0<t\ll s\ll 1$.

In fact the authors expect that one can resolve the disjoint singular sets $S,T$ in $T^7/\Ga$ completely independently, so that there should exist torsion-free $\Gt$-structures $(\vp_{s,t}^N,g_{s,t}^N)$ on $N$ which resolve $S$ at length-scale $s>0$ and $T$ at length scale $t>0$, whenever $0<s,t\ll 1$. But this is not proved. The methods of \cite{Joyc1,Joyc2,Joyc3} give such $(\vp_{s,t}^N,g_{s,t}^N)$ when $s,t$ are about the same size, that is, $0<s,t\ll 1$ with $s\approx t$, and the modification above gives $(\vp_{s,t}^N,g_{s,t}^N)$ when $t$ is much smaller than $s$, that is, $0<t\ll s\ll 1$, but the two constructions do not overlap.

We can also generalize all this to resolve orbifolds $T^7/\Ga$ with singular set $S\amalg T_1\amalg \cdots\amalg T_k$, for $S$ as above and $T_1,\ldots,T_k$ as for $T$ above, where we iteratively resolve $S,T_1,T_2,\ldots,T_k$ in this order, at length-scales $s,\ab t_1,\ab t_2,\ab\ldots,\ab t_k$ respectively, where~$0<t_k\ll t_{k-1}\ll \cdots\ll t_1\ll s\ll 1$.

Here is an example, taken from \cite[\S 2.1]{Joyc1}. The labelling has been changed to be compatible with \eq{nc2eq1}.

\begin{ex} Let $T^7=\R^7/\Z^7$ with standard $\Gt$-structure $(\vp_0,g_0)$ given in \eq{nc2eq1}--\eq{nc2eq2}. Let $\Ga=\an{\al,\be,\ga}\cong\Z_2^3$, where $\al,\be,\ga$ are involutions acting by
\begin{align*}
\al(x_1,\ldots,x_7)&=(-x_1,-x_2,x_3,-x_4,x_5,x_6,-x_7),\\
\be(x_1,\ldots,x_7)&=(\ha-x_1,x_2,-x_3,-x_4,x_5,-x_6,x_7),\\
\ga(x_1,\ldots,x_7)&=(x_1,\ha-x_2,-x_3,\ha-x_4,-x_5,x_6,x_7).
\end{align*}
These preserve the $\Gt$-structure $(\vp_0,g_0)$. The only elements of $\Ga$ with fixed points are $1,\al,\be,\ga$, where $\al,\be,\ga$ each fix 16 $T^3$'s. The quotient $T^7/\Ga$ is simply-connected, with singular set:
\begin{itemize}
\setlength{\itemsep}{0pt}
\setlength{\parsep}{0pt}
\item[(a)] 4 copies of $T^3$ from the fixed points of $\al$.
\item[(b)] 4 copies of $T^3$ from the fixed points of $\be$.
\item[(c)] 4 copies of $T^3$ from the fixed points of $\ga$.
\end{itemize}
These 12 $T^3$'s are pairwise disjoint. Near each $T^3$, the space $T^7/\Ga$ is modelled on $T^3\t(\R^4/\{\pm 1\})$.

Split the singular set as $S\amalg T$, where $S$ is $k$ of the 12 $T^3$'s, and $T$ is the remaining $12-k$ of the $T^3$'s, for $0<k<12$. Then the construction above applies. The 7-manifold $N$ resolving $T^7/\Ga$ is described in \cite[\S 2.2]{Joyc1}, and is simply-connected with $b^2(N)=12$ and $b^3(N)=43$. The construction above yields a family of torsion-free $\Gt$-structures $(\vp_{s,t}^N,g_{s,t}^N)$ on $N$, which have holonomy~$\Gt$. 
\label{nc7ex1}	
\end{ex}

We can do the same thing for many of the examples in \cite{Joyc1,Joyc2,Joyc3}.

It is also possible to do a similar two-stage (or multiple-stage) resolution process for examples of $\Gt$-manifolds coming from the `twisted connected sum' constructions of \cite{CHNP1,CHNP2,Kova,KoLe}. These work by gluing together noncompact $\Gt$-manifolds $\cS^1\t Y_1$ and $\cS^1\t Y_2$, where $Y_1,Y_2$ are Asymptotically Cylindrical Calabi--Yau 3-folds satisfying a matching condition at their infinite ends. We could instead allow $Y_1,Y_2$ to be Calabi--Yau orbifolds, with compact singular sets $\Si_1,\Si_2$ locally modelled on~$\C\t(\C^2/\{\pm 1\})$. 

Gluing $\cS^1\t Y_1$ and $\cS^1\t Y_2$ together using \cite{CHNP1,CHNP2,Kova,KoLe} would give a family of compact $\Gt$-orbifolds $(M',\vp_s,g_s)$, with singular set $(\cS^1\t\Si_1)\amalg(\cS^1\t\Si_2)$ locally modelled on $\R^3\t(\R^4/\{\pm 1\})$. We could then resolve this using Theorem \ref{nc6thm} generalized as in \S\ref{nc65} and \S\ref{nc71}(B),(E), to get torsion-free $\Gt$-structures $(\vp_{s,t}^N,g_{s,t}^N)$ on a 7-manifold $N$. This is the same $N$ we would get by applying \cite{CHNP1,CHNP2,Kova,KoLe} to glue $\cS^1\t\ti Y_1$ and $\cS^1\t\ti Y_2$, where $\ti Y_1,\ti Y_2$ are the crepant resolutions of $Y_1,Y_2$, so we get new torsion-free $\Gt$-structures on known examples of 7-manifolds with holonomy~$\Gt$. Carrying out this procedure for some explicit examples could be a project for future study.

\subsection{\texorpdfstring{Resolving orbifolds $(T^3\t K3)/\Ga$}{Resolving orbifolds (T³×K3)/Γ}}
\label{nc73}

Next we apply our construction to resolve $\Gt$-orbifolds of the form $(T^3\t X)/\Ga$, where $X$ is a hyperK\"ahler $K3$ surface. 

Take $T^3=\R^3/\La$, where $\La$ is a lattice in $\R^3$ (e.g.\ $\La=\Z^3$). Let $X$ be a hyperK\"ahler $K3$ surface with metric $h$, complex structures $I,J,K$, and hyperK\"ahler forms $\om^I,\om^J,\om^K$. Define a torsion-free $\Gt$-structure $(\vp,g)$ on $T^3\t X$ as in \eq{nc2eq13}. Suppose we can find a finite group $\Ga$ with an action on $T^3\t X$ preserving $(\vp,g)$, so $M'=(T^3\t X)/\Ga$ is a 7-orbifold with $\Gt$-structure $(\vp,g)$, such that the singular set of $M'$ can be resolved using our construction.

As $\Ga$ acts by isometries, one can show that in this case the action of $\Ga$ on $T^3\t X$ must be a product of $\Ga$-actions on $T^3$ and on $X$. To get orbifold singularities of the right type, it is necessary that any $1\ne\ga\in\Ga$ with fixed points in $T^3\t X$ must be of order 2. There are two possibilities for such~$\ga$:
\begin{itemize}
\setlength{\itemsep}{0pt}
\setlength{\parsep}{0pt}
\item[(a)] $\ga$ acts trivially on $T^3$, and fixes $2k$ points on $X$, where $\ga$ acts on $H^2(X;\R)$ with $10+k$ eigenvalues $+1$ and $12-k$ eigenvalues $-1$. Hence $\Fix(\ga)$ is $2k$ copies of $T^3$ in $T^3\t X$.
\item[(b)] $\ga$ fixes $l$ circles $\cS^1$ in $T^3$, where $l=2$ or 4, and $\ga$ fixes a 2-submanifold $\Si$ in $X$. Hence $\Fix(\ga)$ is $l$ copies of $\cS^1\t\Si$ in~$T^3\t X$.
\end{itemize}
The corresponding orbifold stratum in $M'=(T^3\t X)/\Ga$ is $\Fix(\ga)/\De$, where~$\De=\{\de\in\Ga:\de\ga=\ga\de\}/\{1,\ga\}$.

In case (b), after applying an $\SO(3)$ transformation to $\an{I,J,K}_\R$ we can suppose that $\ga$ fixes $I$ and changes the signs of $J,K$. Hence $\ga$ is a {\it nonsymplectic involution\/} of the complex symplectic $K3$ surface $(X,I,\om^J+i\om^K)$. Nonsymplectic involutions have been completely classified by Nikulin \cite{Niku}, see also \cite{AlNi} and \cite[\S 3]{KoLe}. Nikulin shows \cite[\S 4]{Niku} that $\Si$ must be either the disjoint union of two $T^2$'s, or the disjoint union of a genus $g$ surface and $k$ $\CP^1$'s, where there are 64 possibilities for $(g,k)$ with $0\le g\le 10$ and~$0\le k\le 9$.

A huge amount is known about $K3$ surfaces, and using this we could construct many examples of hyperK\"ahler $K3$ surfaces $X$ with suitable actions of a finite group $\Ga$. Basically, everything is determined by the choice of faithful $\Ga$-representation on the $K3$ lattice $H^2(X;\Z)\cong 2(-E_8)\op 3H$, which fixes a positive definite subspace $\an{\om^I,\om^J,\om^K}_\R\subset H^2(X;\R)$ satisfying some conditions.

One possibility would be to take $X$ to be a resolution of $T^4/\De$, and the $\Ga$-action on $X$ to lift from a $\Ga$-action on $T^4/\De$. Then $M'=(T^3\t X)/\Ga$ would be a partial resolution of $(T^3\t T^4)/(\Ga\ltimes\De)$, and the $\Gt$-manifolds $N$ we eventually construct would be a complete resolution of $T^7/(\Ga\ltimes\De)$. In this case, the 7-manifold $N$ we get would be the same as if we resolved $T^7/(\Ga\ltimes\De)$ in a single step using \cite{Joyc1,Joyc2,Joyc3}, as in \S\ref{nc72}. However, not all examples are of this type.  

\begin{ex} Let $C$ be the nonsingular sextic curve in $\CP^2$ 
\begin{equation*}
C=\bigl\{[z_0,z_1,z_2]\in\CP^2:z_0^6+z_1^6+z_2^6=0\bigr\}.
\end{equation*}
Then $C$ has genus 10, by the degree-genus formula. Define $\pi:X\ra\CP^2$ to be the double cover of $\CP^2$ branched over $C$. Then $X$ is a complex $K3$ surface, with complex structure $I$. Write $\al:X\ra X$ to be the holomorphic involution swapping the sheets of the double cover, so that $\pi\ci\al=\pi$. Then $\al$ is a nonsymplectic involution, with $\Fix(\al)=C\subset X$. 

Consider the antiholomorphic involution of $\CP^2$ mapping $[z_0,z_1,z_2]\mapsto[\bar z_0,\ab\bar z_1,\ab\bar z_2]$. This lifts through $\pi$ to two antiholomorphic involutions of $X$, which we write as $\be$ and $\al\be$, where we take $\be$ to have fixed set $\Fix(\be)=\{x\in X:\pi(x)\in\RP^2\subset\CP^2\}$, and $\al\be$ to be free. Then $\Fix(\be)$ is the double cover of $\RP^2$, that is, $\Fix(\be)\cong\cS^2$, and $\Fix(\al\be)=\es$, and $\Ga=\{1,\al,\be,\al\be\}\cong\Z_2^2$ acts on~$X$.

By the Calabi Conjecture we can extend $(X,I)$ to a hyperK\"ahler $K3$ surface $(X,I,J,K,h,\om^I,\om^J,\om^K)$, where we can choose $J,K$ such that $\al,\be$ act by
\begin{align*}
\al&:\om^I\longmapsto \om^I, & \al&:\om^J\longmapsto -\om^J, & \al&:\om^K\longmapsto -\om^K, \\
\be&:\om^I\longmapsto -\om^I, & \be&:\om^J\longmapsto \om^J, & \be&:\om^K\longmapsto -\om^K.
\end{align*}
Let $T^3=\R^3/\Z^3$ with coordinates $(x_1,x_2,x_3)$, $x_i\in\R/\Z$, and let $\Ga$ act on $T^3$ by
\e
\al:(x_1,x_2,x_3)\longmapsto (x_1,-x_2,-x_3),\;\> \be:(x_1,x_2,x_3)\longmapsto (-x_1,x_2,\ha -x_3).
\label{nc7eq1}
\e
The fixed locus of $\al$ in $T^3$ is four $\cS^1$'s, given by $x_2,x_3\in\{0,\ha\}$. Also the fixed locus of $\be$ is four $\cS^1$'s, and $\al\be$ acts freely on~$T^3$.

Define a torsion-free $\Gt$-structure $(\vp,g)$ on $T^3\t X$ as in \eq{nc2eq13}. Then $\Ga$ preserves $(\vp,g)$, so $M'=(T^3\t X)/\Ga$ is a compact orbifold with torsion-free $\Gt$-structure $(\vp,g)$. The fixed set of $\al$ in $T^3\t X$ is four copies of $\cS^1\t C$, but the four are exchanged in pairs by $\be$, so these contribute two copies of $\cS^1\t C$ to the singular set of $M'$. The fixed set of $\be$ in $T^3\t X$ is four copies of $\cS^1\t\cS^2$, which are exchanged in pairs by $\al$, and contributes two copies of $\cS^1\t\cS^2$ to the singular set of $M'$. As $\al\be$ acts freely, it contributes no singular points, and the singular sets from $\al,\be$ are disjoint. Hence the singular set of $M'$ is two copies of $\cS^1\t C$, where $C$ has genus 10, and two copies of~$\cS^1\t\cS^2$.

The de Rham cohomology $H^k_{\rm dR}(M';\R)$ is the subspace of $H^k_{\rm dR}(T^3\t X;\R)$ fixed by $\Ga$. It is a general fact about $K3$ surfaces $X$ that if $\ga:X\ra X$ is an involution and the induced action of $\ga$ on $H^2(X)\cong\R^{22}$ has $a$ eigenvalues $+1$ and $22-a$ eigenvalues $-1$ then $\chi(\Fix(\al))=2a-20$. Since $\chi(\Fix(\al))=-18$, $\chi(\Fix(\be))=2$ and $\chi(\Fix(\al\be))=0$ from above, we see that $H^2(X)\cong\R\op\R^{11}\op\R^{10}$ where $\al=1$ and $\be=-1$ on $\R$, and $\al=-1$ and $\be=1$ on $\R^{11}$, and $\al=\be=-1$ on $\R^{10}$. Using this and \eq{nc7eq1} we can compute the action of $\Ga$ on
\begin{equation*}
H^k_{\rm dR}(T^3\t X;\R)\cong\bigop\nolimits_{i+j=k}H^i_{\rm dR}(T^3;\R)\ot H^j_{\rm dR}(X;\R).
\end{equation*}
By taking $\Ga$-invariant parts, we find that $b^0(M')=1$, $b^1(M')=b^2(M')=0$, and $b^3(M')=23$. Moreover, $M'$ is simply-connected.

The singular set $L$ is a product $\cS^1\t (C\amalg C\amalg \cS^2 \amalg\cS^2)$, and the  metric $g\vert_L$ on $L$ is a product metric. Hence by \S\ref{nc71}(B) we can apply Theorem \ref{nc6thm} to get a compact 7-manifold $N$ carrying torsion-free $\Gt$-structures  $(\vp_t^N,g_t^N)$. Proposition \ref{nc6prop1} shows that $b^0(N)=1$, $b^1(N)=0$, $b^2(N)=4$ and $b^3(N)=67$, using that $b^0(L)=4$ and $b^1(L)=44$ (which follows from $b^1(C)=20$), and that $N$ is simply-connected, since $M'$ is. Hence $\Hol(g_t^N)=G_2$ by Theorem \ref{nc2thm2}. 

Both the first author \cite[\S 12]{Joyc3} and Kovalev and Lee \cite[Table 1]{KoLe} construct examples of compact 7-manifolds with holonomy $\Gt$ and $b^2=4$, $b^3=67$. In fact our 7-manifold $N$ can be obtained from the Kovalev--Lee construction \cite{KoLe} by instead using the metric $\d x_1^2+\d x_2^2+R^2\d x_3^2$ on $T^3$ for $R\gg 0$, so that the $x_3$ circle is very long. As $R\ra\iy$, $N\cong N_1\#_{\rm twist}N_2$ splits as a `twisted connect sum' into a region $N_1$ from $x_3\in (-\frac{1}{4},\frac{1}{4})$ containing the resolution of the $\al$ fixed points, and a region $N_2$ from $x_3\in (\frac{1}{4},\frac{3}{4})$ containing the resolution of the $\be$ fixed points, where $N_i\cong\cS^1\t Y_i$ for $Y_i$ an Asymptotically Cylindrical Calabi--Yau 3-fold obtained as a resolution of ~$(X\t\C^*)/\Z_2$.
\label{nc7ex2}	
\end{ex}

\begin{ex} Work in the situation of Example \ref{nc7ex2}, but instead of \eq{nc7eq1}, let $\Ga$ act on $T^3$ by 
\begin{equation*}
\al:(x_1,x_2,x_3)\longmapsto (x_1,-x_2,-x_3),\;\> \be:(x_1,x_2,x_3)\longmapsto (-x_1,x_2,-x_3).
\end{equation*}
Again the fixed set of $\al$ in $T^3\t X$ is four copies of $\cS^1\t C$, but now $\an{\be}=\Z_2$ acts freely on each copy, so they contribute four copies of $(\cS^1\t C)/\Z_2$ to the singular set of $M'$. Similarly, the fixed set of $\be$ in $T^3\t X$ is four copies of $\cS^1\t\cS^2$, but $\an{\al}=\Z_2$ acts freely on each copy, so they contribute four copies of $(\cS^1\t\cS^2)/\Z_2$ to the singular set of $M'$. Thus, the singular set $L$ of $M'$ is the disjoint union of four copies of $(\cS^1\t C)/\Z_2$ and  four copies of~$(\cS^1\t\cS^2)/\Z_2$.

We resolve $M'$ using Theorem \ref{nc6thm} modified as in \S\ref{nc66}, twisting by the unique nontrivial principal $\Z_2$-bundles $Z$ on $(\cS^1\t C)/\Z_2$ and $(\cS^1\t\cS^2)/\Z_2$ induced by the $\Z_2$-actions on $\cS^1\t C$ and $\cS^1\t\cS^2$. Then we can take the $Z$-twisted 1-form to be $\la=\d x$ for $x\in\R/\Z$ the coordinate on $\cS^1$ on each component, and we can apply our construction as in \S\ref{nc71}(B),(D), to get a compact 7-manifold $N$ with torsion-free $\Gt$-structures $(\vp_t^N,g_t^N)$. Calculation shows that the $Z$-twisted Betti numbers are
\begin{align*}
b^0\bigl((\cS^1\t C)/\Z_2,Z\bigr)&=0,&
b^1\bigl((\cS^1\t C)/\Z_2,Z\bigr)&=11,\\
b^0\bigl((\cS^1\t\cS^2)/\Z_2,Z\bigr)&=0,&
b^1\bigl((\cS^1\t\cS^2)/\Z_2,Z\bigr)&=1.
\end{align*}
As before we have $b^0(M')=1$, $b^1(M')=b^2(M')=0$, and $b^3(M')=23$, and $M'$ is simply-connected, so $N$ is simply-connected, and $\Hol(g_t^N)=G_2$. Using \eq{nc6eq18} we see that the Betti numbers of $N$ are $b^2(N)=0$ and~$b^3(N)=71$.

Kovalev and Lee \cite[Table 1]{KoLe} also construct a compact 7-manifold with holonomy $\Gt$ and $b^2=0$, $b^3=71$. But in contrast to Example \ref{nc7ex2}, here we cannot relate our 7-manifold $N$ to a Kovalev--Lee example by stretching an $\cS^1$ in $T^3$ to separate the $\al$ and $\be$ fixed points, since these both lie over $(0,0,0)$ in $T^3$. So this example may be new.
\label{nc7ex3}	
\end{ex}

\subsection{Describing metrics on some Calabi--Yau 3-folds}
\label{nc74}

Suppose $Y$ is a 3-dimensional projective complex orbifold with trivial canonical bundle, all of whose singularities are locally modelled on $\C\t(\C^2/\{\pm 1\})$. Write $\Si\subset Y$ for the singular set, so that $\Si$ is a smooth projective curve. Using the Calabi Conjecture in the orbifold category gives a Calabi--Yau metric $h$ on $Y$. As in \S\ref{nc26}, $M'=\cS^1\t Y$ is a 7-orbifold with torsion-free $\Gt$-structure $(\vp,g)$, with singular set $L=\cS^1\t\Si$ locally modelled on~$\R^3\t(\R^4/\{\pm 1\})$.

Since $g=\d x^2+h$ restricts to a product metric on $L=\cS^1\t\Si$, as in \S\ref{nc71}(A) Theorem \ref{nc6thm} applies to construct torsion-free $\Gt$-structures $(\vp_t^N,g_t^N)$ on a resolution $N$ of $\cS^1\t Y$. In this case everything is equivariant under the obvious action of $\cS^1$ on $\cS^1\t Y$, and we have $N=\cS^1\t\ti Y$, for $\ti Y$ the crepant resolution of $Y$, which is the blow up of $Y$ along $\Si$, and the $\Gt$-structures $(\vp_t^N,g_t^N)$ come from Calabi--Yau metrics $\ti h_t$ on $\ti Y$ for small~$t>0$.

Of course, we already knew by the Calabi Conjecture that $\ti Y$ admits Calabi--Yau metrics. However, our construction gives a precise description of the Calabi--Yau metrics $\ti h_t$ on $\ti Y$ in the limit as the $\CP^1$ fibres of $\pi:\ti N\ra N$ over the singular set $\Si$ become small, and we believe this description is new.

\begin{ex} Define $Y$ to be the orbifold-smooth hypersurface
\begin{equation*}
Y=\bigl\{[z_0,z_1,z_2,z_3,z_4]\in\CP^4_{1,1,2,2,2}:z_0^8+z_1^8+z_2^4+z_3^4+z_4^4=0\bigr\}
\end{equation*}
in the weighted projective space $\CP^4_{1,1,2,2,2}$. Then $Y$ is a Calabi--Yau 3-orbifold. The singularities of $Y$ are inherited from the orbifold singularities of $\CP^4_{1,1,2,2,2}$, are locally modelled on $\C\t(\C^2/\{\pm 1\})$, and are the genus 6 projective curve
\begin{equation*}
\Si=\bigl\{[0,0,z_2,z_3,z_4]\in\CP^4_{1,1,2,2,2}:z_2^4+z_3^4+z_4^4=0\bigr\}.
\end{equation*}
The blow-up $\ti Y$ of $Y$ along $\Si$ is a Calabi--Yau 3-fold. As above we can describe Calabi--Yau metrics $\ti h_t$ on $\ti Y$ close to orbifold Calabi--Yau metrics $h$ on~$Y$.
\label{nc7ex4}	
\end{ex}

\subsection{\texorpdfstring{Conjectural examples from $($Calabi--Yau 3-fold$\t\cS^1)/\an{\io}$}{Conjectural examples from (Calabi--Yau 3-fold × S¹)/<ι>}}
\label{nc75}

Suppose $Y$ is a Calabi--Yau 3-fold with metric $h$, and $\tau:Y\ra Y$ is an antiholomorphic isometric involution with fixed locus $L'$, a special Lagrangian 3-fold in $Y$. Then as in \S\ref{nc71}(iii) we can take $M=\cS^1\t Y$ with the induced torsion-free $\Gt$-structure with product metric $g=\d x^2+h$, and we can define an involution $\io:M\ra M$ by $\io(x,y)=(-x,\tau(y))$, so that $\io$ has fixed locus $L=\{0,\ha\}\t L'$. To apply Theorem \ref{nc6thm}, we need there to exist a closed, coclosed, nonvanishing 1-form $\la'$ on $(L',h\vert_{L'})$. As in \S\ref{nc71}, if the Calabi--Yau metric $h$ comes from the Calabi Conjecture then we know basically nothing about $h\vert_{L'}$, so we do not know whether any of the harmonic 1-forms $\la$ on $L'$ are nonvanishing.

However, if the Calabi--Yau manifold $(Y,h)$ is close to some kind of singular limit, then we may have at least a conjectural description of $h$, which may be enough to say something useful about $h\vert_{L'}$. In particular, in some special cases we may expect $L'$ to be a disjoint union of $T^3$'s with $h\vert_{L'}$ $C^0$-close to a flat metric, and then there exist closed, coclosed, nonvanishing 1-forms $\la'$ on $L'$ as in \S\ref{nc71}(A),(E), so we could apply Theorem \ref{nc6thm} to resolve $M/\an{\io}$, giving a compact 7-manifold $N$ with torsion-free $\Gt$-structures~$(\vp_t^N,g_t^N)$.

The famous `SYZ Conjecture' aims to explain Mirror Symmetry of Calabi--Yau 3-folds $Y,\check Y$. The original formulation by Strominger, Yau and Zaslow \cite{SYZ} concerned dual special Lagrangian fibrations $\pi:Y\ra B$, $\check\pi:\check Y\ra B$, with generic fibre $T^3$. Later Kontsevich and Soibelman \cite[\S 3]{KoSo}, Morrison \cite{Morr}, and others extended the conjecture to describe how the Calabi--Yau metrics $h,\check h$ on $Y,\check Y$ degenerate as one approaches the `large complex structure limit'. Here is a partial statement of this conjectural picture, focussing on one side of the mirror:

\begin{conj}[Strominger--Yau--Zaslow \cite{SYZ}, Kontsevich--Soibelman \cite{KoSo}, and Morrison \cite{Morr}] Suppose $(Y,J,\Om,h)$ is a Calabi--Yau $3$-fold which is close to the `large complex structure limit'. Then there should exist a compact real\/ $3$-manifold\/ $B$ and a continuous map $\pi:Y\ra B,$ called the `SYZ fibration'. 

There is a closed subset\/ $\De\subset B$ called the `discriminant', a trivalent graph in $B$. Outside a small open neighbourhood of\/ $\pi^{-1}(\De)$ in $Y,$ the map $\pi:Y\ra B$ is smooth, with fibres $Y_b=\pi^{-1}(b)$ for $b\in B\sm\De$ special Lagrangian tori\/ $T^3$ in $Y,$ with approximately flat metrics $h\vert_{Y_b},$ of diameters $\mathop{\rm diam}(Y_b)\ll\mathop{\rm diam}(Y)$. Close to $\pi^{-1}(\De),$ the map $\pi$ need not be smooth, and the $Y_b$ may be singular.
\label{nc7conj}
\end{conj}

The analogue for $K3$ surfaces was proved by Gross and Wilson \cite{GrWi}.

We are particularly interested in the claim that the nonsingular SYZ fibres $Y_b$ are tori $T^3$ with metrics $h\vert_{Y_b}$ close to flat. Suppose that in the situation of Conjecture \ref{nc7conj}, we can find an antiholomorphic involution $\tau:Y\ra Y$ which preserves the SYZ fibration $\pi:Y\ra B$, in the sense that there is an involution $\si:B\ra B$ with $\pi\ci\tau=\si\ci\pi$. Suppose too that $\si$ has only finitely many fixed points $b_1,\ldots,b_k$ in $B$, all of which lie in $B\sm\De$, and the fixed locus $L'$ of $\tau$ is the disjoint union of fibres~$Y_{b_1}\amalg \cdots\amalg Y_{b_k}$.   

If all this holds, and Conjecture \ref{nc7conj} is true, and $h$ is sufficiently close to the `large complex structure limit', then $L'$ would be $k$ copies of $T^3$ with metric $h\vert_{L'}$ close to flat, so we could apply Theorem \ref{nc6thm} as in \S\ref{nc71}(A),(E). But this is not rigorous, as Conjecture \ref{nc7conj} is not yet proved.

Here is an example in which we expect this picture to hold:

\begin{ex} Let $Y_\ep$ be a hypersurface in $\CP^1\t\CP^1\t\CP^1\t\CP^1$ of the form
\begin{align*}
Y_\ep=\bigl\{&([w_0,w_1],[x_0,x_1],[y_0,y_1],[z_0,z_1])\in(\CP^1)^4:\\
&(w_0x_0y_0z_0)^2+(w_1x_1y_1z_1)^2+\ep P(w_0,w_1,x_0,x_1,y_0,y_1,z_0,z_1)=0\bigr\},
\end{align*}
where $\ep>0$ is small, and $P(w_0,w_1,x_0,x_1,y_0,y_1,z_0,z_1)$ is a complex polynomial which is polyhomogeneous of degree $(2,2,2,2)$ in $(w_0,w_1),\ab (x_0,x_1),\ab(y_0,y_1),\ab(z_0,z_1)$, and satisfies the complex conjugate equation
\e
P(\bar w_1,\bar w_0,\bar x_1,\bar x_0,\bar y_1,\bar y_0,\bar z_1,\bar z_0)=\overline{P(w_0,w_1,x_0,x_1,y_0,y_1,z_0,z_1)}
\label{nc7eq2}
\e
for all $w_0,\ldots,z_1$. If $P$ is generic under these conditions then $Y_\ep$ is nonsingular for all small $\ep>0$, and is a Calabi--Yau 3-fold, which approaches the `large complex structure limit' as $\ep\ra 0$. Define an antiholomorphic involution $\tau_\ep:Y_\ep\ra Y_\ep$ by
\begin{equation*}
\tau_\ep:([w_0,w_1],[x_0,x_1],[y_0,y_1],[z_0,z_1])\longmapsto([\bar w_1,\bar w_0],[\bar x_1,\bar x_0],[\bar y_1,\bar y_0],[\bar z_1,\bar z_0]).
\end{equation*}
This is well defined by \eq{nc7eq2}. The fixed locus $L'_\ep$ of $\tau_\ep$ in $Y_\ep$ may be written
\begin{align*}
L'_\ep=\bigl\{&([1,e^{i\th_1}],[1,e^{i\th_2}],[1,e^{i\th_3}],[1,e^{i\th_4}]):\th_1,\ldots,\th_4\in \R/2\pi\Z,\\
&1+e^{2i(\th_1+\th_2+\th_3+\th_4)}+\ep P(1,e^{i\th_1},1,e^{i\th_2},1,e^{i\th_3},1,e^{i\th_4})=0\bigr\}.
\end{align*}
When $\ep$ is small, the second line approximates $1+e^{2i(\th_1+\th_2+\th_3+\th_4)}=0$, so that $e^{i(\th_1+\th_2+\th_3+\th_4)}=i=e^{i\pi/2}$ or $e^{i(\th_1+\th_2+\th_3+\th_4)}=-i=e^{-i\pi/2}$. Hence we have
\begin{align*}
L'_\ep \approx & \bigl\{([1,e^{i\th_1}],\ldots,[1,e^{i\th_4}]):\th_1,\ldots,\th_4\in \R/2\pi\Z,\; \th_1+\cdots+\th_4=\ts\frac{\pi}{2}\bigr\}\\
\
&\amalg\,\bigl\{([1,e^{i\th_1}],\ldots,[1,e^{i\th_4}]):\th_1,\ldots,\th_4\in \R/2\pi\Z,\; \th_1+\cdots+\th_4=-\ts\frac{\pi}{2}\bigr\},
\end{align*}
which is the disjoint union of two copies of~$T^3$. 

The authors expect that $Y_\ep$ has an SYZ fibration $\pi_\ep:Y_\ep\ra\cS^3$ for small $\ep>0$, as in Conjecture \ref{nc7conj}, and that $\tau_\ep$ is compatible with the involution $\si:\cS^3\ra\cS^3$, $\si:(x_1,x_2,x_3,x_4)\mapsto (x_1,-x_2,-x_3,-x_4)$ fixing two points $(\pm 1,0,0,0)$ in $\cS^3\sm\De$, and the two $T^3$'s are the fibres of $\pi_\ep$ over~$(\pm 1,0,0,0)$.

Set $M=\cS^1\t Y_\ep$ for small $\ep>0$, with torsion-free $\Gt$-structure $(\vp,g)$ as in \S\ref{nc25}. Define an involution $\io:M\ra M$ by $\io(x,y)=(-x,\tau_\ep(y))$. Then $\io$ preserves $(\vp,g)$, and has fixed locus $\{0,\ha\}\t L'_\ep$, which is four copies of $T^3$. We can show that $Y$ has Betti numbers $b^0(Y)=1$, $b^1(Y)=0$, $b^2(Y)=4$, $b^3(Y)=138$, so that $M$ has Betti numbers $b^0(M)=1$, $b^1(M)=1$, $b^2(M)=4$, and $b^3(M)=142$. We can compute the action of $\io$ on $H^*(M;\R)$, and so show that $M/\an{\io}$ has Betti numbers $b^0(M/\an{\io})=1$, $b^1(M/\an{\io})=0$, $b^2(M/\an{\io})=0$, and $b^3(M/\an{\io})=73$. Thus Proposition \ref{nc6prop1} shows that the resolution $N$ of $M/\an{\io}$ has Betti numbers $b^0(N)=1$, $b^1(N)=0$, $b^2(N)=2$, and $b^3(N)=79$. Kovalev and Lee \cite[Table 1]{KoLe} construct a compact 7-manifold with holonomy $\Gt$ and $b^2=2$, $b^3=79$.
\label{nc7ex5}	
\end{ex}

If Conjecture \ref{nc7conj} could be proved, it is likely that this construction would yield many examples of compact 7-manifolds with holonomy~$\Gt$.

\section{Directions for future study} 
\label{nc8}

There are now several natural directions of study to pursue following our main theorem. Some of these are:
\begin{enumerate}[(i)]
\item One could try to further generalize the theorem to the case of singularities of the form $\R^3 \times (\C^2 / \Ga)$ for other discrete subgroups $\Ga$ of $\mathrm{Sp}(1)$, not just $\Z_2$. It is possible that most of the analysis can be adapted to this setting. There are ALE hyperK\"ahler $4$-manifolds that are asymptotic at infinity to $\C^2 / \Ga$ for many discrete groups $\Ga$. See \cite{Joyc3, Kron} for details. The difficulty in applying such a generalization of our main theorem would be in finding initial $\Gt$-orbifolds with the correct types of singularities.
\item It is likely that a version of our main theorem can be proved in the setting of $\mathrm{Spin}(7)$-manifolds. This is currently being investigated by a student of the second author.
\item The notion of a {\it $\Gt$-instanton} was introduced by Donaldson--Thomas \cite{DoTh} and elaborated on in \cite{DoSe}. These are special connections on vector bundles over $\Gt$-manifolds, which are analogues of ASD connections on $4$-manifolds. A natural problem is to prove existence of such connections on the $\Gt$-manifolds produced by our theorem, using a gluing construction. Precisely this problem was considered in the context of the first author's original construction of compact $\Gt$-manifolds by Walpuski~\cite{Walp}. 
\item New topological and analytic invariants of $\Gt$-structures have been introduced by Crowley--Nordstr\"om \cite{CrNo} and Crowley--Goette--Nordstr\"om \cite{CGN}. In particular these invariants can sometimes be used to distinguish distinct $\Gt$-structures on the same underlying manifold. It would be interesting to compute these invariants for the $\Gt$-manifolds produced by our theorem. 
\item The setup for the analysis used to prove our theorem seems to be related to some recent work of Donaldson \cite{Dona}. This relation deserves closer study. It is possible that such a further study could improve or replace part of the work we did in \S\ref{nc3},\ref{nc4},\ref{nc5} of the present paper.
\item When the closed, coclosed $1$-form $\lambda$ on $L$ is \emph{not} nonvanishing, our construction fails. Generically, $\lambda$ should have finitely many isolated zeroes. We cannot resolve the space at these points. In this case there is some evidence that we are instead \emph{almost} producing objects that have isolated conical singularities. This is because one can show that \emph{topologically}, near an isolated zero of $\lambda$ our resolved manifold looks like a cone over $\CP^3$. There is a $\Gt$-cone over $\CP^3$, due to Bryant and Salamon~\cite{BrSa}. 

If we could further modify the closed forms $\vp^N_t$, $\psi^N_t$ near the zeroes of $\lambda$ to interpolate toward the Bryant--Salamon cone, then one could hope to produce the first known compact examples of \emph{conically singular (CS)\/ $\Gt$-manifolds}. This is currently being investigated by the second author and Jason Lotay. They have proved an analogue of Theorem \ref{nc2thm3} in the category of CS manifolds, using weighted Sobolev spaces. However, constructing the appropriate closed 3- and 4-forms is likely to be quite difficult. 

The analysis of $\Gt$-conifolds has been studied by the second author and Lotay \cite{Kari2,KaLo}. In particular, if they succeed in producing CS $\Gt$-manifolds by generalizing the main theorem of the present paper, then these CS $\Gt$-manifolds could then be \emph{further resolved\/} to produce compact smooth $\Gt$-manifolds using the main result of \cite{Kari2}.
\end{enumerate}

\appendix

\section{Appendix: Two calculations needed in \S\ref{nc34}} 
\label{ncA}

In this appendix we present the details of two calculations that were needed in \S\ref{nc34}, for the proofs of Lemma \ref{nc3lem2} and \ref{nc3lem3}. They are actually general results about the linear algebra of a $4$-dimensional real vector space $V$ equipped with the standard hyperK\"ahler package. In the lemmas this space $V$ is the fibre of $\nu$ at some $x \in L$. Let $V$ be an oriented $4$-dimensional real vector space, equipped with metric $g$, a volume form $\vol$, three orthogonal complex structures $J_1,J_2,J_3$ satisfying the quaternionic relations, and a triple $\om_1,\om_2,\om_3$ of $2$-forms satisfying $g(u,v) = \om_k (J_k u, v)$. We will also need the relation \eq{nc2eq12} that
\e 
\ga \w \om_k = \ast( J_k \ga) \quad \text{ for any $1$-form $\ga$}.
\label{ncAeq1}
\e
Let $f^1,f^2,f^3,f^4$ be an oriented orthonormal basis for $V^*$ and let $e_1,e_2,e_3$ denote the standard basis of $\R^3$.

\begin{prop} Consider the linear maps
\begin{equation*}
A : V^* \otimes \Lambda^3 V^* \longra \Lambda^4 V^*, \qquad A (\ga\ot  \eta) = \ga \w \eta,
\end{equation*}
and
\begin{align*}
B &: S^2 V^* \otimes \R^3 \longra V^* \otimes \Lambda^3 V^*, \\
B &: h^k_{pq} f^p f^q \otimes e_k \longmapsto h^k_{pq} [ f^p \otimes (f^q \w \om_k) + f^q \otimes (f^p \w \om_k) ].
\end{align*}
The map $B$ surjects onto the kernel of\/ $A$.
\label{ncAprop1}
\end{prop}

\begin{proof} Note that $V^* \cong \Lambda^3 V^*$ via the isomorphism $\ga \mapsto \ast \gamma$. Thus we have $A : V^* \otimes \Lambda^3 V^* \to \Lambda^4 V^*$ with $A ( s_{pm} f^p \otimes \ast f^m ) = s_{pm} f^p \w \ast f^m = s_{pm} g^{pm} \vol$, so $s_{pm} f^p \otimes \ast f^m \in \Ker A$ if and only if $s_{pm} g^{pm}= 0$. Using \eq{ncAeq1}, we compute
\e
\begin{split} 
B( h^k_{pq}& f^p f^q \otimes e_k ) = h^k_{pq} [ f^p \otimes (f^q \w \om_k) + f^q \otimes (f^p \w \om_k) ] 
\\ 
& = h^k_{pq} [f^p \otimes \ast(J_k f^q) + f^q \otimes \ast (J_k f^p) ] = 2 h^k_{pq} f^p \otimes \ast(J_k f^q).
\end{split}
\label{ncAeq2}
\e
Thus $A [ B( h^k_{pq} f^p f^q \otimes e_k ) ] = 2 h^k_{pq} g( f^p, J_k f^q) \vol = 2 h^k_{pq} \om_k (f^p, f^q) \vol = 0$, since $\om_k$ is skew and $h^k_{pq}$ is symmetric in $p,q$. So the image of $B$ always lies in the kernel of $A$.

In terms of the basis $f^1, \ldots, f^4$, the kernel of $A$ is spanned by two types of elements of $V^* \otimes \Lambda^3 V^*$: those of the form $f^{p_0} \otimes \ast f^{q_0}$ for $p_0 \neq q_0$ and those of the form $f^{p_0} \otimes \ast f^{p_0} - f^{q_0} \otimes \ast f^{q_0}$ for $p_0 \neq q_0$. We can see from \eq{nc2eq11} that there exists a unique $k_0 \in \{1, 2, 3 \}$ such that $J_{k_0} f^{q_0} = \pm f^{p_0}$. In the first case, set $h^{k_0}_{p_0 p_0} = \pm \frac{1}{2}$, and all other $h^k_{pq} = 0$. It follows from \eq{ncAeq2} that for this choice, $B(h^k_{pq} f^p f^q \otimes e_k) = f^{p_0} \otimes \ast f^{q_0}$. In the second case, set $h^{k_0}_{p_0 q_0} = h^{k_0}_{q_0 p_0} = \pm \frac{1}{2}$, and all other $h^k_{pq} = 0$. Noting that $J_{k_0} f^{p_0} = \mp f^{q_0}$, it is each to check that for this choice, equation \eq{ncAeq2} gives $B(h^k_{pq} f^p f^q \otimes e_k) = f^{p_0} \otimes \ast f^{p_0} - f^{q_0} \otimes \ast f^{q_0}$. Thus $B$ indeed surjects onto the kernel of $A$.
\end{proof}

\begin{rem} Proposition \ref{ncAprop1} is used in the proof of Lemma \ref{nc3lem2}, where $B (\bar \al) = -\bar \al \cdot \bar{\dot\vp}{}^2_{2,1}$ and $\bar{\dot\vp}{}^2_{2,1} = - (\om_1 \otimes e_1 + \om_2 \otimes e_2 + \om_3 \otimes e_3)$.
\label{ncArem1}
\end{rem}

\begin{prop} 
\label{ncAprop2}
Consider the linear map
\begin{align*}
C &:  V \otimes V^* \otimes \R^3 \longra V^* \otimes V^* \otimes \Lambda^2 \R^3, \\
C &: s^{pk}_q f_p \otimes f^q \otimes e_k \longmapsto \sum_{j=1}^3 s^{pk}_q (f_p \cdot \om_j) \otimes f^q \otimes (e_k \w e_j).
\end{align*}
The map $C$ is an isomorphism.
\end{prop}
\begin{proof}
The middle factor of $V^*$ is just coming along for the ride, so we need to check that the map $V \otimes \R^3 \to V^* \otimes \Lambda^2 \R^3$ which on decomposable elements is given by $f_p \otimes e_k \mapsto \sum_{j=1}^3 (f_p \cdot \om_j) \otimes (e_j \w e_k)$ is an isomorphism. The dimensions agree, so we need only show the kernel is trivial. Let $\sum_{k=1}^3 s^p_k f_p \otimes e_k \in V \otimes \R^3$ be in the kernel. Then $\sum_{j,k=1}^3 s^p_k (f_p \cdot \om_j) \otimes (e_j \w e_k) = 0$. It is easy to check that in an orthonormal frame we have $f_p \cdot \om_j = - J_j f^p = - (J_j)^p_q f^q$, so we have
\begin{equation*}
0  \!=\! \sum_{p,q=1}^4 \sum_{j,k=1}^3 s^p_k (J_j)^p_q f^q \!\otimes\! (e_j \!\w\! e_k)  \!= \!\sum_{p,q=1}^4 \sum_{j < k} (s^p_k (J_j)^p_q \!-\! s^p_j (J_k)^p_q) f^q \!\otimes\! (e_j \!\w\! e_k).
\end{equation*}
We deduce that for any fixed pair $j \neq k$, we have
\begin{equation*}
0 = \sum_{p=1}^4 (s^p_k (J_j)^p_q - s^p_j (J_k)^p_q).
\end{equation*}
Let $i$ be the remaining index in $\{1, 2, 3\}$ and without loss of generality order them so that $i,j,k$ is a cyclic permutation of $1,2,3$. Multiplying the above equation by $(J_i)^q_m$ and summing over $q$, we find
\e 
\begin{aligned}
0 & = \sum_{p=1}^4 (s^p_k (J_i)^q_m (J_j)^p_q - s^p_j (J_i)^q_m (J_k)^p_q) \\
& = \sum_{p=1}^4 (s^p_k (J_i J_j)^p_m - s^p_j (J_i J_k)^p_m) = \sum_{p=1}^4 (s^p_k (J_k)^p_m + s^p_j (J_j)^p_m).
\end{aligned}
\label{ncAeq3}
\e
For $i \in \{1, 2, 3\}$ and $m \in \{1, 2, 3, 4\}$, define $t_{im} = \sum_{p=1}^4 s^p_i (J_i)^p_m$. In \eq{ncAeq3} we deduced that $t_{jm} + t_{km} = 0$ if $j \neq k$. Thus $t_{1m} = - t_{2m} = t_{3m} = - t_{1m}$, so $t_{1m} = 0$ and similarly $t_{im} = 0$ for all $i, m$. Now multiply $t_{im} = 0$ by $(J_i)^m_q$ and sum over $m$. From $(J_i)^2 = - I$, we find
\begin{equation*}
0 = (J_i)^m_q t_{im} = \sum_{p=1}^4 s^p_i (J_i)^m_q (J_i)^p_m = -\sum_{p=1}^4 s^p_i \de^p_q = -s^q_i.
\end{equation*}
Thus the kernel is indeed trivial and $C$ is an isomorphism.
\end{proof}

\begin{rem} Proposition \ref{ncAprop2} is used in the proof of Lemma \ref{nc3lem3}, where $C (\Ga) = - \Ga \cdot \bar{\dot\vp}{}^2_{2,1}$ and $\bar{\dot\vp}{}^2_{2,1} = - (\om_1 \otimes e_1 + \om_2 \otimes e_2 + \om_3 \otimes e_3)$.
\label{ncArem2}
\end{rem}

\end{document}